\documentclass[10pt]{article}
\usepackage[left=1in,top=1in,right=1in,bottom=1in,letterpaper]{geometry}

\usepackage{amssymb,amsmath,enumerate}
\usepackage{latexsym,amsfonts,amscd,amsxtra,amstext}

\usepackage{hyperref}
\hypersetup{colorlinks=true, linkcolor=blue, citecolor=blue, urlcolor=blue, pdftitle={Removing Data Heterogeneity Influence enhances Network Topology Dependence of Decentralized SGD}, pdfauthor={Kun Yuan, Sulaiman A. Alghunaim, Xinmeng Huang}}

\usepackage[pdftex]{graphicx}
\usepackage{amsmath}
\usepackage{amssymb}
\usepackage{empheq}
\usepackage{dsfont}
\usepackage{booktabs}
\usepackage{color}
\usepackage{xcolor}
\usepackage[ruled,vlined]{algorithm2e}
\usepackage{subcaption}

\SetKwInput{KwInput}{Input}                
\SetKwInput{KwOutput}{Output}              
\SetKwInput{KwRequire}{Require}              

\newcommand{\define}{\;\stackrel{\Delta}{=}\;}

\newcommand{\ba}{\left[ \begin{array}}
	\newcommand{\ea}{\\ \end{array} \right]}

\def\g{{\boldsymbol{g}}}

\def\s{{\boldsymbol{s}}}

\def\x{{\boldsymbol{x}}}
\def\y{{\boldsymbol{y}}}
\def\z{{\boldsymbol{z}}}

\newcommand{\vg}{{\mathbf{g}}}

\newcommand{\vs}{{\mathbf{s}}}

\newcommand{\vx}{{\mathbf{x}}}
\newcommand{\vy}{{\mathbf{y}}}
\newcommand{\vz}{{\mathbf{z}}}

\newcommand{\cA}{{\mathcal{A}}}

\newcommand{\cF}{{\mathcal{F}}}

\newcommand{\cN}{{\mathcal{N}}}
\newcommand{\cO}{{\mathcal{O}}}

\newcommand{\cW}{{\mathcal{W}}}

\newcommand{\bxi}{{\boldsymbol{\xi}}}

\newcommand{\EE}{\mathbb{E}}
\newcommand{\RR}{\mathbb{R}}

\def\tran{^{\mathsf{T}}}
\def\one{\mathds{1}}

\newcommand{\diag}{{\mathrm{diag}}} 
\newcommand{\grad}{{\nabla}}    
\newcommand{\tr}{{\mathrm{tr}}} 

\newcommand{\vvvert}{{\vert\kern-0.25ex\vert\kern-0.25ex\vert}}

\allowdisplaybreaks

\newtheorem{theorem}{Theorem}
\newtheorem{assumption}{Assumption}
\newtheorem{remark}{Remark}
\newtheorem{lemma}{Lemma}
\newtheorem{proposition}[theorem]{Proposition} 
\newtheorem{corollary}{Corollary}
\newcommand{\BlackBox}{\rule{1.5ex}{1.5ex}} 
\newenvironment{proof}{\par\noindent{\bf Proof\ }}{\hfill\BlackBox\\[2mm]}

\begin{document}
	
	\title{Removing Data Heterogeneity Influence Enhances Network Topology Dependence of Decentralized SGD}
	\date{Version 2 \hspace{5mm} \today}
	\author{Kun Yuan\footnote{K. Yuan is with DAMO Academy,
			Alibaba (US) Group, Bellevue, WA 98004, USA. \textup{kun.yuan@alibaba-inc.com}} \and Sulaiman A.\ Alghunaim\footnote{S. Alghunaim is with Dept. Electrical Engr., Kuwait University, Safat 13060, Kuwait. \textup{sulaiman.alghunaim@ku.edu.kw}}
		\and Xinmeng Huang\footnote{X. Huang is with  Grad. Group in AMCS, Univ. of Pennsylvania, Philadelphia, PA 19104, USA. \textup{xinmengh@sas.upenn.edu}}
	} \maketitle

	\begin{abstract}
	{ We consider the decentralized stochastic optimization problems,  where a network of $n$ nodes, each owning a local cost function, cooperate to find a minimizer of the globally-averaged cost. A widely studied decentralized algorithm for this problem is decentralized SGD (D-SGD), in which each node 
	averages only with its neighbors.  D-SGD is efficient in single-iteration communication, but it is very sensitive to the network topology. For smooth objective functions, the transient stage (which measures the number of iterations the algorithm has to experience before achieving the linear speedup stage) of D-SGD is on the order of ${\Omega}(n/(1-\beta)^2)$ and  $\Omega(n^3/(1-\beta)^4)$ for strongly and generally convex cost functions, respectively, where $1-\beta \in (0,1)$ is a topology-dependent quantity that approaches $0$ for a large and sparse network. Hence,  D-SGD suffers from slow convergence for large and sparse networks.

			In this work, we study the non-asymptotic convergence property of the  D$^2$/Exact-diffusion algorithm. By eliminating the influence of data heterogeneity between nodes, D$^2$/Exact-diffusion is shown to have an enhanced transient stage that is on the order of $\tilde{\Omega}(n/(1-\beta))$ and  $\Omega(n^3/(1-\beta)^2)$ for strongly and generally convex cost functions (where $\tilde{\Omega}(\cdot)$ hides all logarithm factors), respectively. Moreover, when D$^2$/Exact-diffusion is implemented with gradient accumulation and multi-round gossip communications, its transient stage can be further improved to $\tilde{\Omega}(1/(1-\beta)^{\frac{1}{2}})$ and $\tilde{\Omega}(n/(1-\beta))$ for strongly and generally convex cost functions, respectively. These established results for D$^2$/Exact-Diffusion have the best ({\em i.e.}, weakest) dependence on network topology to our knowledge compared to existing decentralized algorithms. We also conduct numerical simulations to validate our theories. 
		}
	\end{abstract}
	
	\section{Introduction}
	Large-scale optimization and learning have become an essential tool in many practical applications. State-of-the-art performance has been reported in various fields such as signal processing, control, reinforcement learning, and deep learning. The amount of data needed to achieve satisfactory results in these tasks is typically very large. Moreover,  increasing the size of training data can significantly improve the ultimate performance in these tasks. For this reason, the scale of optimization and learning nowadays calls for efficient distributed 
	solutions across multiple computing nodes (\textit{e.g}, machines).

	This work considers a network of $n$ collaborative nodes connected  through a given topology. Each node owns a private and local cost function $f_i:\mathbb{R}^d \rightarrow \real \mathbb{R}$ and the goal of the network is to find a solution, denoted by $x^\star$, of the {\em stochastic optimization} problem
	\begin{align}\label{eq:general-prob}
	&\ \min_{x \in \RR^d}\quad f(x) = \frac{1}{n}\sum_{i=1}^n f_i(x), \quad \mbox{where} \quad f_i(x) = \mathbb{E}_{{\xi}_i \sim D_i}F(x; {\xi}_i).
	\end{align}
	In this problem, $\xi_i$ is a random variable that represents the local data available at node $i$, and it follows a local distribution $D_i$. 
	Each node $i$ has the access to the stochastic gradient $\nabla F(x_i; \xi_i)$ of its private cost, but has to communicate to exchange information with other nodes. In practice, the local data distribution $D_i$ within each node is generally different, and hence, $f_i(x) \neq f_j(x)$ holds for any node $i$ and $j$. For {\em convex costs}, the data heterogeneity across the network can be characterized by $b^2 = \frac{1}{n}\sum_{i=1}^n \|\nabla f_i(x^\star)\|^2$. If all local data samples follow the same distribution $D$, we have $f_i(x) = f_j(x)$ for any $i, j$ and hence $b^2 = 0$.


	One of the leading algorithms to solve problem \eqref{eq:general-prob} is parallel SGD (P-SGD) \cite{zinkevich2010parallelized}.  In P-SGD,  each node computes its local stochastic gradient and then synchronizes across the entire network to find the  globally averaged stochastic gradient used to update the model parameters (\textit{i.e.}, solution estimates).  The global synchronization step needed to compute the  globally averaged stochastic gradient can be implemented via  Parameter Server \cite{smola2010architecture,li2014scaling} or Ring-Allreduce \cite{ring-allreduce}, which suffers from either significant bandwidth cost or high latency, see \cite[Table I]{ying2021bluefog} for detailed discussion. Decentralized SGD (D-SGD) \cite{lopes2008diffusion,nedic2009distributed,chen2012diffusion} is a promising alternative to P-SGD  due to its ability to reduce the communication overhead \cite{lian2017can,assran2019stochastic,lian2018asynchronous,yuan2021decentlam}. D-SGD is based on \textit{local averaging} (also known as {\em gossip averaging}) in which each node computes the locally averaged model parameters with their direct neighbors as opposed to the global average. Moreover, no global synchronization step is required in D-SGD. On a delicately-designed sparse topology such as one-peer exponential graph \cite{assran2019stochastic,ying2021exponential}, each node only needs to communicate with \textit{one} neighbor per iteration in D-SGD, resulting in a much cheaper communication cost compared to P-SGD, see the discussion in \cite{assran2019stochastic,bluefog2021_4616052,chen2021accelerating} and \cite[Table I]{ying2021bluefog}.

	Apart from its efficient single-iteration communication, D-SGD can asymptotically achieve the same linear speedup  as P-SGD  \cite{lian2017can,lian2018asynchronous,assran2019stochastic,koloskova2020unified}. {{\em Linear speedup} refers to a property in distributed algorithms where the number of iterations needed to reach an $\epsilon$-accurate solution reduces linearly with the number of nodes.} The \textit{transient stage} \cite{pu2019sharp}, which refers to those iterations before an algorithm reaches its linear speedup stage, is an important metric to measure the convergence performance of decentralized algorithms. \textbf{The convergence rate and  transient stage of D-SGD is very sensitive to the network topology.} For example, the transient stage of D-SGD for generally convex or non-convex objective functions is on the order of $O(n^3/(1-\beta)^4)$ \cite{koloskova2020unified}, where $1-\beta$ measures the network topology connectivity. For a large and sparse network in which $1-\beta$  approaches to $0$, D-SGD will suffer from an extremely long transient stage and it may not be able to reach the linear speedup stage given limited training time and computing resource budget. 	For this reason, D-SGD may end up with a low-quality solution that is  significantly worse than that obtained by P-SGD. As a result, 
	 improving the network topology dependence (\textit{i.e.}, making the convergence rate less sensitive to network topology) in D-SGD is crucial to enhance its convergence rate and solution accuracy. 
	
	The main factor in D-SGD contributing to its strong dependence on the network topology is the data heterogeneity across each node as shown in \cite{koloskova2020unified,yuan2020influence,pu2019sharp}. This naturally motivates us to examine whether removing the influence of data heterogeneity ({\em i.e.}, $b^2$) can improve the dependence on the topology ({\em i.e.}, $1-\beta$) of D-SGD. 

	\subsection{Main Results}
	This work revisits the D$^2$ algorithm \cite{tang2018d}, which is  also known as Exact-Diffusion \cite{yuan2017exact1,yuan2018exact2,yuan2020influence} (or NIDS \cite{li2019decentralized}). D$^2$/Exact-Diffusion is a decentralized optimization algorithm that can remove the influence of data heterogeneity \cite{yuan2018exact2,li2019decentralized}, but it is  unclear whether D$^2$/Exact-Diffusion has an improved network topology dependence compared to D-SGD in the transient stage.  In this work, we establish non-asymptotic convergence rates for D$^2$/Exact-Diffusion under both the generally-convex and strongly-convex  settings. The established bounds show that D$^2$/Exact-Diffusion has the best known network topology dependence compared with existing results. In particular, 
	this paper establishes that D$^2$/Exact-Diffusion at iteration $T$ converges with rate 
	\begin{align}
	\frac{1}{T+1}\sum_{k=0}^T\big(\mathbb{E}f(\bar{x}^{(k)}) - f(x^\star)\big) &= O\left(\frac{\sigma}{\sqrt{n T}} + \frac{\sigma^{2/3}}{(1-\beta)^{1/3} T^{2/3}} + \frac{1}{(1-\beta)T}\right) \hspace{2.6cm} \mbox{(G-C)} \label{D2-rate-GC} \\
	\frac{1}{H_T}\sum_{k=0}^T h_k \big(\mathbb{E}f(\bar{x}^{(k)}) - f(x^\star) \big) &= \tilde{O}\left(\frac{\sigma^2}{n T} + \frac{\sigma^2}{(1-\beta) T^2} + 
	\frac{1}{1-\beta} \exp\{-(1-\beta) T \}  \right) \hspace{1.6cm} \mbox{(S-C)} \label{D2-rate-SC} 
	\end{align}
	where $\bar{x}^{(k)} = \frac{1}{n}\sum_{i=1}^n x_i^{(k)}$, $x_i^{(k)}$ is the estimate of agent $i$ at iteration $k$,  and $\sigma^2$ denotes the variance of the stochastic gradient $\nabla F(x; \bxi_i)$. The weights $h_k \ge 0$ are given in Lemma \ref{lm-sc-ergodic-consenus} and $H_T = \sum_{k=0}^T h_k$. The notation $\tilde{O}(\cdot)$ hides all logarithm factors.  Moreover, (G-C) stands for the generally convex scenario while (S-C) stands for strongly convex one. Below, we compare this result with D-SGD. 
	
	\vspace{1mm}
	\noindent \textbf{Transient stage and linear speedup}. When $T$ is sufficiently large, the first term $\sigma^2/(nT)$ in \eqref{D2-rate-SC} (or $\sigma/\sqrt{nT}$ in \eqref{D2-rate-GC}) will dominate the rate.  In this scenario, D$^2$/Exact-Diffusion requires $T = \Omega(1/n\epsilon)$ in \eqref{D2-rate-SC} (or $T = \Omega(1/n\epsilon^2)$ in \eqref{D2-rate-GC}) iterations to reach a desired $\epsilon$-accurate solution for strongly convex (or generally convex) problems, which is inversely proportional to the network size $n$. Therefore,  we say an algorithm reaches the linear-speedup stage if for some $T$, the term involving $nT$ is dominating the rate. Rates \eqref{D2-rate-GC} and \eqref{D2-rate-SC} corroborate that D$^2$/Exact-Diffusion, similar to P-SGD,  achieves linear speedup for sufficiently large $T$. We note that D$^2$/Exact-Diffusion can also achive linear speedup   in the non-convex setting  \cite{tang2018d,alghunaim2021unified}. 	The {\em transient stage}  is the amount of iterations needed to reach the linear-speedup stage, which is an important metric to measure the scalability of distributed algorithms \cite{pu2019sharp}. For example, let us consider D$^2$/Exact-Diffusion in the strongly convex scenario \eqref{D2-rate-SC}: to reach linear speedup, $T$ has to satisfy  $(1-\beta)T^2\le nT$  (\textit{i.e.,} $T \ge n/(1-\beta)$). Therefore, the transient stage in D$^2$/Exact-Diffusion for strongly-convex scenario requires  $\tilde{\Omega}(n/(1-\beta))$ iterations. 
	
	\vspace{1mm}
	\noindent \textbf{Comparison with D-SGD.} Table \ref{table-dsgd_vs_D2} lists the convergence rates for D-SGD and D$^2$/Exact-Diffusion for both generally and strongly convex scenarios.  Compared to D-SGD, it is observed in \eqref{D2-rate-GC} and \eqref{D2-rate-SC} that D$^2$/Exact-Diffusion has eliminated the data heterogeneity $b^2$ term. Note that the term related to data heterogeneity $b^2$ has the strongest topology dependence on $1-\beta$ for D-SGD. In Table \ref{table-transient-stage-local} we list the transient stage of D$^2$/Exact-Diffusion and other existing algorithms. It is observed that D$^2$/Exact-Diffusion has an improved transient stage in terms of $1-\beta$ compared to D-SGD by removing the influence of data heterogeneity. Gradient tracking methods \cite{xu2015augmented,di2016next,nedic2017achieving,qu2018harnessing,xin2020improved} can also remove the data heterogeneity, but their transient stage established in existing prior works still suffers from worse network topology dependence than D$^2$/Exact-Diffusion and even D-SGD. In other words, D$^2$/Exact-Diffusion enjoys the state-of-the-art topology dependence in the generally and strongly convex scenarios.
	
	\begin{table}[t]
		\centering
		\begin{tabular}{ccc}
			\toprule
			\textbf{Scenario} & \textbf{D$^2$/Exact-Diffusion}                         & \textbf{D-SGD}                    \\ \midrule
			Generally-convex  & \eqref{D2-rate-GC}     & \eqref{D2-rate-GC}~+~$O\left(\frac{b^{2/3}}{(1-\beta)^{2/3} T^{2/3}}\right)$ \vspace{1mm}\\
			Strongly-convex  & \eqref{D2-rate-SC}    & \eqref{D2-rate-SC}~+~$O\left(\frac{b^2}{(1-\beta)^2 T^2}\right)$       \\
			\bottomrule
		\end{tabular}
		\caption{\small Convergence rate comparison with D-SGD  from \cite{koloskova2020unified}.}
		\label{table-dsgd_vs_D2}
	\end{table}
	
	\vspace{1mm}
	{\noindent \textbf{Further improvement with multi-round gossip communication.} Another orthogonal idea to improve network topology dependence is to run multiple gossip steps per D$^2$/Exact-Diffusion update. By utilizing multiple gossip communications and gradient accumulation, the transient stage of D$^2$/Exact-Diffusion can be significantly improved in both generally- and strongly-convex scenarios, see the last row in Table \ref{table-transient-stage-local}. 
	} 
	
	\subsection{Contributions}
	This work makes the following contributions:
	\begin{itemize}
		\item We revisit the D$^2$/Exact-Diffusion algorithm \cite{tang2018d,yuan2017exact1,yuan2018exact2,li2019decentralized,yuan2020influence} and establish its non-asymptotic convergence rate  under the {\em generally-convex} settings. By removing the influence of data heterogeneity, D$^2$/Exact-Diffusion is shown to improve the transient stage of D-SGD from $\Omega(n^3/(1-\beta)^4)$ to $\Omega(n^3/(1-\beta)^2)$, which is less sensitive to the network topology.
		
		\item We also establish the non-asymptotic convergence rate of  D$^2$/Exact-Diffusion under the {\em strongly-convex} settings. It is shown that D$^2$/Exact-Diffusion improves the transient stage of D-SGD from $\tilde{\Omega}(n/(1-\beta)^2)$ to $\tilde{\Omega}(n/(1-\beta))$. Furthermore, we prove that the transient stage of D-SGD  is lower bounded by $\tilde{\Omega}(n/(1-\beta))$ with homogeneous data ({\em i.e.}, $b^2=0$) in the strongly-convex scenario. This implies that the dependence of D-SGD on the network topology can only match with D$^2$/Exact-Diffusion if the data is homogeneous\footnote{The transient stage of D-SGD is lower bounded by $\Omega(n/(1-\beta)^2)$ in the heterogeneous case \cite{koloskova2020unified,pu2019sharp}}, which is typically an impractical assumption.

		
		\item We further improve network topology dependence by integrating {\em multiple gossip} and {\em gradient accumulation} to D$^2$/Exact-Diffusion. With these two useful techniques, the transient stage of D$^2$/Exact-Diffusion improves from $\Omega(n^3/(1-\beta)^2)$ to $\tilde{\Omega}(n/(1-\beta))$ in the generally-convex scenario, and $\tilde{\Omega}(n/(1-\beta))$ to $\tilde{\Omega}(1/(1-\beta)^{1/2})$ in the strongly-convex scenario. D$^2$/Exact-Diffusion with multiple gossip communications has a significantly better dependence on topology and network size than existing algorithms \cite{koloskova2020unified,pu2019sharp,pu2020distributed,huang2021improve,koloskova2021improved}.

	\end{itemize}

\begin{table}[]
\centering
\begin{tabular}{rcc}
\toprule
\textbf{Scenario}           & \textbf{Generally-convex}                                 & \textbf{Strongly-convex}                                                \\ \midrule
{D-SGD} \cite{koloskova2020unified,pu2019sharp}              & $\Omega\left(\frac{n^3}{(1-\beta)^4}\right)$     & $\tilde{\Omega}\left(\frac{n}{(1-\beta)^2}\right)$          \vspace{1mm}   \\
{Gradient Tracking} \cite{pu2020distributed}  & N.A.                                             & $\tilde{\Omega}\left(\frac{n}{(1-\beta)^3}\right)$           \vspace{1mm}  \\
\textbf{D$^2$/Exact-Diffusion (Ours)} & $\Omega\left(\frac{n^3}{(1-\beta)^2}\right)$     & $\tilde{\Omega}\left(\frac{n}{1-\beta}\right)$               \vspace{1mm}  \\
\textbf{D$^2$/ED-MG (Ours)}           & $\tilde{\Omega}\left(\frac{n}{(1-\beta)}\right)$ & $\tilde{\Omega}\left(\frac{1}{(1-\beta)^{\frac{1}{2}}}\right)$ \\ \bottomrule
\end{tabular}
\caption{\small Transient stage comparison between D-SGD, gradient tracking,  D$^2$/Exact-Diffusion, and D$^2$/Exact-Diffusion with multi-round gossip communication (D$^2$/ED-MG for short) in the strongly convex and generally convex settings. Note that $1-\beta \in (0,1)$. Notation $\tilde{\Omega}(\cdot)$ hides all logarithm factors. The smaller the transient stage is, the faster the algorithm will achieve linear speedup.}
\label{table-transient-stage-local}
\end{table}
	
	\subsection{Related Works}
	\noindent \textbf{Decentralized optimization.} Distributed optimization algorithms can be (at least) traced back to the work \cite{tsitsiklis1986distributed}. Decentralized gradient descent (DGD) \cite{nedic2009distributed,lopes2008diffusion,chen2012diffusion} and dual averaging \cite{duchi2011dual} are among the earliest decentralized optimization algorithms. DGD can have several forms depending on the combination/communication step order  such as consensus or diffusion \cite{sayed2014adaptation} (diffusion is also called adapt-then-combine DGD or just DGD in many recent works). Both DGD and dual averaging  suffer from a bias caused by data heterogeneity even under deterministic settings (\textit{i.e.,} no gradient noise exists) \cite{nedic2009distributed,chen2013distributed,yuan2016convergence} -- see more explanation in Sec.~\ref{sec-dsgd-bias}. Numerous algorithms have been proposed to address this issue, such as alternating direction method of multipliers (ADMM) methods \cite{wei2012distributed,shi2014linear}, explicit bias-correction methods (such as EXTRA~\cite{shi2015extra}, Exact-Diffusion~\cite{yuan2017exact1,yuan2018exact2}, NIDS~\cite{li2019decentralized}, and gradient tracking \cite{xu2015augmented, di2016next, nedic2017achieving, qu2018harnessing} -- see \cite{alghunaim2020decentralized}), and dual acceleration \cite{scaman2017optimal, scaman2018optimal, uribe2020dual}. These algorithms, in the deterministic setting, can converge to the exact solution without any bias. On the other hand, decentralized stochastic methods (in which the gradient is noisy) have also gained a lot of attentions recently. Since Decentralized SGD (D-SGD) has the same asymptotic linear speedup as P-SGD \cite{chen2012diffusion,sayed2014adaptation,lian2017can,koloskova2020unified,pu2019sharp} but with a more efficient single-iteration communication, it has been extensively studied in the context of large-scale machine learning (such as deep learning). 
	
	\vspace{1mm}
	\noindent \textbf{Data heterogeneity and network dependence.} It is well known that D-SGD is largely affected by the gradient noise, but it was unclear how the data heterogeneity influences the performance of D-SGD. The work \cite{yuan2020influence}  clarified that the error caused by the data heterogeneity can be greatly amplified when network topology is sparse, which can even be larger than the gradient noise error. The works \cite{pu2019sharp} and \cite{koloskova2020unified} also showed that the error term caused by data heterogeneity has the worst dependence on the network topology for D-SGD. It is thus conjectured that removing the influence of data heterogeneity can improve the topology dependence of D-SGD. The D$^2$/Exact-Diffusion algorithm and gradient tracking methods have been studied under stochastic settings in \cite{tang2018d} and \cite{pu2020distributed,xin2020improved,lu2019gnsd,zhang2019decentralized,xin2022fast}, respectively. However, the analysis in these works does not reveal whether the removal of data heterogeneity can improve the dependence on network topology. The work \cite{yuan2020influence} studied D$^2$/Exact-Diffusion in the \textit{steady-state} (asymptotic) regime and for {\em strongly-convex} costs. Under this setting, \cite{yuan2020influence} showed that D$^2$/Exact-Diffusion has an improved network topology dependence by removing data heterogeneity, but it is unclear whether the improved steady-state performance in \cite{yuan2020influence} carries over to D-SGD's {\em non-asymptotic performance} (convergence rate). This paper clarifies the improvement in the {\em non-asymptotic} convergence rate in D$^2$/Exact-Diffusion for {\em both} generally and strongly convex scenarios. These results demonstrate that removing influence of data heterogeneity improves the dependence on network topology for D-SGD. In addition, the established  dependence on network topology for D$^2$/Exact-Diffusion in the strongly-convex scenario coincides with  lower bound of D-SGD with {\em homogeneous} dataset. 
	
	\vspace{1mm}
	\noindent \textbf{Transient stage.} As to the transient stage of D-SGD, the work \cite{pu2019sharp} shows that it is $\Omega(n/(1-\beta)^2)$ for strongly-convex settings, and the result from \cite{koloskova2020unified} imply that it is $\Omega(n^3/(1-\beta)^4)$ for generally-convex settings. In comparison, we establish that D$^2$/Exact-Diffusion has an improved $\Omega(n/(1-\beta))$ and $\Omega(n^3/(1-\beta)^2)$ transient stage for strongly and generally convex scenarios, respectively. This work does not study the transient stage of D$^2$/Exact-Diffusion for the non-convex scenario as the current analysis cannot be directly extended to such setting. Note that \cite{tang2018d} and \cite{xin2020improved} provides an $\Omega(n^3/(1-\beta)^6)$ transient stage for D$^2$/Exact-Diffusion and stochastic gradient tracking under the  non-convex setting, respectively.  These transient analysis results, however, are worse than D-SGD in terms of network topology dependence.  There are some recent works  \cite{yuan2021decentlam,yu2019linear,lin2021quasi} that target to alleviate the influence of data heterogeneity on decentralized stochastic {\em momentum} SGD, but they do not show an improved dependence on network topology.
	
	\vspace{1mm}
	\noindent \textbf{Multi-round gossip communication.} Multi-round gossip communication has been utilized in recent works to boost the  performance of decentralized algorithms. For example, \cite{berahas2018balancing} employs multi-round gossip to balance communication and computation burdens, \cite{scaman2017optimal} develops an optimal decentralized algorithm based upon multi-round gossip for smooth and strongly convex problems in the deterministic scenario, and \cite{li2020decentralized} achieves near-optimal communication complexity with multi-round gossip and increasing penalty parameters. For decentralized and stochastic optimization, a recent work \cite{lu2021optimal} imposes multi-round gossip communication and gradient accumulation to stochastic gradient tracking to significantly improve the convergence rate under the non-convex setting. However, these works do not show how multi-round gossip communication can improve the network topology dependence in the strongly and generally convex scenarios. In addition, our analysis utilizes the special structure of D$^2$/Exact-Diffusion to achieve $\tilde{\Omega}(1/(1-\beta)^{1/2})$ transient stage for the strongly-convex scenario. The analysis in \cite{lu2021optimal}, which is tailored for gradient tracking, cannot be extended to achieve $\tilde{\Omega}(1/(1-\beta)^{1/2})$ transient stage for the  strongly-convex problems, see Remark \ref{remark-weakest-dependence}. 
	
	
	\vspace{1mm}
	\noindent \textbf{Parallel works.} Simultaneously and independently\footnote{\cite{huang2021improve} appeared online on May 11, 2021 and our work appeared online on May 17, 2021.}, a parallel work \cite{huang2021improve} has established a result under strongly-convex scenario that is partially similar to this work. 
	However, it does not study the generally-convex scenario. 
Another parallel work \cite{koloskova2021improved} established that gradient tracking has an improved network topology dependence that nearly-matches D$^2$/Exact-Diffusion (up to a $\log(1-\beta)$ term)\footnote{\cite{koloskova2021improved} appeared online around November 15, 2021.}. This implies D$^2$/Exact-Diffusion has a slightly better theoretical dependence on network topology. Moreover, D$^2$/Exact-Diffusion is more communication-efficient than gradient tracking since it only requires one communication round per iteration. We note that different from \cite{huang2021improve} and \cite{koloskova2021improved}, additional results on lower bound of homogeneous D-SGD and transient stage on D$^2$/Exact-Diffusion with multi-round gossip communication are also established in our work. Moreover, our analysis techniques are also different from \cite{huang2021improve} and \cite{koloskova2021improved}. A detailed comparison with \cite{koloskova2021improved} is listed in Table \ref{table-D2_vs_GT}. 

	\subsection{Notations}\label{sec:notations}
	Throughout the paper we let $x_i^{(k)} \in \mathbb{R}^d$ denote the local solution estimate for node $i$ at iteration $k$. Furthermore, we define the matrices
	\begin{subequations}
	\begin{align}
	\vx^{(k)} &= [x_1^{(k)}, \ldots, x_n^{(k)}]^T \in \RR^{n\times d}, \\
	\nabla F(\vx^{(k)};\bxi^{(k)}) &= [\nabla F_1(x_1^{(k)};\xi_1^{(k)}),\ldots, \nabla F_n(x_n^{(k)};\xi_n^{(k)})]^T \in \RR^{n\times d}, \\
	\nabla f(\vx^{(k)}) &= [\nabla f_1(x_1^{(k)}),\ldots, \nabla f_n(x_n^{(k)})]^T \in \RR^{n\times d},
	\end{align}
	\end{subequations}
	which collect all local variables/gradients across the network. Note that $\nabla f(\vx^{(k)}) = \mathbb{E}[\nabla F(\vx^{(k)};\bxi^{(k)})]$. We use $\mathrm{col}\{a_1, \ldots, a_n\}$
	and $\mathrm{diag}\{a_1, \ldots, a_n\}$ to denote a column vector and a diagonal matrix formed from $a_1, \ldots, a_n$. We let $\mathds{1}_n =
	\mathrm{col}\{1, \ldots , 1\} \in \RR^n$ and $I_n \in \RR^{n\times n}$ denote the identity matrix. For a symmetric matrix $A \in \RR^{n\times n}$, we let $\lambda_i(A)$  to be the $i$th largest eigenvalue and $\rho(A)=\max_i |\lambda_i(A)|$ denote the spectral radius of matrix $A$. In addition, we let $[n]=\{1,\ldots, n\}$ for any positive integer $n$. Suppose that $A\in \mathbb{R}^{n\times n}$ is a positive semidefinite matrix with  eigen-decomposition $A = U\Lambda U^T$ where $U\in \mathbb{R}^{n\times n}$ is an orthogonal matrix and $\Lambda\in \mathbb{R}^{n\times n}$ is a non-negative diagonal matrix. Then, we let $A^{\frac{1}{2}} = U\Lambda^{\frac{1}{2}}U^T \in \mathbb{R}^{n\times n}$ be the square root of the matrix $A$. Note that $A^{\frac{1}{2}}$ is also positive semidefinite and $A^{\frac{1}{2}} \times A^{\frac{1}{2}} = A$. For a vector $a$, we let $\|a\|$ denote its $\ell_2$ norm. For a matrix $A$, we let $\|A\|$ denote its $\ell_2$ norm and $\|A\|_F$ denote its Frobenius norm. We use $\lesssim$ and $\gtrsim$ to indicate inequalities that hold by omitting  absolute constants. 
	
	\section{Preliminaries and Assumptions}
	\label{sec:ass}

	
	
	
	
	
	
	
	\subsection{Weight Matrix}
	To model the decentralized communication, we let $w_{ij} \ge 0$ be the weight used by node $i$ to scale information flowing from node $j$ to node $i$. We let $\cN_i$ denote the neighbors of node $i$, including node $i$ itself.  We let $w_{ij}=0$ if  node $j \notin \cN_i$. We define the weight matrix $W = [w_{ij}] \in \RR^{n\times n}$ and assume the following condition.
	\begin{assumption}[\sc Weight matrix]
		\label{ass-W}
		The network is strongly connected and the weight matrix $W$ is doubly stochastic and symmetric, {\em i.e.}, $W = W^T$ and $W\mathds{1}_n = \mathds{1}_n$. 
		$\hfill \square$
	\end{assumption}
	\begin{remark}[\sc Spectral gap]\label{rm-beta}
		Under Assumption \ref{ass-W}, it holds that $1 = \lambda_1(W) > \lambda_2(W) \ge \cdots \ge \lambda_n(W) > -1$. If we let $\beta = \rho(W - \frac{1}{n}\mathds{1}\mathds{1}^T)$, it follows that $\beta = \max\{|\lambda_2(W)|, |\lambda_n(W)|\} \in (0,1)$. The network quantity $1-\beta$ is called the spectral gap of $W$, which reflects the connectivity of the network topology. The scenario $1 - \beta \to 1$ implies a well-connected topology (e.g., for fully connected topology, we can choose $W = \frac{1}{n}\mathds{1}\mathds{1}^T$ and hence $\beta = 0$). In contrast, the scenario $1 - \beta \to 0$ implies a badly-connected topology. $\hfill \square$ 
	\end{remark}
		\subsection{D$^2$/Exact-Diffusion Algorithm}
	Using notations $\vx^{(k)}$ and $\nabla F(\vx^{(k)};\bxi^{(k)})$ introduced in Sec.~\ref{sec:notations}, the D$^2$ algorithm \cite{tang2018d}, also known as Exact-Diffusion in  \cite{yuan2017exact1,yuan2018exact2,yuan2020influence} or NIDS in \cite{li2019convergence}, can be written as 
	\begin{align}\label{eq:d2}
	\vx^{(k+1)} = \bar{W}\Big( 2 \vx^{(k)} - \vx^{(k-1)} - \gamma \big(\nabla F(\vx^{(k)};\bxi^{(k)}) - \nabla  F(\vx^{(k-1)};\bxi^{(k-1)})\big)\Big), \quad \forall~ k=1,2,\ldots
	\end{align}
	where  $\bar{W} := (W + I)/2$ and $\gamma$ is the learning rate (stepsize) parameter. The algorithm is initialized with $\vx^{(1)} = \bar{W}\big(\vx^{(0)} - \gamma \big(\nabla F(\vx^{(0)};\bxi^{(0)})\big)$ for any $\vx^{(0)}$. Different from the vanilla decentralized SGD (D-SGD), D$^2$/Exact-Diffusion exploits the last two consecutive iterates and stochastic gradients to update the current variable. Such algorithm construction is proved to be able to remove the influence of data heterogeneity, see \cite{tang2018d,yuan2018exact2,yuan2020influence,li2019decentralized}. Recursion \eqref{eq:d2} can be conducted in a decentralized manner as listed in Algorithm \ref{Algorithm: D2} (see \cite{yuan2017exact1,yuan2020influence} for details). A fundamental difference from  
	D-SGD lies in the solution correction step. If $x_i^{(k)} - \psi_i^{(k)}$ is removed, then Algorithm \ref{Algorithm: D2} reduces to D-SGD.
	
	\begin{algorithm}[h!]
		\caption{D$^2$/Exact-Diffusion}
		\label{Algorithm: D2}
		
		\textbf{Require:} Let $\bar{W}=(I+W)/2$ and initialize $x_i^{(0)}=0$; let $\psi_i^{(0)} = x_i^{(0)}$.
		
		\vspace{1mm}
		\textbf{for} $k=0,1,2,...$, every node $i$ \textbf{do}
		
		\vspace{0.5mm}
		\hspace{5mm} 	Sample $\xi^{(k)}_{i}$ and calculate $g_i^{(k)} = \nabla F(x_i^{(k)};\xi_i^{(k)})$; \
		
		\hspace{5mm} Update $\psi_i^{(k+1)} = x_i^{(k)} - \gamma g_i^{(k)}$; $\hspace{3.55cm}  \triangleright  \mbox{ \footnotesize{local gradient descent step}}$
		
		\hspace{5mm} Update $\phi_i^{(k+1)} = \psi_i^{(k+1)} + x_i^{(k)} - \psi_i^{(k)}$; $\hspace{2.3cm}  \triangleright  \mbox{ \footnotesize{solution correction step}}$
		
		\hspace{5mm} Update	$x^{(k+1)}_i = \sum_{j\in \mathcal{N}_i} \bar{w}_{ij}\, \phi_j^{(k+1)}$; $\hspace{2.77cm}   \triangleright \mbox{ \footnotesize{communication step}}$
	\end{algorithm}
	
\noindent	\textbf{Primal-dual form.}
	For  analysis purposes, we rewrite recursion \eqref{eq:d2} into the following primal-dual equivalent form \cite{yuan2018exact2,li2019decentralized}:
	\begin{align}
	\begin{cases}\label{eq:d2-pd}
	\vx^{(k+1)} = \bar{W}\big( \vx^{(k)} - \gamma \nabla F(\vx^{(k)};\bxi^{(k)}) \big) - V \vy^{(k)}, \\
	\vy^{(k+1)} = \vy^{(k)} + V \vx^{(k+1)},\quad \forall~ k = 0, 1,2,\ldots
	\end{cases}
	\end{align}
	where $\vy = [\y_1,\ldots,\y_n]^T \in \mathbb{R}^{n\times d}$, $\y_i \in \RR^d$, and $V=(I-\bar{W})^{1/2} \in \mathbb{R}^{n\times n}$ is a positive semi-definite matrix. For initialization, we let $\vy^{(0)} = 0$. The equivalence between \eqref{eq:d2} and \eqref{eq:d2-pd} can be easily verified, see \cite{yuan2017exact1,yuan2020influence}.
	
	\subsection{Optimality Condition}
	The primal-dual recursion \eqref{eq:d2-pd} facilitates the following optimality condition of problem \eqref{eq:general-prob}. 
	\begin{lemma} [\sc Optimality Condition]\label{lm-opt-cond}
		Assume each cost function $f_i(x)$ in problem \eqref{eq:general-prob} is convex. 
		Then, there exists some primal-dual pair $(\vx^\star, \vy^\star)$, {in which $\vy^\star$ lies in the range space of $V$}, that satisfies 
	\begin{subequations}
	        	\begin{align}
		\gamma \bar{W} \nabla f(\vx^\star) + V \vy^\star &= 0, \label{eq-opt-cond-1} \\
		V \vx^\star &= 0, \label{eq-opt-cond-2}
		\end{align}
	\end{subequations}
		and it holds that $\x_1^\star = \cdots = \x_n^\star = \x^\star$ where $\x^\star$ is a global solution to problem \eqref{eq:general-prob}.  $\hfill \blacksquare$
	\end{lemma}
	The proof of the above lemma is simple and can be referred to \cite[Lemma~3.1]{shi2015extra}. It is worth noting that when there is no gradient noise, \textit{i.e.}, $\nabla F(\vx^{(k)};\bxi^{(k)}) = \nabla f(\vx^{(k)})$, the fixed point $(\vx^o, \vy^o)$ of the primal-dual recursion \eqref{eq:d2-pd} satisfies the optimality condition \eqref{eq-opt-cond-1}--\eqref{eq-opt-cond-2}. This implies the iterates $\x^{(k)}$ generated by the D$^2$/Exact-Diffusion algorithm will converge to a global solution $\x^\star$ of problem \eqref{eq:general-prob} in expectation. Such conclusion holds without any assumption whether the data distribution is homogeneous or not.  
	
	\subsection{D-SGD and Data-Heterogeneity Bias}
	\label{sec-dsgd-bias}
	The (adapt-then-combine) D-SGD algorithm (also known as diffusion in \cite{lopes2008diffusion,chen2012diffusion,sayed2014adaptation}) \cite{nedic2009distributed,lopes2008diffusion,chen2012diffusion,lian2017can} has the update form:
	\begin{align}
	\label{diffusion}
	\vx^{(k+1)} = W\big( \vx^{(k)} - \gamma \nabla F(\vx^{(k)};\bxi^{(k)}) \big).
	\end{align}
	Without any auxiliary variable to help correct the gradient direction like D$^2$/Exact-Diffusion recursion \eqref{eq:d2-pd}, D-SGD suffers from a solution deviation caused by data heterogeneity even if there is no gradient noise. Since data heterogeneity exists, we have $b^2 = \frac{1}{n}\sum_{i=1}^n \|\nabla f_i(x^\star)\|^2 > 0$, which implies there exists at least one node $i$ such that $\grad f_i(x^\star) \neq 0$. 
	Suppose there is no gradient noise and $\vx^{(k)}$ is initialized as a consensual solution $\vx^\star = [\x^\star,\ldots, \x^\star]^T$ where $\x^\star$ is a solution to problem \eqref{eq:general-prob}, which satisfies $\mathds{1}_n^T \nabla f(\vx^\star) = 0$. Following recursion \eqref{diffusion}, we have
	\begin{align}\label{xnsdbsd8}
	\vx^{k+1} = W\big( \vx^{\star} - \gamma \nabla f(\vx^{\star}) \big) = \vx^{\star} - \gamma W \nabla f(\vx^{\star}) \neq \vx^\star,
	\end{align}
	in which the last inequality holds because $W\nabla f(\vx^\star) \neq 0$ in general. Relation \eqref{xnsdbsd8} implies that even if D-SGD starts from the optimal consensual solution, it can still jump away to a biased solution due to the data heterogeneity affect.
	
	\subsection{Assumptions}
	We now introduce some standard assumptions that will be used throughout the paper.
	
	\begin{assumption}[\sc Convexity]
		\label{ass:convex}
		Each cost function $f_i(x)$ is convex. 
	\end{assumption}
	
	\begin{assumption}[\sc Smoothness]
		\label{ass:smoothness}
		Each local cost function $f_i(x)$ is differentiable, and there exists a constant $L$ such that for each $x, y\in\RR^d$
		\begin{align}
		\|\nabla f_i(x) - \nabla f_i(y)\| \le L \|x - y\|,\quad \forall\, i \in [n]. \label{smooth-1}
		\end{align}
	\end{assumption}
	
	\begin{assumption}[\sc Gradient noise]\label{ass:gradient-noise} 
	Define the filtration  $\cF^{(k)} \hspace{-0.8mm}=\hspace{-0.8mm} \big\{\hspace{-0.8mm} \{\hspace{-0.4mm}\xi_i^{(k)}\hspace{-0.4mm}\}_{i=1}^n, 
		\{\hspace{-0.4mm}x_i^{(k)}\hspace{-0.4mm}\}_{i=1}^n, \ldots, \{\hspace{-0.4mm}\xi_i^{(0)}\hspace{-0.4mm}\}_{i=1}^n,  \{\hspace{-0.4mm}x_i^{(0)}\hspace{-0.4mm}\}_{i=1}^n\hspace{-0.4mm} \big\}$. Then, it is assumed that for any $k$ and $i$ that 
	\begin{subequations}
	        	\begin{align}
		\mathbb{E}[\nabla F(x_i^{(k)};\xi_i^{(k)}) - \nabla f_i(x_i^{(k)})|\cF^{(k-1)}] &= 0, \label{gd-1}\\
		\mathbb{E}[\|\nabla F(x_i^{(k)};\xi_i^{(k)}) - \nabla f_i(x_i^{(k)})\|^2|\cF^{(k-1)}] &\le \sigma^2 \label{gd-2},
		\end{align}
	\end{subequations}
		for some constant $\sigma^2 \geq 0$. Moreover, we assume $\xi_i^{(k)}$ are independent of each other for any $k$ and $i$. 
	\end{assumption}

	\section{Fundamental Transformation}
	In this section we transform the primal-dual update of the D$^2$/Exact-Diffusion algorithm \eqref{eq:d2-pd} into another equivalent recursion. This transformation, inspired by \cite{yuan2018exact2,yuan2020influence}, is fundamental to establish the refined convergence results for D$^2$/Exact-Diffusion. 
	
	Let  $(\vx^\star, \vy^\star)$ be one pair of variables that satisfies the optimality conditions in Lemma \ref{lm-opt-cond}. Subtracting \eqref{eq-opt-cond-1} and \eqref{eq-opt-cond-2} from the primal-dual recursion \eqref{eq:d2-pd}, we have
\begin{subequations}
        	\begin{align}
	\vx^{(k+1)} - \vx^\star &= \bar{W}\big( \vx^{(k)} - \vx^\star - \gamma(\nabla F(\vx^{(k)};\bxi^{(k)}) - \nabla f(\vx^\star)) \big) - V (\vy^{(k)} - \vy^\star), \label{eq:d2-pd-1}\\
	\vy^{(k+1)} - \vy^\star &= \vy^{(k)} - \vy^\star + V (\vx^{(k+1)} - \vx^\star).\label{eq:d2-pd-2}
	\end{align}
\end{subequations}
	To simplify the notation, we define $\vs^{(k)}:=\nabla F(\vx^{(k)};\bxi^{(k)}) - \nabla f(\vx^{(k)})$ as the gradient noise at iteration $k$. By substituting \eqref{eq:d2-pd-1} into \eqref{eq:d2-pd-2}, and recalling that $I - V^2 = \bar{W}$, we obtain 
	\begin{align}\label{xzn2websd7}
		\left[
		\begin{array}{c}
		\hspace{-2mm}\vx^{(k+1)} - \vx^\star  \hspace{-1.5mm}\\
		\hspace{-2mm}\vy^{(k+1)} - \vy^\star \hspace{-1.5mm}
		\end{array}
		\right] \hspace{-0.6mm} = \hspace{-0.6mm}
		\underbrace{\left[
			\begin{array}{cr}
			\hspace{-2mm}\bar{W} & -V \hspace{-1.5mm}\\
			\hspace{-2mm}V \bar{W} & \bar{W} \hspace{-1.5mm}
			\end{array}
			\right]}_{B}
		\hspace{-1mm}
		\left[
		\begin{array}{c}
		\hspace{-2mm}\vx^{(k)} - \vx^\star \hspace{-1.5mm}\\
		\hspace{-2mm}\vy^{(k)} - \vy^\star \hspace{-1.5mm}
		\end{array}
		\right]
		- \gamma 
		\left[
		\begin{array}{c}
		\hspace{-2mm}\bar{W}(\nabla f(\vx^{(k)}) - \nabla f(\vx^\star) + \vs^{(k)}) \hspace{-1.5mm}\\
		\hspace{-2mm} V \bar{W}(\nabla f(\vx^{(k)}) - \nabla f(\vx^\star) + \vs^{(k)}) \hspace{-1.5mm}
		\end{array}
		\right].
		\end{align}
	The main difficulty analyzing the convergence of the above recursion (as well as \eqref{eq:d2-pd-1}--\eqref{eq:d2-pd-2}) is that terms $\vx^{(k)} - \vx^\star$ and $\vy^{(k)} - \vy^\star$ are entangled together to update $\vx^{(k+1)} - \vx^\star$ or $\vy^{(k+1)} - \vy^\star$. For example, the update of $\vx^{(k+1)} - \vx^\star$ relies on both $\bar{W}(\vx^{(k)} - \vx^\star)$ and $-V(\vy^{(k)} - \vy^\star)$. In the following, we identify a 
	change of basis and transform \eqref{xzn2websd7} into another equivalent form so that the involved iterated variables can be ``decoupled''. To this end, we need to introduce a fundamental decomposition lemma. This lemma was first established in \cite{yuan2018exact2}. We have improved this lemma by establishing an upper bound of an important term (see \eqref{bound-XLXR}) that is critical for our later analysis. 
	
	\begin{lemma}[\sc Fundamental decomposition]\label{lm-decom}
		Under Assumption \ref{ass-W}, the matrix $B \in \mathbb{R}^{2n\times 2n}$ in \eqref{xzn2websd7} can be
		diagonalized as
		\begin{align}\label{eq-fundamental-decomposision}
		B = \underbrace{[r_1\ r_2\ c X_R]}_{X}
		\underbrace{\ba{ccc}
			1 & 0 & 0\\
			0 & 1 & 0\\
			0 & 0 & D_1
			\ea}_{D}
		\underbrace{\ba{c}
			\ell_1^T \\
			\ell_2^T \\
			X_L/c
			\ea}_{X^{-1}}
		\end{align}
		for any constant $c > 0$ where $D \in \mathbb{R}^{2n \times 2n}$ is a diagonal matrix. Moreover, we have 
		\begin{align}\label{zb236asd00991}
		r_1 = 
		\ba{c}
		\mathds{1}_n\\
		0
		\ea, \quad 
		r_2 = 
		\ba{c}
		0 \\
		\mathds{1}_n
		\ea, \quad
		\ell_1 = 
		\ba{c}
		\frac{1}{n}\mathds{1}_n\\
		0
		\ea, \quad
		\ell_2 = 
		\ba{c}
		0 \\
		\frac{1}{n}\mathds{1}_n
		\ea
		\end{align}
		and $X_R\in\RR^{2n\times 2(n-1)}$, $X_L\in \RR^{2(n-1)\times 2n}$.  Also, the matrix $D_1$ is a
		diagonal matrix with
		diagonal entries strictly less than $1$ in magnitude and
		\begin{align}\label{D1-norm2}
		\|D_1\| = \bar{\lambda}_2^{1/2}, \quad \mbox{where} \quad \bar{\lambda}_2 = \frac{1+\lambda_2(W)}{2}.
		\end{align}
		Furthermore,  it holds that 
		\begin{align}\label{bound-XLXR}
		\|X_L\|\|X_R\| \le \bar{\lambda}^{-1/2}_n
		\end{align}
		where $\bar{\lambda}_n = (1+\lambda_n(W))/2$.  (Proof is in Appendix \ref{app:decom}). $\hfill \blacksquare$
	\end{lemma}
	
	
	Left-multiplying  $X^{-1}$ to both sides of \eqref{xzn2websd7} and using Lemma \ref{lm-decom}, we can get the transformed recursion given in Lemma \ref{lm-transform-error-dyanmic}. Note that Lemma \ref{lm-transform-error-dyanmic} is also an improved version of \cite[Lemma~3]{yuan2020influence}, which will facilitate  our sharper convergence analysis of the D$^2$/Exact-Diffusion algorithm. 
		\begin{lemma}[\sc Transformed Recursion] \label{lm-transform-error-dyanmic}
			Under Assumptions \ref{ass-W}, the D$^2$/Exact-Diffusion error recursion \eqref{xzn2websd7} can be transformed into
			\begin{align}
			\label{transformed-error-dynamics}
			\ba{c}
			\bar{\z}^{(k+1)} \\
			\check{\vz}^{(k+1)}
			\ea	
			\hspace{-1mm}&=\hspace{-1mm} 
			\ba{c}
			\bar{\z}^{(k)} - \frac{\gamma}{n}\mathds{1}^T (\nabla f(\vx^{(k)}) - \nabla f(\vx^\star)) \\
			D_1 \check{\vz}^{(k)} - \frac{\gamma}{c} \check{\vg}^{(k)}
			\ea - \gamma  
			\ba{c}
			\bar{\vs}^{(k)}\\
			\frac{1}{c} \check{\vs}^{(k)}
			\ea,
			\end{align}
			where $\check{\vg}^{(k)}, \bar{\vs}^{(k)}$ and $\check{\vs}^{(k)}$ are defined as
		\begin{subequations}
		        	\begin{align}
			\check{\vg}^{(k)} & \define \big(X_{L, \ell} Q_R  + X_{L,r} Q_R (I - \bar{\Lambda}_R)^{\frac{1}{2}} \big) \bar{\Lambda}_R Q_R^T (\nabla f(\vx^{(k)}) - \nabla f(\vx^\star)) \label{g-check}\\
			\bar{\vs}^{(k)} & \define \frac{1}{n}\mathds{1}^T\vs^{(k)} \label{s-bar}\\
			\check{\vs}^{(k)} & \define  \big(X_{L, \ell} Q_R  + X_{L,r} Q_R (I - \bar{\Lambda}_R)^{\frac{1}{2}} \big) \bar{\Lambda}_R Q_R^T \vs^{(k)}. \label{s-check}
			\end{align}
		\end{subequations}
			Here, the matrices $X_{L,\ell}, X_{L,r} \in \RR^{2(n-1) \times n} $ are the left and right part of the matrix $X_{L} = [X_{L,\ell}\ \ X_{L,r}]$, respectively,  $\bar{\Lambda}_R = \mathrm{diag}\{\bar{\lambda}_2(W),\dots, \bar{\lambda}_n(W)\} \in \RR^{(n-1)\times (n-1)}$, and $Q_R\in \RR^{n\times (n-1)}$ is defined in \eqref{zbnzb0zbzbz09}.
			The relation between the original and the transformed error vectors are 
			\begin{align}\label{w-y-vs-bz-cz}
			\ba{c}
			\hspace{-1mm}\bar{\z}^{(k)}\hspace{-1mm} \\
			\hspace{-1mm}\check{\vz}^{(k)}\hspace{-1mm}
			\ea = 
			\ba{c}
			\ell_1^T\\
			X_L/c 
			\ea
			\ba{c}
			\hspace{-1mm}\vx^{(k)} - \vx^\star\hspace{-1mm}\\
			\hspace{-1mm}\vy^{(k)} - \vy^\star\hspace{-1mm}
			\ea
			\quad \mbox{and} \quad
			\ba{c}
			\hspace{-1mm}\vx^{(k)} - \vx^\star\hspace{-1mm}\\
			\hspace{-1mm}\vy^{(k)} - \vy^\star\hspace{-1mm}
			\ea
			\hspace{-1mm}	= {\hspace{-1mm}
				\ba{ccc}
				\hspace{-1mm}r_1 \hspace{-1mm}&\hspace{-1mm} c X_R\hspace{-1mm}
				\ea}\hspace{-1mm}
			\ba{c}
			\hspace{-1mm}\bar{\z}^{(k)}\hspace{-1mm} \\
			\hspace{-1mm}\check{\vz}^{(k)}\hspace{-1mm}
			\ea.\hspace{-1mm}
			\end{align}
			Note that $\bar{\z}^{(k)} \in \RR^{1 \times d}$ and $\check{\vz}^{(k)} \in \RR^{2(n-1)\times d}$ (Proof is in Appendix \ref{app-transform}). $\hfill \blacksquare$
	\end{lemma}
	
	\begin{remark}[\sc recursion interpretation]
		Using the left relation in \eqref{w-y-vs-bz-cz} and the definition of $\ell_1$ in  \eqref{zb236asd00991}, it holds that 
		\begin{align}
		[\bar{\z}^{(k)}]^T = \frac{1}{n}\mathds{1}^T(\vx^{(k)} - \vx^\star) = \bar{x}^{(k)} - x^\star.
		\end{align}
		Therefore, $\bar{\z}^{(k)}$ gauges the distance between the averaged variable $\bar{x}^{(k)}$ and the solution $x^\star$. On the other hand, using the right relation in \eqref{w-y-vs-bz-cz} and the definition of $r_1$ in \eqref{zb236asd00991}, it holds that $\vx^{(k)} - \vx^\star = \mathds{1}_n \bar{\z}^{(k)} + c X_{R,u} \check{\vz}^{(k)} = \bar{\vx}^{(k)} - \vx^\star + c X_{R,u} \check{\vz}^{(k)}$, where $X_{R,u} \in \RR^{n\times 2(n-1)}$ is the upper part of matrix $X_R = [X_{R,u};X_{R,d}]$. This implies 
		\begin{align}\label{zzznnn}
		c X_{R,u} \check{\vz}^{(k)} = \vx^{(k)} - \bar{\vx}^{(k)}.
		\end{align}
		Hence, $\check{\vz}^{(k)}$ measures the consensus error, {\em i.e.}, the distance between ${\vx}^{(k)}$ and $\bar{\vx}^{(k)}$.  $\hfill \square$
	\end{remark}
	
	The following proposition establishes the magnitude of $\|\check{\vz}^{(0)}\|^2$, which will be used in later derivations. 
	
	
	{
		\begin{proposition}[\sc The magnitude of $\|\check{\vz}^{(0)}\|_F^2$] \label{remark-z0}
			If we initialize $\vx^{(0)} = 0$, $\vy^{(0)} = 0$ and set $c = \|X_L\|$, it holds that $\|\check{\vz}^{(0)}\|_F^2 \le \frac{\gamma^2 \bar{\lambda}_2^2 \|\nabla f(\vx^\star)\|_F^2}{1-\bar{\lambda}_2}$. If we further assume $\|\nabla f(\vx^\star)\|_F^2 = \sum_{i=1}^n\|\nabla f_i(x^\star)\|^2 = O(n)$, it follows that $\|\check{\vz}^{(0)}\|_F^2 = O(\frac{n \gamma^2 \bar{\lambda}_2^2}{1-\bar{\lambda}_2})$ (Proof is in Appendix \ref{app-proof-remark}).  $\hfill \blacksquare$
		\end{proposition}
	}
	

	\section{Convergence Results: Generally-Convex Scenario}
	\label{sec:general-convex}
	With Assumption \ref{ass:gradient-noise}, it is easy to verify that 
	\begin{align}\label{avg-noise}
	\mathbb{E}[\|\bar{\vs}^{(k)}\|^2|\cF^{(k-1)}] \le \frac{\sigma^2}{n}, \quad k=1,2,\ldots
	\end{align}
	We first establish a descent lemma for the D$^2$/Exact-Diffusion algorithm in the generally-convex setting, which describes how $\mathbb{E}\|\bar{\z}^{(k)}\|^2$ evolves with iteration.
	\begin{lemma}[\sc Descent Lemma]\label{lm-descent-gc}
		Under Assumptions \ref{ass:convex}--\ref{ass:gradient-noise} and learning rate $\gamma < \frac{1}{4L}$, it holds for $k=0,1,\dots$ that 
		\begin{align}\label{23bsd999-3}
		& \mathbb{E}\|\bar{\z}^{(k+1)}\|^2 \le \mathbb{E}\|\bar{\z}^{(k)}\|^2 \hspace{-0.5mm}-\hspace{-0.5mm} \gamma \big(\mathbb{E} f(\bar{x}^{(k)}) \hspace{-0.5mm}-\hspace{-0.5mm} f(x^\star)\big) \hspace{-0.5mm}+\hspace{-0.5mm} \frac{3L\gamma }{2n \bar{\lambda}_n} \mathbb{E}\|\check{\vz}^{(k)}\|_F^2 \hspace{-0.5mm}+\hspace{-0.5mm} \frac{\gamma^2 \sigma^2}{n},\quad k=0,1,\dots 
		\end{align}
		where $\bar{\lambda}_n = \lambda_n(\bar{W}) = (1 + \lambda_n(W))/2$. (Proof is in Appendix \ref{app-lm-descent-gc})
		$\hfill \blacksquare$
	\end{lemma}
	With inequality \eqref{23bsd999-3}, we have for $T\ge 0$ that 
	\begin{align}\label{c286}
	\frac{1}{T+1}\sum_{k=0}^T\big(\mathbb{E} f(\bar{x}^{(k)}) \hspace{-0.5mm}-\hspace{-0.5mm} f(x^\star)\big) \le \frac{\mathbb{E}\|\bar{\z}^{(0)}\|^2}{\gamma (T+1)} + \frac{3L}{2n \bar{\lambda}_n(T+1)}\sum_{k=0}^T \mathbb{E}\|\check{\vz}^{(k)}\|_F^2 + \frac{\gamma \sigma^2}{n}.
	\end{align}
	We next bound the ergodic consensus term on the right-hand-side. 
	{
		\begin{lemma}[\sc Consensus Lemma]\label{lm-consensus}
			Under Assumptions \ref{ass-W}-\ref{ass:gradient-noise} and learning rate $\gamma \le  \frac{(1-\beta_1)\bar{\lambda}_n^{1/2}}{4\bar{\lambda}_2 L} $, it holds that
			\begin{align}\label{xxcznzxnxcn}
			\mathbb{E}\|\check{\vz}^{(k+1)}\|_F^2 &\le  \left(\frac{1+\beta_1}{2}\right) \mathbb{E}\|\check{\vz}^{(k)}\|_F^2 + \frac{16 n \gamma^2 \bar{\lambda}_2^2 L}{1-\beta_1} (\mathbb{E}f(\bar{x}^{(k)}) - f(x^\star))  +  4 n \gamma^2  \bar{\lambda}_2^2 \sigma^2, \quad k=0,1,\cdots
			\end{align}
			where $\beta_1 = \bar{\lambda}^{1/2}_2$, $\bar{\lambda}_2 = (1 + \lambda_2(W))/2$, and $\bar{\lambda}_n = (1 + \lambda_n(W))/2$ (Proof is in Appendix \ref{app-gc-consensus}).
			$\hfill \blacksquare$
		\end{lemma}
		
		\begin{lemma}[\sc Ergodic Consensus Lemma]\label{lm-ergodic-consensus}
			Under the same assumptions as Lemma \ref{lm-consensus}, it holds that (Proof is in Appendix \ref{app-lm-ergodic-consensus})
			\begin{align}\label{eq:ergodic-consensus}
			\frac{1}{T+1}\sum_{k=0}^T \mathbb{E}\|\check{\vz}^{(k)}\|_F^2 \le &\  \frac{32 n \gamma^2 \bar{\lambda}_2^2 L}{(1-\beta_1)^2(T+1)}\sum_{k=0}^{T}\big(\mathbb{E} f(\bar{x}^{(k)}) - f(x^\star) \big) +  \frac{8 n\gamma^2\bar{\lambda}_2^2\sigma^2}{1-\beta_1} + \frac{3\mathbb{E}\|\check{\vz}^{(0)}\|_F^2}{(1-\beta_1)(T+1)} 
.			\end{align}
			$\hfill \blacksquare$
	\end{lemma}}
	With inequalities \eqref{c286} and \eqref{eq:ergodic-consensus}, and the fact that $\beta^2_1 = \bar{\lambda}_2 = (1 + \lambda_2(W))/2 \le (1+\beta)/2$ where $\beta=\rho(W - \frac{1}{n}\mathds{1}\mathds{1}^T)$ is defined in Remark \ref{rm-beta}, we can show the following convergence result for D$^2$/Exact-Diffusion in the generally convex scenario. 
	\begin{theorem}[\sc Convergence Property]\label{thm-generally-convex}
		Under Assumptions \ref{ass-W}-\ref{ass:gradient-noise} and learning rate
		\begin{align}
		\gamma = \min\left\{\frac{1}{4L}, \frac{(1-\beta_1)\bar{\lambda}_n^{1/2}}{10 L \bar{\lambda}_2}, \left(\frac{r_0}{r_1(T+1)}\right)^{\frac{1}{2}}, \left(\frac{r_0}{r_2(T+1)}\right)^{\frac{1}{3}}, \left(\frac{r_0}{r_3}\right)^{\frac{1}{3}}\right\}
		\end{align}
		where $r_0, r_1$, $r_2$ and $r_2$ are constants defined in \eqref{sndsnnds}, and $\bar{\lambda}_n$ is bounded away from $0$, then it holds that
		\begin{align}\label{xbsd87}
		\frac{1}{T+1}\sum_{k=0}^T\big(\mathbb{E}f(\bar{x}^{(k)}) - f(x^\star)\big) = O\Big( \frac{\sigma}{\sqrt{nT}} + 
		\frac{\sigma^{\frac{2}{3}}}{(1-\beta)^{\frac{1}{3}}T^{\frac{2}{3}}} +  \frac{1}{(1-\beta)T}\Big).
		\end{align}
		(Proof is in Appendix \ref{app-thm-gc})
		$\hfill \blacksquare$
	\end{theorem}
	\begin{corollary}[\sc Transient stage]
		Under the same assumptions as Theorem \ref{thm-generally-convex}, the transient stage for D$^2$/Exact-Diffusion is on the order of $\Omega\left(\frac{n^3}{(1-\beta)^2}\right)$. 
	\end{corollary}
	\begin{proof}
		To achieve the linear speedup stage, $T$ has to be large enough so that 
		\begin{align}
		\frac{\sigma^{\frac{2}{3}}}{(1-\beta)^{\frac{1}{3}}T^{\frac{2}{3}}}\lesssim \frac{\sigma}{\sqrt{nT}} \quad \text{and} \quad  \frac{1}{(1-\beta)T} \lesssim \frac{\sigma}{\sqrt{nT}},
		\end{align}
		which is equivalent to $T \gtrsim \max\big\{ \frac{n^3}{(1-\beta)^2 \sigma^2}, \frac{n}{(1-\beta)^2 \sigma^2}\big\} = \Omega\big(\frac{n^3}{(1-\beta)^2}\big)$.
	\end{proof}
	\begin{remark} The result from \cite{koloskova2020unified} indicates that the transient stage of D-SGD for the generally convex scenario is on the order of $\Omega\big(\frac{n^3}{(1-\beta)^4}\big)$. By removing the influence of the data heterogeneity, D$^2$/Exact-Diffusion improves the transient stage to $\Omega\big(\frac{n^3}{(1-\beta)^2}\big)$, which  has a better network topology dependence on $1-\beta$. 
	$\hfill \square$
	\end{remark}
	\begin{remark}\label{rmk-gt}
	An independent and parallel work in \cite{koloskova2021improved} shows that gradient tracking can also improve the transient stage of D-SGD. The convergence rate and transient stage in the generally-convex scenario established in \cite{koloskova2021improved} are listed in Table \ref{table-D2_vs_GT}. It is observed that the transient stage of D$^2$/Exact-Diffusion is better than gradient tracking by a factor $\log^2(\frac{1}{1-\beta}(1+\log(\frac{1}{1-\beta})))$. Moreover, D$^2$/Exact-Diffusion is more communication-efficient than gradient tracking since it only requires one  communication  round  per  iteration. 
	\end{remark}
    \begin{table}[t]
		\centering
		\begin{tabular}{rcc}
			\toprule
			\textbf{Algorithm} & \textbf{Convergence rate}                         & \textbf{Transient stage}                    \\ \midrule
			\textbf{Gradient-tracking} \cite{koloskova2021improved}  & $O\left(\frac{\sigma}{\sqrt{n T}} + \frac{\tau^{1/3}\sigma^{2/3}}{ T^{2/3}} + \frac{\tau}{T}\right)$     & $\Omega\left(\frac{n^3}{(1-\beta)^2}\log^2(\frac{1}{1-\beta}(1 \hspace{-0.5mm}+\hspace{-0.5mm} \log(\frac{1}{1-\beta})))\right)$ \vspace{1mm}\\
			\textbf{D$^2$/Exact-Diffusion (Ours)}  & $O\left(\frac{\sigma}{\sqrt{n T}} \hspace{-0.5mm} +\hspace{-0.5mm} \frac{\sigma^{2/3}}{(1-\beta)^{1/3} T^{2/3}} \hspace{-0.5mm}+\hspace{-0.5mm} \frac{1}{(1-\beta)T}\right)$    & $\Omega\left(\frac{n^3}{(1-\beta)^2}\right)$       \\
			\bottomrule
		\end{tabular}
		\caption{\small Convergence rate and transient stage comparison in the generally-convex scenario between our results and a parallel work on gradient tracking \cite{koloskova2021improved}. In the table, $\tau = \frac{2}{1-\beta}\log(\frac{50}{1-\beta}(1 + \log(\frac{1}{1-\beta})))+1$ is in \cite[Lemma 20]{koloskova2021improved}. }
		\label{table-D2_vs_GT}
	\end{table}	
	
	\section{Convergence results: Strongly-Convex Scenario}
	
	\subsection{Convergence Analysis of D$^2$/Exact-Diffusion}
	\label{sec:strongly-convex}
	
	In this subsection we establish the convergence rate of D$^2$/Exact-Diffusion in the strongly convex scenario and examine its transient stage. 
	
	\begin{assumption}[\sc strongly convex] \label{ass:sc}
		Each $f_i(x)$ is strongly convex, {\em i.e.}, there exists a constant $\mu > 0$ such that for any $x, y \in \RR^d$ we have:
		\begin{align}
		f_i(y) \ge f_i(x) + \langle \nabla f_i(x), y - x \rangle + \frac{\mu}{2}\|y-x\|^2 , \quad \forall~i\in[n].
		\end{align} 
		$\hfill \square$
	\end{assumption}
	
	\begin{lemma}[\sc Descent Lemma]\label{lm-sc-descent}
		When Assumptions \ref{ass:smoothness}--\ref{ass:sc} hold and learning rate $\gamma < \frac{1}{4L}$, it holds for $k=0,1,\ldots$ that
		\begin{align}\label{23bsd999-sc}
		& \mathbb{E}\|\bar{\z}^{(k+1)}\|^2 \le \left(1 - \frac{\gamma \mu}{2}\right)\mathbb{E}\|\bar{\z}^{(k)}\|^2 \hspace{-0.5mm}-\hspace{-0.5mm} \gamma \big(\mathbb{E} f(\bar{x}^{(k)}) \hspace{-0.5mm}-\hspace{-0.5mm} f(x^\star)\big) \hspace{-0.5mm}+\hspace{-0.5mm} \frac{5L\gamma }{2n \bar{\lambda}_n} \mathbb{E}\|\check{\vz}^{(k)}\|_F^2 \hspace{-0.5mm}+\hspace{-0.5mm} \frac{\gamma^2 \sigma^2}{n}.
		\end{align}
		(Proof is in Appendix \ref{app-sc-descent})
		$\hfill \blacksquare$
	\end{lemma}
	With inequality \eqref{23bsd999-sc}, we have 
	\begin{align}\label{23bsd999-sc-thm-n1}
	\mathbb{E} f(\bar{x}^{(k)}) \hspace{-0.5mm}-\hspace{-0.5mm} f(x^\star) \le    \left(1 - \frac{\gamma \mu}{2}\right)\frac{\mathbb{E} \|\bar{\z}^{(k)}\|^2}{\gamma} - \frac{\mathbb{E}\|\bar{\z}^{(k+1)}\|^2}{\gamma}  \hspace{-0.5mm}+\hspace{-0.5mm} \frac{5L }{2n \bar{\lambda}_n} \mathbb{E}\|\check{\vz}^{(k)}\|_F^2 \hspace{-0.5mm}+\hspace{-0.5mm} \frac{\gamma \sigma^2}{n}.
	\end{align}
	If we take the uniform average for both sides over $k=0,\ldots, T$, the term $(1-\frac{\gamma \mu}{2})\frac{\mathbb{E}\|\bar{\z}^{(k)}\|^2}{\gamma}$ from the $k$th iteration cannot cancel the term  $-\frac{\mathbb{E}\|\bar{\z}^{(k)}\|^2}{\gamma}$ from the $(k+1)$th iteration. Inspired by \cite{stich2019local}, we instead take the weighted average for both sides over $k=0,\ldots,T$ so that 
	\begin{align}
	&\ \frac{1}{H_T}\sum_{k=0}^T h_k \big(\mathbb{E} f(\bar{x}^{(k)}) \hspace{-0.5mm}-\hspace{-0.5mm} f(x^\star)\big) \nonumber \\
	\le&\  \frac{1}{H_T}\sum_{k=0}^T h_k \Big( \frac{(1 - \frac{\gamma \mu}{2})\mathbb{E} \|\bar{\z}^{(k)}\|^2}{\gamma} - \frac{\mathbb{E}\|\bar{\z}^{(k+1)}\|^2}{\gamma}  \Big) + \frac{5L}{2 n H_T \bar{\lambda}_n}\sum_{k=0}^T  h_k \mathbb{E}\|\check{\vz}^{(k)}\|_F^2  + \frac{\gamma \sigma^2}{n},
	\end{align}
	where $h_k\ge 0$ is some weight to be determined, and $H_T = \sum_{k=0}^T h_k$. If we let $h_{k} = (1 - \frac{\gamma \mu}{2}) h_{k+1}$ for $k=0,1,\ldots$, the above inequality becomes
	\begin{align}\label{znb2375zzz-n1}
	\frac{1}{H_T}\sum_{k=0}^T h_k \big(\mathbb{E} f(\bar{x}^{(k)}) \hspace{-0.5mm}-\hspace{-0.5mm} f(x^\star)\big) \le \frac{h_0 \mathbb{E} \|\bar{\z}^{(0)}\|^2}{H_T \gamma} + \frac{5L}{2 n H_T \bar{\lambda}_n}\sum_{k=0}^T  h_k \mathbb{E}\|\check{\vz}^{(k)}\|_F^2  + \frac{\gamma \sigma^2}{n}.
	\end{align}
	We next bound the ergodic consensus term in the right-hand-side. 
	\begin{lemma}[\sc ergodic consensus lemma]\label{lm-sc-ergodic-consenus}
		Under Assumptions \ref{ass-W}, \ref{ass:smoothness}, \ref{ass:gradient-noise}, and \ref{ass:sc} and if learning rate satisfies $\gamma \le \frac{(1-\beta_1) \bar{\lambda}_n^{1/2}}{4L\bar{\lambda}_2} $, then it holds that 
		{
			\begin{align}\label{1bsd0925-sc-lm}
			\frac{1}{H_T}\sum_{k=0}^T h_k \mathbb{E}\|\check{\vz}^{(k)}\|_F^2  \le  \frac{4 C h_0}{H_T (1-\beta_1)} \hspace{-0.5mm}+\hspace{-0.5mm} \frac{8n\gamma^2\bar{\lambda}_2^2\sigma^2}{1-\beta_1}  \hspace{-0.5mm}+\hspace{-0.5mm} \frac{128 n\gamma^2 \bar{\lambda}_2^2 L}{(1-\beta_1)^2 H_T} \sum_{k=0}^{T} h_{k} \big(\mathbb{E} f(\bar{x}^{(k)}) \hspace{-0.5mm}-\hspace{-0.5mm} f(x^\star) \big),
			\end{align}}
		where {$C = \mathbb{E}\|\check{\vz}^{(0)}\|_F^2$}, the positive weights $\{h_k\}_{k=0}^\infty$ satisfy 
		\begin{align}\label{h-condition}
		h_k \le h_\ell \left( 1 + \frac{1-\beta_1}{4} \right)^{k-\ell} \mbox{ for any } k\ge 0 \mbox{ and } 0\le \ell \le k, 
		\end{align}
		and $H_T = \sum_{k=0}^T h_k$. (Proof is in Appendix \ref{app-sc-ergodic-consensus-lm})
		$\hfill \blacksquare$
	\end{lemma}
	With inequalities \eqref{znb2375zzz-n1} and \eqref{1bsd0925-sc-lm}, and the fact that $\beta^2_1 = \bar{\lambda}_2 = (1 + \lambda_2(W))/2 \le (1+\beta)/2$ with $\beta=\rho(W - \frac{1}{n}\mathds{1}\mathds{1}^T)$, we can establish the convergence property of D$^2$/Exact-Diffusion as follows.
	\begin{theorem}[\sc Convergence property]\label{thm2}
		Under Assumptions \ref{ass-W}, \ref{ass:smoothness}, \ref{ass:gradient-noise}, and \ref{ass:sc}, if 
		\begin{align}\label{gamma-sc}
		\gamma = \min\left\{\frac{1}{4L}, \frac{1-\beta_1}{26L} \Big(\frac{\bar{\lambda}_n^{1/2}}{\bar{\lambda}_2}\big), \frac{2\ln(2 n \mu \mathbb{E}\|\bar{\z}^{(0)}\|^2 T^2/[\sigma^2 (1-\beta)])}{\mu T}\right\},
		\end{align}
		with $\bar{\lambda}_n$ is bounded away from zero, it holds that
		\begin{align}\label{thm2-results}
		\frac{1}{H_T}\sum_{k=0}^T h_k \big(\mathbb{E} f(\bar{x}^{(k)}) - f(x^\star) \big) =  \tilde{O}\left(
		\frac{\sigma^2}{n T} + \frac{\sigma^2}{(1-\beta) T^2} + \frac{1}{1-\beta}\exp\{-(1-\beta) T \} \right). 
		\end{align}
		where $h_k$ and $H_T$ are defined in Lemma \ref{lm-sc-ergodic-consenus}. Notation $\tilde{O}(\cdot)$ hides logarithm terms. 
		(Proof is in Appendix \ref{app-thm-sc-convergence})
	
		$\hfill \blacksquare$
	\end{theorem}
	
	\begin{corollary}[\sc Transient stage]\label{coro-ts-sc}
		Under assumptions in Theorem \ref{thm2}, the transient stage for D$^2$/Exact-Diffusion in the strongly convex scenario is on the order of  $\tilde{\Omega}(\frac{n}{1-\beta})$. 
	\end{corollary}
	\begin{proof} The third term \eqref{thm2-results} decays exponentially fast and hence can be ignored compared to the first two terms. To reach the linear speedup, it is enough to set 
		\begin{align}
		\frac{\sigma^2}{(1-\beta)T^2} \lesssim \frac{\sigma^2}{nT}, \quad \mbox{which amounts to} \quad T \gtrsim \frac{n}{1-\beta}.
		\end{align}
		We use $\tilde{\Omega}(\cdot)$ rather than $\Omega(\cdot)$ because some logarithm factors are hidden inside.
	\end{proof}
	\begin{remark} It is established in \cite{koloskova2020unified,pu2019sharp} that the transient stage of D-SGD for the strongly-convex scenario is on the order of $\Omega(\frac{n}{(1-\beta)^2})$. By removing the influence of the data heterogeneity, D$^2$/Exact-Diffusion improves the transient stage to $\tilde{\Omega}(\frac{n}{1-\beta})$ which has an improved dependence on $1-\beta$. This improved transient stage is consistent with those established in parallel works \cite{huang2021improve,koloskova2021improved}.
	$\hfill \square$
	\end{remark}

	\subsection{Transient Stage Lower Bound of the Homogeneous D-SGD}
	In Sec.~\ref{sec:strongly-convex}, we have shown that D$^2$/Exact-Diffusion, by removing the influence of data heterogeneity, can improve the transient stage of D-SGD from 
	$\Omega(\frac{n}{(1-\beta)^2})$ to $\Omega(\frac{n}{1-\beta})$. In this section, we ask what is the optimal transient stage of D-SGD  if the data distributions are homogeneous (\textit{i.e., there is no influence of data heterogeneity})? Can D-SGD have a better network topology dependence than
	D$^2$/Exact-Diffusion in certain scenarios? The answer reveals that D-SGD dependence on the network topology can match D$^2$/Exact-Diffusion only under the homogeneous setting and always worse in heterogeneous setting. {\em In any cases, D-SGD cannot be more robust to network topology than 
	D$^2$/Exact-Diffusion.}

	To this end, we let $\cF_{\mu,L}=\{f: f \text{ is $\mu$-strongly convex}\text{, $L$-smooth with }\nabla f(x^\star)=0\}$ with some fixed global $x^\star$, $\cO_{\sigma^2}$ be all possible  gradient oracles with $\sigma^2$-bounded noise, $\cW_\beta=\{W:W \text{ satisfies Assumption \ref{ass-W}}$ and $ \rho(W-\frac{1}{n}\mathds{1}_n\mathds{1}_n^T)\leq \beta\}$,  and $\cA$ consists of D-SGD algorithms with all possible hyper-parameter choices (such as learning rate $\gamma$) {but with homogeneous dataset}, then  we consider the following minimax lower bound for D-SGD:
	\begin{equation}
	    T^{\rm dsgd}_{\text{trans}}=\min\limits_{A\in \cA}\max_{W\in\cW_\beta}\max_{\cO_{\sigma^2}}\max_{f_i\in\cF_{\mu,L}}\{\text{Transient time of (A)}\}.
	\end{equation}
	This definition implies that D-SGD cannot be associated with a shorter transient stage than $T^\text{dsgd}_\text{trans}$  without further assumptions. Note that this lower bound only applies to  vanilla D-SGD with  single-round gossip communication. There might be algorithms (e.g., D-SGD with multi-round gossip communications per update) that enjoy provably shorter transient time.
	
		\begin{theorem}[Lower Bound] \label{thm-lower-bound}
		The transient time for D-SGD in the homogeneous scenario, i.e., $b^2=0$, is lower bounded by 
		\begin{equation}
		T^{\rm dsgd}_{\rm trans}=\tilde{\Omega}\left(\frac{n}{1-\beta}\right).
		\end{equation}
		(Proof is in Appendix \ref{app-lower-bound-prop})
		$\hfill \blacksquare$
	\end{theorem}
	
	From Corollary \ref{coro-ts-sc} and Theorem \ref{thm-lower-bound}, we observe that the transient stage of D$^2$/Exact-Diffusion in the strongly-convex scenario coincides with the lower bound of homogeneous D-SGD in terms of the dependence on network topology (\textit{i.e.}, the influence of $\beta$) and network size $n$.
	This implies that  D-SGD has the same transient stage as D$^2$/Exact-Diffusion under the impractical homogeneous case and worse dependence in the heterogeneous case \cite{koloskova2020unified,pu2019sharp}. Hence, the dependence of D$^2$/Exact-Diffusion on network topology is no worse than D-SGD and always better under the practical heterogeneous case.
	
		\section{D$^2$/Exact-Diffusion with Multi-Round Gossip}
		In this section, we will show that the utilization of multi-round gossip communication in D$^2$/Exact-Diffusion can further improve the dependence on network topology. Motivated by \cite{lu2021optimal}, we propose the multi-step D$^2$/Exact-Diffusion described in Algorithm \ref{Algorithm: ED_multiple_gossip}. There are two fundamental differences between Algorithm \ref{Algorithm: ED_multiple_gossip} and the (vanilla) D$^2$/Exact-Diffusion in Algorithm \ref{Algorithm: D2}: {\em gradient accumulation} and {\em fast gossip averaging}. The details in the fast gossip averaging, which is inspired by \cite{liu2011accelerated}, are listed in Algorithm \ref{Algorithm: fast-gossip}. Note that Algorithm  \ref{Algorithm: fast-gossip} has a damping (or interpolation) step in the output, which is critical to guarantee the convergence for D$^2$/Exact-Diffusion with multi-round gossip communication. 
		
		\begin{algorithm}[h]
			\caption{D$^2$/Exact-Diffusion with multiple gossip steps}
			\label{Algorithm: ED_multiple_gossip}
			
			\textbf{Require:} Initialize $x_i^{(0)}=0$, $\psi_i^{(0)} = x_i^{(0)}$, the rounds of gossip steps $R$, and damping ratio $\tau\in(0,1)$.
			
			\vspace{1mm}
			\textbf{for} $k=0,1,2,...$, every node $i$ \textbf{do}
			
			\vspace{0.5mm}
			\hspace{5mm} 	\colorbox{pink}{Sample $\{\xi^{(k,r)}_{i}\}_{r=1}^R$ independently and let  $g_i^{(k)} = \frac{1}{R}\sum\limits_{r=1}^R\nabla F(x_i^{(k)};\xi_i^{(k,r)})$;} $ \triangleright  \mbox{ \footnotesize{grad. accumulation}}$ 
			
			\hspace{5mm} Update $\psi_i^{(k+1)} = x_i^{(k)} - \gamma g_i^{(k)}$; $\hspace{5 cm}  \triangleright  \mbox{ \footnotesize{local gradient descent step}}$
			
			\hspace{5mm} Update $\phi_i^{(k+1)} = \psi_i^{(k+1)} + x_i^{(k)} - \psi_i^{(k)}$; $\hspace{3.75cm}  \triangleright  \mbox{ \footnotesize{solution correction step}}$
			
			\hspace{5mm} \colorbox{pink}{Update	$x^{(k+1)}_i = \textbf{FastGossipAverage}(\{\phi^{(k+1)}_i\}_{i=1}^n,{W}, R,\tau)$;}
			$\hspace{0.33 cm}   \triangleright \mbox{ \footnotesize{multiple gossip communication}}$
		\end{algorithm}
		
		\begin{algorithm}[h]
			\caption{$x_i =  \textbf{FastGossipAverage}(\{\phi_i\}_{i=1}^n,{W}, R,\tau)$}
			\label{Algorithm: fast-gossip}
			
			\textbf{Require:} $\{\phi_i\}_{i=1}^n$, ${W}$, $R$, $\tau$; let  ${z_i^{(0)} = z_i^{(-1)}=\phi_i}$ and step size $\eta = \frac{1-\sqrt{1-\beta^2}}{1+\sqrt{1+\beta^2}}$.
			
			\vspace{1mm}
			\textbf{for} $r=0,1,2,..., R-1$, every node $i$ \textbf{do}
			
			\vspace{0.5mm}
			\hspace{5mm} Update	$z^{(r+1)}_i = (1+\eta)\sum_{j\in\mathcal{N}_i}{w}_{ij} z^{(r)}_i-\eta z_i^{(r-1)}$;
			$\hspace{2.83cm}   \triangleright \mbox{ \footnotesize{fast gossip averaging}}$
			
			\textbf{Output:}	$x_i = (1-\tau){z}^{(R)}_i +\tau z^{(0)}_i$;
			$\hspace{5.67cm}   \triangleright \mbox{ \footnotesize{damping step}}$
		\end{algorithm}
		
		\subsection{Fast Gossip Averaging}
		Using $\vz^{(r)} = [z_1^{(r)},\dots, z_n^{(r)}]^T \in \RR^{n\times d}$, the fast gossip average update (Algorithm \ref{Algorithm: fast-gossip}) can be described by 
	\begin{subequations}
	    	\begin{align}
		\vz^{(r+1)}&= (1+\eta){W}\vz^{(r)}-\eta \vz^{(r-1)},\quad \textbf{for} \,\, r=0,1,\dots,R-1\label{eqn:fast-gossip1}\\
		{\vx}& = (1-\tau){\vz}^{(R)}+\tau{\vz}^{(0)}.\label{eqn:fast-gossip2}
		\end{align}
	\end{subequations}
		Since $\vz^{(-1)}=\vz^{(0)}$, it holds from \eqref{eqn:fast-gossip1}--\eqref{eqn:fast-gossip2} that  $\vz^{(r)}=M^{(r)}\vz^{(0)}$ where $M^{(r)}\in\RR^{n\times n}$ is defined by:
		\begin{align}
		M^{(-1)}&=M^{(0)}=I\label{eqn:M1}\\
		M^{(r+1)}&= (1+\eta){W}M^{(r)}-\eta M^{(r-1)},\quad \textbf{for} \,\, r=0,1,\dots,R-1.\label{eqn:M2}
		\end{align}
		Since $W$ is symmetric and doubly stochastic (Assumption \ref{ass-W}), the matrix $M^{(r)}$ is also symmetric and doubly stochastic for each $r=0,\dots, R$. Furthermore, the following result holds.   
		\begin{proposition} \label{prop-M-property}
			Under Assumption \ref{ass-W}, it holds that
			\begin{align}
M^{(r)} = (M^{(r)})^T, \quad M^{(r)}\one =\one ,\quad \mbox{and}\quad \rho(M^{(r)}-\frac{1}{n}\one\one^T)\leq \sqrt{2}\Big(1-\sqrt{1-\beta}\Big)^r.
			\end{align}
			where $\beta =\max\{|\lambda_2(W)|,|\lambda_n(W)|\}$. Therefore, the non-unit eigenvalues of $M^{(r)}$ vanish to zero as $r\to \infty$.
			(Proof is in Appendix \ref{app:M-propert}).
			$\hfill \blacksquare$
		\end{proposition}
		\noindent When the rounds of gossip steps $R$ are  sufficiently large, we can achieve the following important proposition:
		\begin{proposition}\label{prop-lambda-mbar}
			We let $\bar{M} = (1-\tau)M^{(R)} +\tau I$. If Assumption \ref{ass-W} holds and $R = \big\lceil \frac{\ln(n)+4}{\sqrt{1-\beta}}\big\rceil$  and $\tau = \frac{1}{2n}$, then it holds that
		\begin{align}\label{znzbab09823}
			\lambda_k(\bar{M})\in \left[\tau-(1-\tau)\rho(M-\frac{1}{n}\one\one^T),~\tau + (1-\tau)\rho(M-\frac{1}{n}\one\one^T)\right]\subseteq \left[\frac{1}{4n},~\frac{3}{4n}\right],\quad \forall\, ~ 2\leq k\leq n.
			\end{align}
			In other words, for $2\leq k\leq n$, $\lambda_k(\bar{M})$ can vanish with respect to $n$  while keeping a constant ratio $\frac{\lambda_2(\bar{M})}{\lambda_n(\bar{M})}\leq 3$  (Proof is in Appendix  \ref{app-proof-mbar}). 
						$\hfill \blacksquare$
		\end{proposition}
		
		\subsection{Reformulating D$^2$/Exact-Diffusion with Multiple Gossip steps}\label{sec-reformulation-D2-mg}
		
		\noindent \textbf{Primal recursion.} With the above discussion, it holds that $\vx^{(k+1)} = \bar{M} \phi^{(k+1)}$ after the fast gossip averaging step in Algorithm \ref{Algorithm: ED_multiple_gossip}. Substituting this relation into Algorithm \ref{Algorithm: ED_multiple_gossip}, we achieve the primal recursion for D$^2$/Exact-Diffusion with multiple gossip steps: 
		\begin{align}\label{eq:d2-multiple-gossip}
		\vx^{(k+1)} = \bar{M}\Big( 2 \vx^{(k)} - \vx^{(k-1)} - \gamma \big(\vg^{(k)} - \vg^{(k-1)}\big)\Big), \quad \forall~ k=1,2,\ldots
		\end{align}
		where $\vg^{(k)} = [g_1^{(k)},\dots, g_n^{(k)}]^T \in \RR^{n\times d}$ and $g_i^{(k)}$ is achieved by the gradient accumulation step in Algorithm \ref{Algorithm: ED_multiple_gossip}, $\bar{M} = (1-\tau)M^{(R)} +\tau I$ and $M^{(R)}$ is achieved by recursions \eqref{eqn:fast-gossip1} and \eqref{eqn:fast-gossip2}. The spectral properties of $\bar{M}$ is given in \eqref{znzbab09823}. 
		
		\noindent \textbf{Primal-dual recursion.} The primal recursion in \eqref{eq:d2-multiple-gossip} is equivalent to the following primal-dual updates 
		\begin{align}
		\begin{cases}\label{eq:d2-pd-mg}
		\vx^{(k+1)} = \bar{M}\big( \vx^{(k)} - \gamma \vg^{(k)} \big) - V \vy^{(k)}, \\
		\vy^{(k+1)} = \vy^{(k)} + \bar{V} \vx^{(k+1)},\quad \forall~ k = 0, 1,2,\ldots
		\end{cases}
		\end{align}
		where $\bar{V}=(I - \bar{M})^{\frac{1}{2}}$. 		Recursions \eqref{eq:d2-multiple-gossip} and \eqref{eq:d2-pd-mg} have two differences from the vanilla D$^2$/Exact-Diffusion recursions \eqref{eq:d2} and \eqref{eq:d2-pd}. First, the weight matrix $\bar{W}$ is replaced by $\bar{M}$. Second, the gradient $\vg^{(k)}$ is achieved via gradient accumulation. This implies that the convergence analysis of D$^2$/Exact-Diffusion with multi-round gossip communication can follow that of vanilla D$^2$/Exact-Diffusion. We only need to pay attentions to the influence of $\bar{M}$ obtained by multi-round gossip steps and the $\vg^{(k)}$ achieved by gradient accumulation.
		
		\subsection{Convergence Rate and Transient Stage}
		The following theorem establishes the convergence property of Algorithm \ref{Algorithm: ED_multiple_gossip} under general convexity. 
		
		\begin{theorem}[\sc Convergence under general convexity]\label{thm-generally-convex-mg}
			With Assumptions \ref{ass-W}-\ref{ass:gradient-noise}, $R = \big\lceil \frac{\ln(n)+4}{\sqrt{1-\beta}}\big\rceil$, and learning rate
			\begin{align}\label{znnbzzbbzz0000000-theory}
			\gamma = \min\left\{\frac{1}{4L},\frac{(1-\tilde{\beta}) \tilde{\lambda}_n^{\frac{1}{2}}}{10 L \tilde{\lambda}_2}, \Big(\frac{\tilde{r}_0}{\tilde{r}_1(K+1)}\Big)^{\frac{1}{2}}, \Big(\frac{\tilde{r}_0}{\tilde{r}_2(K+1)}\Big)^{\frac{1}{3}},\left(\frac{\tilde{r}_0}{\tilde{r}_3}\right)^\frac{1}{3}\right\},
			\end{align}
			where $\tilde{r}_0, \tilde{r}_1$, $\tilde{r}_2$ and $\tilde{r}_3$ are constants defined in \eqref{kzunb},  $\tilde{\beta}$, $\tilde{\lambda}_n$ and $\tilde{\lambda}_2$ are constants defined in \eqref{xbsd87-0-mg}, and $K$ is the number of outer loop, Algorithm \ref{Algorithm: ED_multiple_gossip}  converges at
			\begin{align}\label{xbsd87-mg}
			\frac{1}{K+1}\sum_{k=0}^K\big(\mathbb{E}f(\bar{x}^{(k)}) - f(x^\star)\big) = {O}\Big( \frac{\sigma}{\sqrt{nT}} + 
			\frac{\sigma^{\frac{2}{3}}\ln(n)^\frac{1}{3}}{n^{\frac{1}{3}}(1-\beta)^{\frac{1}{6}}T^{\frac{2}{3}}} +  \frac{\ln(n)}{(1-\beta)^{\frac{1}{2}}T}\Big).
			\end{align}
			where $T = KR$ is the total number of sampled data (or gossip communications) (Proof is in Appendix \ref{app-proof-mg-convex}). 
			
			$\hfill \blacksquare$
		\end{theorem}
		
		\begin{corollary}[\sc Transient stage under general convexity]
			Under the same assumptions as in Theorem \ref{thm-generally-convex-mg}, the transient stage for multi-step D$^2$/Exact-Diffusion is on the order of $\Omega\left(\frac{n\ln(n)^2}{{1-\beta}}\right)=\tilde{\Omega}\left(\frac{n}{{1-\beta}}\right)$. 
		\end{corollary}
		\begin{proof}
			To achieve the linear speedup stage, $T$ has to be large enough such that 
			\begin{align}
			\frac{\sigma^{\frac{2}{3}}\ln(n)^\frac{1}{3}}{n^{\frac{1}{3}}(1-\beta)^{\frac{1}{6}}T^{\frac{2}{3}}} \lesssim \frac{\sigma}{\sqrt{nT}},\quad \quad  \frac{\ln(n)}{(1-\beta)^{\frac{1}{2}}T} \lesssim \frac{\sigma}{\sqrt{nT}}
			\end{align}
			which is equivalent to 
			\begin{align}
			T =  \Omega\left(\frac{n[\ln(n)]^2}{\sigma^2(1-\beta)}\right)
	=\tilde{\Omega}\left(\frac{n}{{1-\beta}}\right).
			\end{align}
		\end{proof}
		The following theorem establishes the convergence performance of Algorithm \ref{Algorithm: ED_multiple_gossip} with strong convexity. 
		
		\begin{theorem}[\sc Convergence under strong  convexity]\label{thm2-mg}With Assumptions \ref{ass-W}, \ref{ass:smoothness}, \ref{ass:gradient-noise}, \ref{ass:sc} and $R = \big\lceil \frac{\ln(n)+4}{\sqrt{1-\beta}} \big\rceil$, if the learning rate satisfies 
			\begin{align}\label{gamma-sc-mg}
			\gamma = \min\left\{ \frac{1}{4L},\frac{1-\tilde{\beta}}{26 L} \Big(\frac{\tilde{\lambda}_n^{1/2}}{\tilde{\lambda}_2}\Big), \frac{2\ln(2 n \mu \mathbb{E}\|\bar{\z}^{(0)}\|^2 K^2/[\tilde{\sigma}^2 (1-\tilde{\beta})])}{\mu K}\right\},
			\end{align}
			where $\tilde{\sigma}$, $\tilde{\beta}$, $\tilde{\lambda}_n$ and $\tilde{\lambda}_2$ are constants defined in \eqref{xbsd87-0-mg}, Algorithm \ref{Algorithm: ED_multiple_gossip}  converges at
			\begin{align}\label{thm2-results-mg}
			\frac{1}{H_K}\sum_{k=0}^K h_k \big(\mathbb{E} f(\bar{x}^{(k)}) - f(x^\star) \big) =  \tilde{O}\left(\frac{\sigma^2}{n T} + \frac{\sigma^2}{n (1-\beta)^\frac{1}{2}T^2} + 
			\exp\{-(1-\beta)^\frac{1}{2} T \}  \right). 
			\end{align}
			where $h_k$ and $H_K$ are defined in Lemma \ref{lm-sc-ergodic-consenus}. (Proof is in Appendix \ref{app-proof-thm2-mg}). 
			$\hfill \blacksquare$
		\end{theorem}
		
		\begin{corollary}[\sc Transient stage]\label{coro-ts-sc-mg}
			Under the same assumptions as Theorem \ref{thm2}, the transient stage for multi-step D$^2$/Exact-Diffusion in the strongly convex scenario is on the order of  $\tilde{\Omega}\big((1-\beta)^{-\frac{1}{2}}\big)$. 
		\end{corollary}
		\begin{proof} The third term \eqref{thm2-results} decays exponentially fast and hence can be ignored compared to the first two terms as long as $(1-\beta)^\frac{1}{2}T=\tilde{\Omega}(1)$, i.e., $T=\tilde{\Omega}((1-\beta)^{-\frac{1}{2}})$, otherwise the exponential term remains  a constant. To reach the linear speedup, it is enough to set $\frac{\sigma^2}{n(1-\beta)^\frac{1}{2}T^2} \lesssim \frac{\sigma^2}{nT}$.
			which amounts to $T = \tilde{\Omega} \big((1-\beta)^{-\frac{1}{2}}\big)$. We use $\tilde{\Omega}(\cdot)$ rather than $\Omega(\cdot)$ because  logarithm factors are hidden inside. 
		\end{proof}
	\begin{remark} \label{remark-weakest-dependence}
	In Corollary \ref{coro-ts-sc-mg}, the transient stage of D$^2$/Exact-Diffusion with multi-round gossip communication has a significantly better (i.e., weaker) dependence on network topology connectivity $1-\beta$ and network size $n$ compared to existing works \cite{koloskova2020unified,pu2019sharp,pu2020distributed,huang2021improve,koloskova2021improved}, see Table \ref{table-transient-stage-local}.
	$\hfill \square$
	\end{remark}
	
	
	\section{Numerical Simulation}
	In this section, we validate the established theoretical results with numerical simulations. 
	
	\subsection{Strongly-Convex Scenario}
	\noindent \textbf{Problem.} We consider the following decentralized least-square problem
    \begin{align}\label{zbnzbba978}
        \min_{x \in \RR^d} \quad \frac{1}{2n}\sum_{i=1}^n  \|A_i x - b_i\|^2
    \end{align}
    where $A_i \in \RR^{M \times d}$ is the coefficient matrix, and $b_i \in \RR^d$ is the measurement. Quantities $A_i$ and $b_i$ are associated with node $i$, and $M$ is the size of local dataset. 
    
    \noindent \textbf{Simulation settings.} In our simulations, we set $d=10$ and $M=1000$. To control the data heterogeneity across the nodes, we first let each node $i$ be associated with a local solution $x^\star_{i}$, and such $x^\star_i$ is generated by $x^\star_i = x^\star + v_i$ where $x^\star\sim \cN(0, I_d)$ is a randomly generated vector while $v_i \sim \cN(0, \sigma^2_h I_d)$ controls the similarity between each local solution. Generally speaking, a large $\sigma^2_h$ results in local solutions  $\{x_i^\star\}$ that are vastly different from each other. With $x_i^\star$ at hand, we can generate local data that follows distinct distributions. At node $i$, we generate each element in $A_i$ following standard normal distribution. Measurement $b_i$ is generated by $b_i = A_i x_i^\star + s_i$ where $s_i \sim \cN(0, \sigma_s^2 I)$ is some white noise. Clearly, solution $x_i^\star$ controls the distribution of the measurements $b$. In this way, we can easily control data heterogeneity by adjusting $\sigma^2_h$. At each iteration $k$, each node will randomly sample a row in $A_i$ and the corresponding element in $b_i$ and use them to evaluate the stochastic gradient. The metric for all simulations in this subsection is $\frac{1}{n}\sum_{i=1}^n\|x_i^{(k)} - x^\star\|^2$ where $x^\star$ is the global simulation to problem \eqref{zbnzbba978} and it has a closed-form $x^\star = (\sum_{i=1}^n A_i^T A_i)^{-1} (\sum_{i=1}^n A_i^T A_i b_i)$. 
    
    \begin{figure}[t!]
	\centering
	\includegraphics[width=0.4\textwidth]{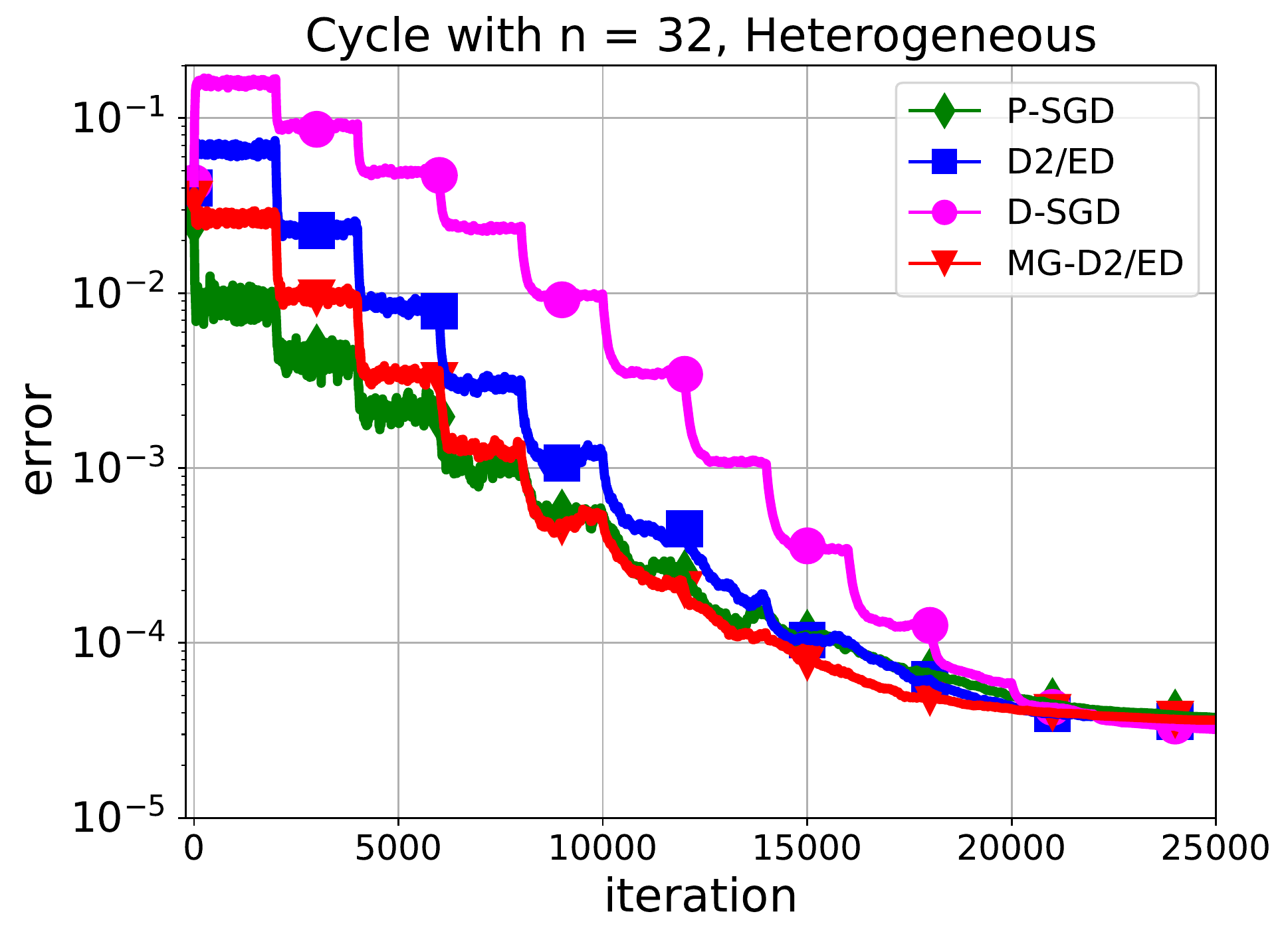}
	\includegraphics[width=0.4\textwidth]{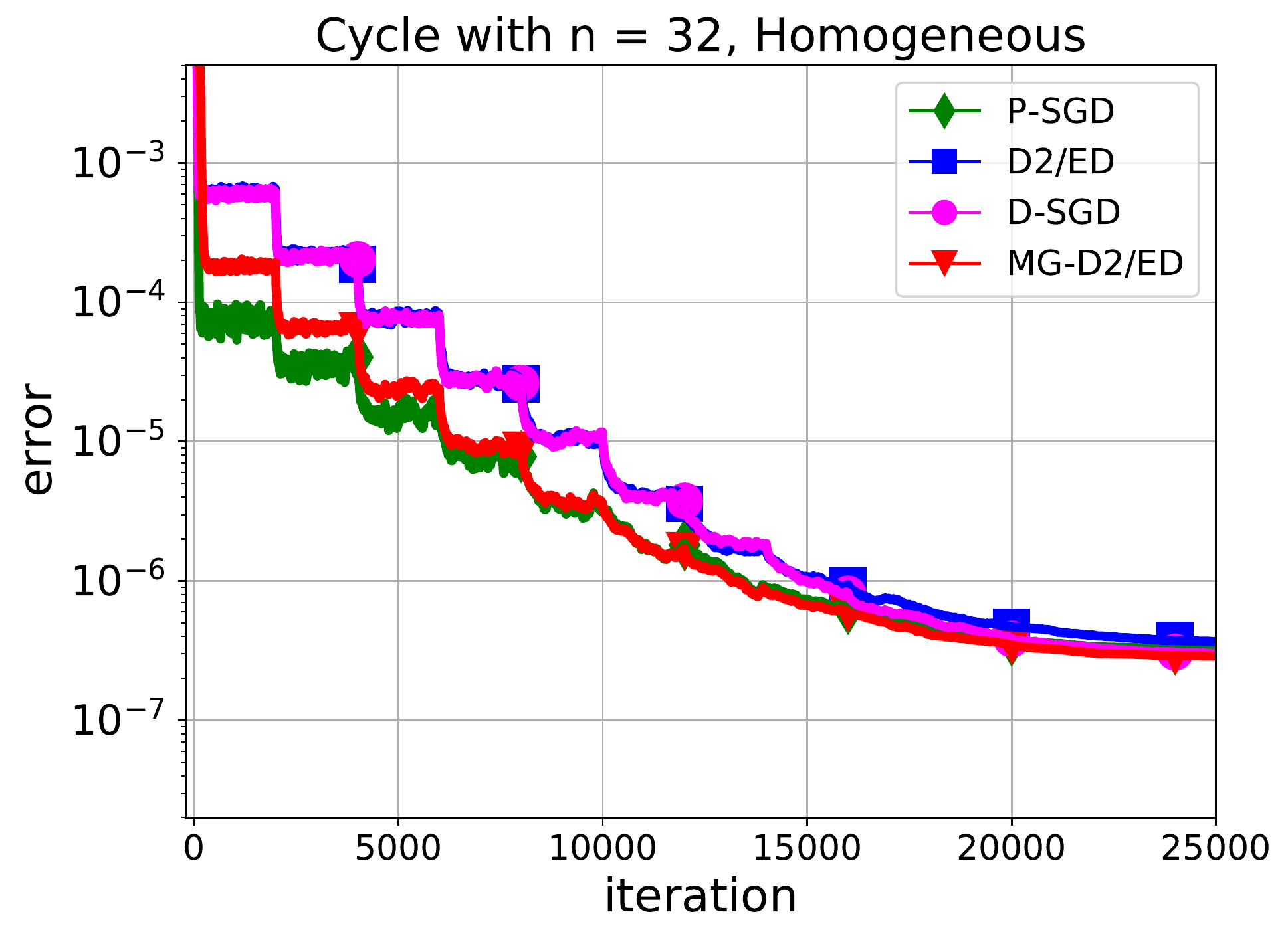}
	\caption{\small Performance of different stochastic algorithms to solve problem \eqref{zbnzbba978}. The left plot is with heterogeneous data while the right is with homogeneous data.
}
	\label{fig:ls-strongly-convex} 
    \end{figure}
    
    \noindent 
    \textbf{Performance with heterogeneous data.} We now compare the convergence performance of Parallel SGD (P-SGD), Decentralized SGD (D-SGD), D$^2$/Exact-Diffusion (D2/ED), and D$^2$/Exact-Diffusion with multi-round gossip communication (MG-D2/ED) when data heterogeneity exists. The target is to examine their robustness to the influence of network topology. To this end, we let $\sigma_h^2 = 0.2$ and organize $n=32$ nodes into a cycle. The left plot in Fig.  \ref{fig:ls-strongly-convex} lists the performances of all algorithms. Each algorithm utilizes the same learning rate which decays by half for every 2,000 gossip communications. In this plot, it is observed that all decentralized algorithms, after certain amounts of transient iterations, can match with P-SGD asymptotically. In addition, we find D-SGD is least robust while MG-D2/ED is most robust to network topology, which aligns with the theoretically established bounds for transient stage in Table \ref{table-transient-stage-local}. 
    

    \noindent 
    \textbf{Performance with homogeneous data.} We next compare there algorithms with homogeneous data. To this end, we let $\sigma_h^2 = 0$ and organize $n=32$ nodes into a cycle. The other settings are the same as in the heterogeneous scenario discussed in the above. The right plot in Fig. \ref{fig:ls-strongly-convex} lists the performances of all algorithms. It is observed that D-SGD and D2/ED have almost the same convergence behaviours, which validates the conclusion in Theorem \ref{thm-lower-bound} that D-SGD can match with D2/ED in the homogeneous data scenario. In addition, we find MG-D2/ED requires less transient iterations than D-SGD and D2/ED to match with P-SGD, indicating that it is more robust to network topology even if in the homogeneous data scenario.

	\begin{figure}[t!]
	\centering
	\includegraphics[width=0.4\textwidth]{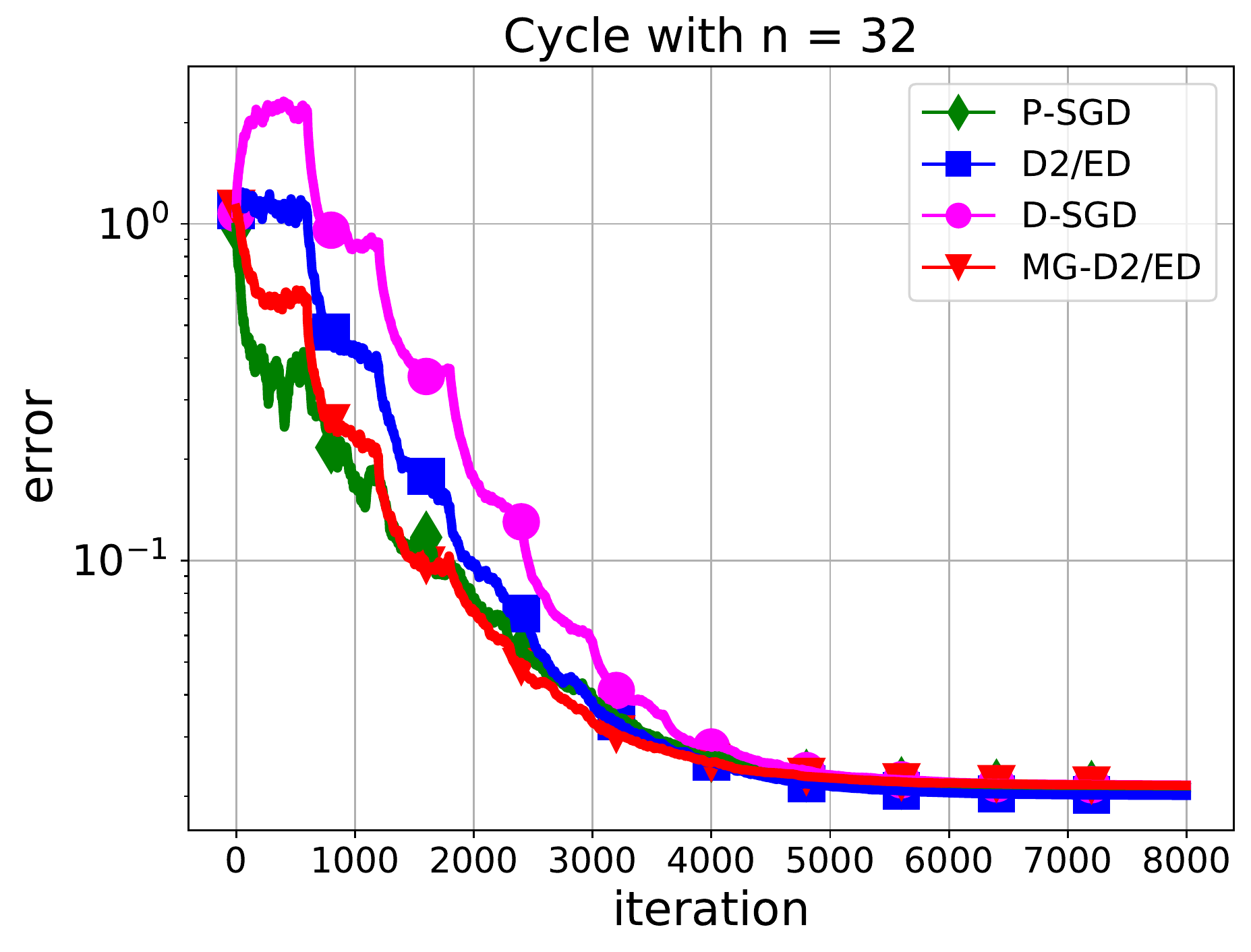}
	\includegraphics[width=0.4\textwidth]{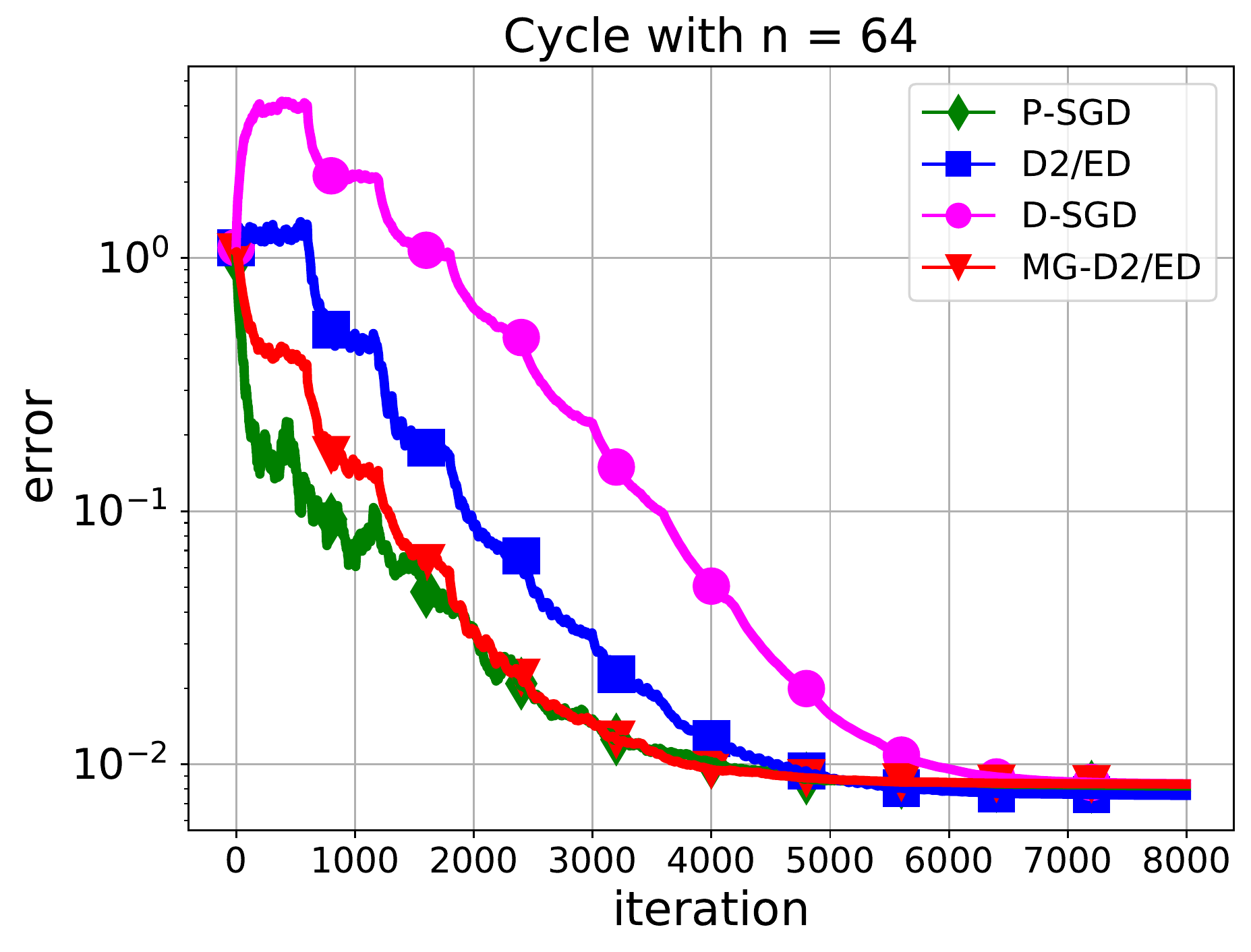}
	\caption{\small Performance of different stochastic algorithms to solve problem \eqref{zbnzbba978-lr}.
	}
	\label{fig:lr-generally-convex}
    \end{figure}
	
	\subsection{Generally-Convex Scenario}
	
	\noindent \textbf{Problem.} We consider the following decentralized logistic regression problem
    \begin{align}\label{zbnzbba978-lr}
        \min_{x \in \RR^d} \quad \frac{1}{n}\sum_{i=1}^n  f_i(x) \quad \quad \mbox{where} \quad \quad f_i(x) = \frac{1}{M}\sum_{m=1}^M \ln\big(1 + \exp(-y_{i,m}h_{i,m}\tran x)\big)
    \end{align}
    where $\{h_{i,m}, y_{i,m}\}_{m=1}^M$ is the training dateset held by node $i$ in which $h_{i,m}\in \mathbb{R}^d$ is a feature vector while $y_{i,m} \in \{-1,+1\}$ is the corresponding label. 
    
    \noindent \textbf{Simulation settings.} Similar to the strongly-convex scenario, each node $i$ is associated with a local solution $x_i^\star$. To generate local dataset $\{h_{i,m}, y_{i,m}\}_{m=1}^M$, we first generate each feature vector $h_{i,m} \sim \cN(0, I_d)$. We label  $y_{i, m} = 1$ with probability $1/(1+ \exp(-y_{i,\ell}h_{i, m}\tran x_i^\star))$; otherwise $y_{i, m} = -1$.
We can  control data heterogeneity by adjusting $\sigma^2_h$. 

    \noindent \textbf{Robustness to network topology.} Fig. \ref{fig:lr-generally-convex} lists the performances of all stochastic algorithms with different network sizes. When size of the cycle graph increases from $32$ to $64$ (the quantity $1/(1-\beta)$ increases from $78.07$ to $311.51$), it is observed that the performance of D-SGD is significantly deteriorated. In contrast, such change of the network topology just influences D2/ED slightly. Furthermore, it is observed that MG-D2/ED is always very close to P-SGD no matter the topology is well-connected or not. These phenomenons are consistent with the established transient stage for the generally-convex scenario in Table \ref{table-transient-stage-local}. 
	
	\begin{figure}[t!]
	\centering
	\includegraphics[width=0.4\textwidth]{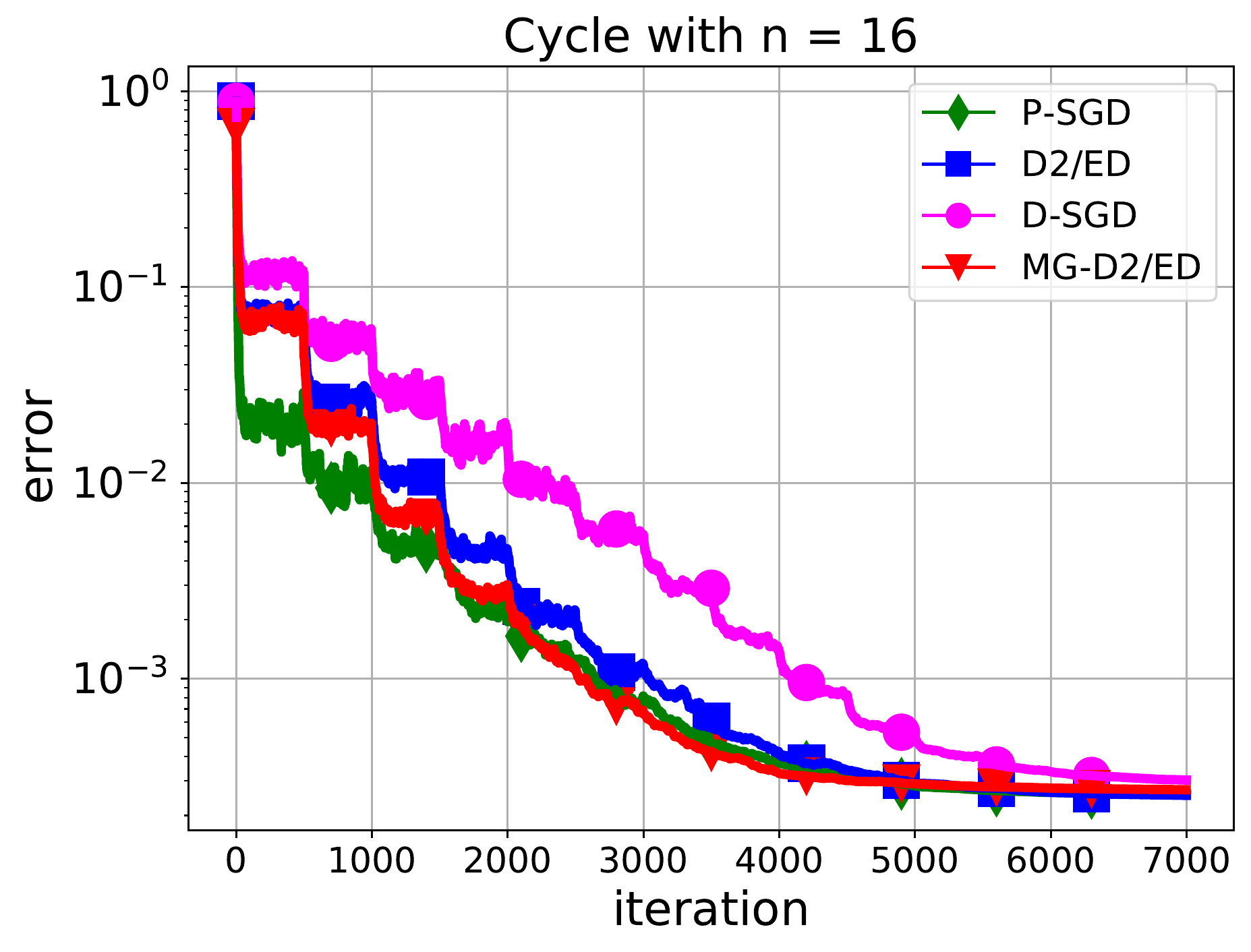}
	\includegraphics[width=0.4\textwidth]{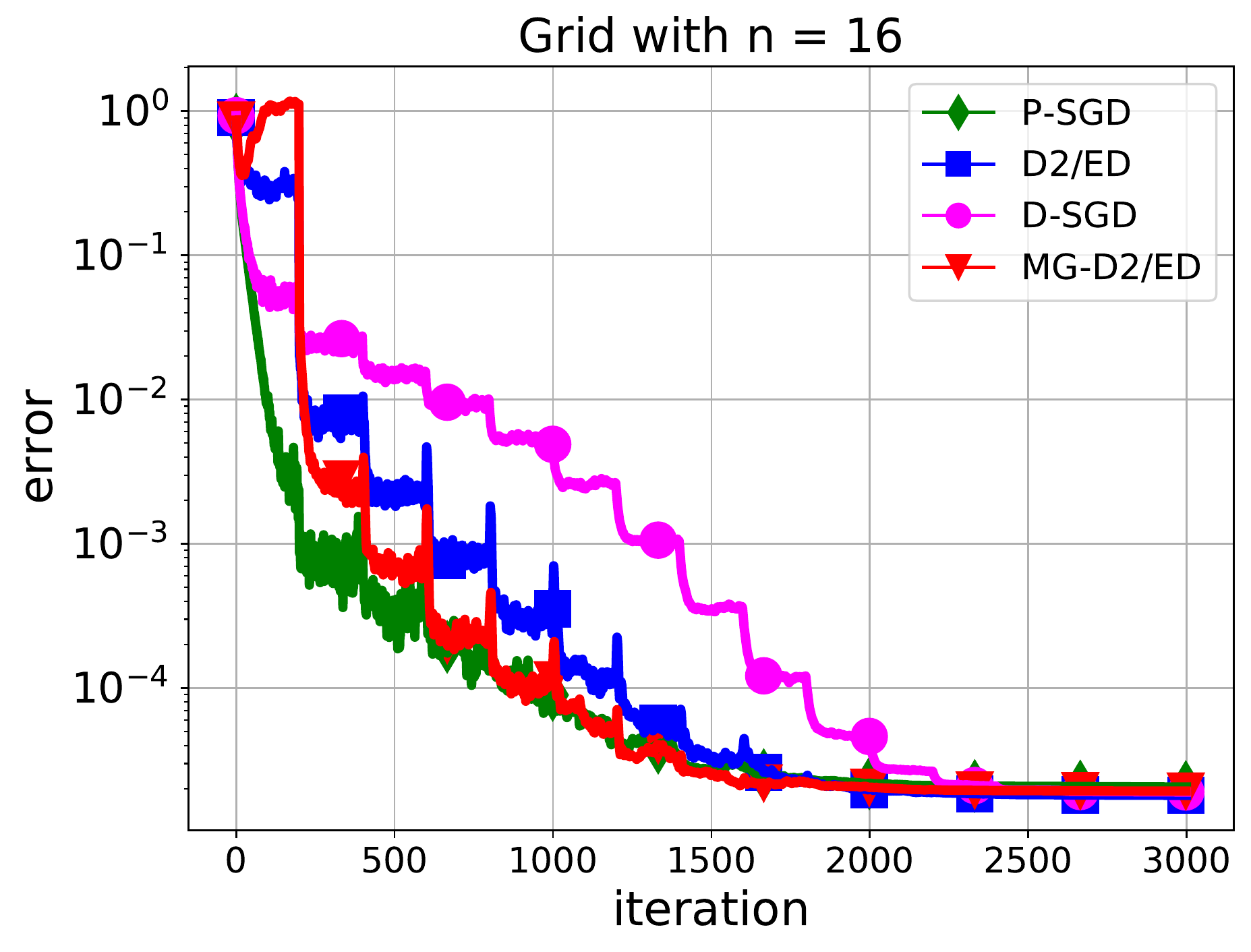}
	\caption{\small Performance of different stochastic algorithms to solve problem \eqref{real-dataset} with different real datasets and topologies. The left plot is with MNIST dataset, and the right is with COVTYPE dataset.
	}
	\label{fig:real-dataset}
    \end{figure}
	
	\subsection{Simulation with Real Datasets}
	This subsection examines the performances of P-SGD, D-SGD, D2/ED, and MG-D2/ED with real datasets. We run experiments for the
regularized logistic regression problem with 
\begin{align}\label{real-dataset}
f_i(x) = \frac{1}{M}\sum_{m=1}^M \ln\big(1 + \exp(-y_{i,m}h_{i,m}\tran x)\big) + \frac{\rho}{2}\|x\|^2
\end{align}
where $\rho > 0$ is a positive constant. We consider two real datasets: MNIST \cite{deng2012mnist} and COVTYPE.binary  \cite{rossi2015network}. The MNIST recognition task has been transformed into a binary classification problem by  considering data with labels 2 and
4. In COVTYPE.binary, we use 50, 000
samples as training data and each data has dimension 54. In MNIST we use 10, 000 samples as
training data and each data has dimension 784. The regularization coefficient $\rho=0.001$ for all simulations.  To promote data heterogeneity, we control the ratio  of the sizes of positive and negative samples for each node. More specifically, in COVTYPE.binary, half of the nodes maintain $54\%$ positive samples while the other half maintain $54\%$  negative samples. Likewise, the ratio is fixed as $70\%$ for MNIST dataset, which is a bit larger since the heterogeneity between all grey-scale handwritten digits of $2$ and $4$ is relatively weak. Except for the fixed ratio of the positive samples to negative ones, all training data are distributed uniformly to each local node.  The left plot in Fig.~\ref{fig:real-dataset} illustrates the performance of various algorithms with MNIST dataset over the cycle graph while the right is with COVTYPE dataset over the grid graph. In all simulations, we find the transient stage as well as the robustness to network topology coincides those established in Table \ref{table-transient-stage-local} well. D2/ED always converges better than D-SGD, and MG-D2/ED is least sensitive to network topology compared to D-SGD and D2/ED. 
	
	\section{Conclusion and Discussion} 
	In this work, we revisited the D$^2$/Exact-Diffusion algorithm \cite{tang2018d,yuan2017exact1,yuan2018exact2,li2019decentralized,yuan2020influence} and studied  its non-asymptotic convergence rate  under both the  generally-convex and strongly-convex settings. By removing the influence of data heterogeneity, D$^2$/Exact-Diffusion is shown to improve the transient stage of D-SGD from $\Omega(n^3/(1-\beta)^4)$ to $\Omega(n^3/(1-\beta)^2)$ and from $\Omega(n/(1-\beta)^2)$ to $\tilde{\Omega}(n/(1-\beta))$ for the generally convex and strongly-convex settings, respectively. This result shows that D$^2$/Exact-Diffusion \cite{tang2018d,yuan2017exact1,yuan2018exact2,li2019decentralized,yuan2020influence} is less sensitive to the network topology. For the strongly-convex scenario, we also proved that our transient stage bound coincides with the lower bound of homogeneous D-SGD in terms of {\em network topology dependence}, which implies that D$^2$/Exact-Diffusion cannot have worse network dependence than D-SGD and has a better dependence in the heterogeneous setting. Moreover, when D$^2$/Exact-diffusion is equipped with gradient accumulation and multi-round gossip communications, its transient stage can be further improved to $\tilde{\Omega}(1/(1-\beta)^{\frac{1}{2}})$ and $\tilde{\Omega}(n/(1-\beta))$ for strongly and generally convex cost functions, respectively.

	There are still several open questions to answer for the family of data-heterogeneity-corrected methods such as EXTRA, D$^2$/Exact-Diffusion, and gradient-tracking. First, it is unclear whether these methods can still have improved dependence on network topology over time-varying topologies. Second, while data-heterogeneity-corrected methods are endowed with superior convergence properties in terms of robustness to heterogeneous data or network topology dependence, D-SGD can still {\em empirically} outperforms them in deep learning applications, see \cite{lin2021quasi,yuan2021decentlam}. Great efforts may still be needed to fill in the gap between theory and real implementation.
	
	
	\bibliographystyle{ieeetr}
	\bibliography{reference}

\begin{thebibliography}{10}

\bibitem{zinkevich2010parallelized}
M.~Zinkevich, M.~Weimer, L.~Li, and A.~J. Smola, ``Parallelized stochastic
  gradient descent,'' in {\em Advances in neural information processing
  systems}, pp.~2595--2603, 2010.

\bibitem{smola2010architecture}
A.~Smola and S.~Narayanamurthy, ``An architecture for parallel topic models,''
  {\em Proceedings of the VLDB Endowment}, vol.~3, no.~1-2, pp.~703--710, 2010.

\bibitem{li2014scaling}
M.~Li, D.~G. Andersen, J.~W. Park, A.~J. Smola, A.~Ahmed, V.~Josifovski,
  J.~Long, E.~J. Shekita, and B.-Y. Su, ``Scaling distributed machine learning
  with the parameter server,'' in {\em 11th $\{$USENIX$\}$ Symposium on
  Operating Systems Design and Implementation ($\{$OSDI$\}$ 14)}, pp.~583--598,
  2014.

\bibitem{ring-allreduce}
A.~Gibiansky, ``Bringing {HPC} techniques to deep learning.''
  \url{https://andrew.gibiansky.com/blog/machine-learning/baidu-allreduce/},
  2017.
\newblock Accessed: 2020-08-12.

\bibitem{ying2021bluefog}
B.~Ying, K.~Yuan, H.~Hu, Y.~Chen, and W.~Yin, ``Bluefog: Make decentralized
  algorithms practical for optimization and deep learning,'' {\em arXiv
  preprint arXiv:2111.04287}, 2021.

\bibitem{lopes2008diffusion}
C.~G. Lopes and A.~H. Sayed, ``Diffusion least-mean squares over adaptive
  networks: Formulation and performance analysis,'' {\em IEEE Transactions on
  Signal Processing}, vol.~56, no.~7, pp.~3122--3136, 2008.

\bibitem{nedic2009distributed}
A.~Nedic and A.~Ozdaglar, ``Distributed subgradient methods for multi-agent
  optimization,'' {\em IEEE Transactions on Automatic Control}, vol.~54, no.~1,
  pp.~48--61, 2009.

\bibitem{chen2012diffusion}
J.~Chen and A.~H. Sayed, ``Diffusion adaptation strategies for distributed
  optimization and learning over networks,'' {\em IEEE Transactions on Signal
  Processing}, vol.~60, no.~8, pp.~4289--4305, 2012.

\bibitem{lian2017can}
X.~Lian, C.~Zhang, H.~Zhang, C.-J. Hsieh, W.~Zhang, and J.~Liu, ``Can
  decentralized algorithms outperform centralized algorithms? a case study for
  decentralized parallel stochastic gradient descent,'' in {\em Advances in
  Neural Information Processing Systems}, pp.~5330--5340, 2017.

\bibitem{assran2019stochastic}
M.~Assran, N.~Loizou, N.~Ballas, and M.~Rabbat, ``Stochastic gradient push for
  distributed deep learning,'' in {\em International Conference on Machine
  Learning (ICML)}, pp.~344--353, 2019.

\bibitem{lian2018asynchronous}
X.~Lian, W.~Zhang, C.~Zhang, and J.~Liu, ``Asynchronous decentralized parallel
  stochastic gradient descent,'' in {\em International Conference on Machine
  Learning}, pp.~3043--3052, 2018.

\bibitem{yuan2021decentlam}
K.~Yuan, Y.~Chen, X.~Huang, Y.~Zhang, P.~Pan, Y.~Xu, and W.~Yin, ``{DecentLaM}:
  Decentralized momentum sgd for large-batch deep training,'' pp.~3029--3039,
  2021.

\bibitem{ying2021exponential}
B.~Ying, K.~Yuan, Y.~Chen, H.~Hu, P.~Pan, and W.~Yin, ``Exponential graph is
  provably efficient for decentralized deep training,'' in {\em Advances in
  Neural Information Processing Systems (NeurIPS)}, 2021.

\bibitem{bluefog2021_4616052}
B.~Ying, K.~Yuan, H.~Hu, Y.~Chen, and W.~Yin, ``{BlueFog}: Make decentralized
  algorithms practical for optimization and deep learning.''
  \url{https://github.com/Bluefog-Lib/bluefog}, 2021.
\newblock Accessed: 2021-05-15.

\bibitem{chen2021accelerating}
Y.~Chen, K.~Yuan, Y.~Zhang, P.~Pan, Y.~Xu, and W.~Yin, ``Accelerating gossip
  sgd with periodic global averaging,'' in {\em International Conference on
  Machine Learning (ICML)}, 2021.

\bibitem{koloskova2020unified}
A.~Koloskova, N.~Loizou, S.~Boreiri, M.~Jaggi, and S.~U. Stich, ``A unified
  theory of decentralized sgd with changing topology and local updates,'' in
  {\em International Conference on Machine Learning (ICML)}, pp.~1--12, 2020.

\bibitem{pu2019sharp}
S.~Pu, A.~Olshevsky, and I.~C. Paschalidis, ``A sharp estimate on the transient
  time of distributed stochastic gradient descent,'' {\em IEEE Transactions On
  Automatic Control, early access}, 2021.

\bibitem{yuan2020influence}
K.~Yuan, S.~A. Alghunaim, B.~Ying, and A.~H. Sayed, ``On the influence of
  bias-correction on distributed stochastic optimization,'' {\em IEEE
  Transactions on Signal Processing}, vol.~68, pp.~4352--4367, 2020.

\bibitem{tang2018d}
H.~Tang, X.~Lian, M.~Yan, C.~Zhang, and J.~Liu, ``$d^2$: Decentralized training
  over decentralized data,'' in {\em International Conference on Machine
  Learning}, pp.~4848--4856, 2018.

\bibitem{yuan2017exact1}
K.~Yuan, B.~Ying, X.~Zhao, and A.~H. Sayed, ``Exact dffusion for distributed
  optimization and learning -- {Part I: Algorithm development},'' {\em IEEE
  Transactions on Signal Processing}, vol.~67, no.~3, pp.~708 -- 723, 2018.

\bibitem{yuan2018exact2}
K.~Yuan, B.~Ying, X.~Zhao, and A.~H. Sayed, ``Exact diffusion for distributed
  optimization and learning—{Part II}: Convergence analysis,'' {\em IEEE
  Transactions on Signal Processing}, vol.~67, no.~3, pp.~724--739, 2018.

\bibitem{li2019decentralized}
Z.~Li, W.~Shi, and M.~Yan, ``A decentralized proximal-gradient method with
  network independent step-sizes and separated convergence rates,'' {\em IEEE
  Transactions on Signal Processing}, vol.~67, no.~17, pp.~4494--4506, 2019.

\bibitem{alghunaim2021unified}
S.~A. Alghunaim and K.~Yuan, ``A unified and refined convergence analysis for
  non-convex decentralized learning,'' {\em arXiv preprint arXiv:2110.09993},
  2021.

\bibitem{xu2015augmented}
J.~Xu, S.~Zhu, Y.~C. Soh, and L.~Xie, ``Augmented distributed gradient methods
  for multi-agent optimization under uncoordinated constant stepsizes,'' in
  {\em IEEE Conference on Decision and Control (CDC)}, (Osaka, Japan),
  pp.~2055--2060, 2015.

\bibitem{di2016next}
P.~Di~Lorenzo and G.~Scutari, ``Next: In-network nonconvex optimization,'' {\em
  IEEE Transactions on Signal and Information Processing over Networks},
  vol.~2, no.~2, pp.~120--136, 2016.

\bibitem{nedic2017achieving}
A.~Nedic, A.~Olshevsky, and W.~Shi, ``Achieving geometric convergence for
  distributed optimization over time-varying graphs,'' {\em SIAM Journal on
  Optimization}, vol.~27, no.~4, pp.~2597--2633, 2017.

\bibitem{qu2018harnessing}
G.~Qu and N.~Li, ``Harnessing smoothness to accelerate distributed
  optimization,'' {\em IEEE Transactions on Control of Network Systems},
  vol.~5, no.~3, pp.~1245--1260, 2018.

\bibitem{xin2020improved}
R.~Xin, U.~A. Khan, and S.~Kar, ``An improved convergence analysis for
  decentralized online stochastic non-convex optimization,'' {\em IEEE
  Transactions on Signal Processing}, vol.~69, pp.~1842--1858, 2021.

\bibitem{pu2020distributed}
S.~Pu and A.~Nedi{\'c}, ``Distributed stochastic gradient tracking methods,''
  {\em Mathematical Programming}, pp.~1--49, 2020.

\bibitem{huang2021improve}
K.~Huang and S.~Pu, ``Improving the transient times for distributed stochastic
  gradient methods,'' {\em arXiv preprint arXiv:2105.04851}, 2021.

\bibitem{koloskova2021improved}
A.~Koloskova, T.~Lin, and S.~U. Stich, ``An improved analysis of gradient
  tracking for decentralized machine learning,'' {\em Advances in Neural
  Information Processing Systems}, vol.~34, 2021.

\bibitem{tsitsiklis1986distributed}
J.~Tsitsiklis, D.~Bertsekas, and M.~Athans, ``Distributed asynchronous
  deterministic and stochastic gradient optimization algorithms,'' {\em IEEE
  transactions on automatic control}, vol.~31, no.~9, pp.~803--812, 1986.

\bibitem{duchi2011dual}
J.~C. Duchi, A.~Agarwal, and M.~J. Wainwright, ``Dual averaging for distributed
  optimization: Convergence analysis and network scaling,'' {\em IEEE
  Transactions on Automatic control}, vol.~57, no.~3, pp.~592--606, 2011.

\bibitem{sayed2014adaptation}
A.~H. Sayed, ``Adaptation, learning, and optimization over networks,'' {\em
  Foundations and Trends in Machine Learning}, vol.~7, no.~ARTICLE,
  pp.~311--801, 2014.

\bibitem{chen2013distributed}
J.~Chen and A.~H. Sayed, ``Distributed pareto optimization via diffusion
  strategies,'' {\em IEEE Journal of Selected Topics in Signal Processing},
  vol.~7, no.~2, pp.~205--220, 2013.

\bibitem{yuan2016convergence}
K.~Yuan, Q.~Ling, and W.~Yin, ``On the convergence of decentralized gradient
  descent,'' {\em SIAM Journal on Optimization}, vol.~26, no.~3,
  pp.~1835--1854, 2016.

\bibitem{wei2012distributed}
E.~Wei and A.~Ozdaglar, ``Distributed alternating direction method of
  multipliers,'' in {\em IEEE Conference on Decision and Control (CDC)}, (Maui,
  HI, USA), pp.~5445--5450, 2012.

\bibitem{shi2014linear}
W.~Shi, Q.~Ling, K.~Yuan, G.~Wu, and W.~Yin, ``On the linear convergence of the
  admm in decentralized consensus optimization,'' {\em IEEE Transactions on
  Signal Processing}, vol.~62, no.~7, pp.~1750--1761, 2014.

\bibitem{shi2015extra}
W.~Shi, Q.~Ling, G.~Wu, and W.~Yin, ``{EXTRA}: An exact first-order algorithm
  for decentralized consensus optimization,'' {\em SIAM Journal on
  Optimization}, vol.~25, no.~2, pp.~944--966, 2015.

\bibitem{alghunaim2020decentralized}
S.~A. Alghunaim, E.~K. Ryu, K.~Yuan, and A.~H. Sayed, ``Decentralized proximal
  gradient algorithms with linear convergence rates,'' {\em IEEE Transactions
  on Automatic Control}, vol.~66, pp.~2787--2794, June 2021.

\bibitem{scaman2017optimal}
K.~Scaman, F.~Bach, S.~Bubeck, Y.~T. Lee, and L.~Massouli{\'e}, ``Optimal
  algorithms for smooth and strongly convex distributed optimization in
  networks,'' in {\em International Conference on Machine Learning},
  pp.~3027--3036, 2017.

\bibitem{scaman2018optimal}
K.~Scaman, F.~Bach, S.~Bubeck, L.~Massouli{\'e}, and Y.~T. Lee, ``Optimal
  algorithms for non-smooth distributed optimization in networks,'' in {\em
  Advances in Neural Information Processing Systems}, pp.~2740--2749, 2018.

\bibitem{uribe2020dual}
C.~A. Uribe, S.~Lee, A.~Gasnikov, and A.~Nedi{\'c}, ``A dual approach for
  optimal algorithms in distributed optimization over networks,'' {\em
  Optimization Methods and Software}, pp.~1--40, 2020.

\bibitem{lu2019gnsd}
S.~Lu, X.~Zhang, H.~Sun, and M.~Hong, ``Gnsd: A gradient-tracking based
  nonconvex stochastic algorithm for decentralized optimization,'' in {\em 2019
  IEEE Data Science Workshop (DSW)}, pp.~315--321, IEEE, 2019.

\bibitem{zhang2019decentralized}
J.~Zhang and K.~You, ``Decentralized stochastic gradient tracking for
  non-convex empirical risk minimization,'' {\em arXiv preprint
  arXiv:1909.02712}, 2019.

\bibitem{xin2022fast}
R.~Xin, U.~A. Khan, and S.~Kar, ``Fast decentralized nonconvex finite-sum
  optimization with recursive variance reduction,'' {\em SIAM Journal on
  Optimization}, vol.~32, no.~1, 2022.

\bibitem{yu2019linear}
H.~Yu, R.~Jin, and S.~Yang, ``On the linear speedup analysis of communication
  efficient momentum sgd for distributed non-convex optimization,'' in {\em
  International Conference on Machine Learning}, pp.~7184--7193, PMLR, 2019.

\bibitem{lin2021quasi}
T.~Lin, S.~P. Karimireddy, S.~U. Stich, and M.~Jaggi, ``Quasi-global momentum:
  Accelerating decentralized deep learning on heterogeneous data,'' {\em arXiv
  preprint arXiv:2102.04761}, 2021.

\bibitem{berahas2018balancing}
A.~S. Berahas, R.~Bollapragada, N.~S. Keskar, and E.~Wei, ``Balancing
  communication and computation in distributed optimization,'' {\em IEEE
  Transactions on Automatic Control}, vol.~64, no.~8, pp.~3141--3155, 2018.

\bibitem{li2020decentralized}
H.~Li, C.~Fang, W.~Yin, and Z.~Lin, ``Decentralized accelerated gradient
  methods with increasing penalty parameters,'' {\em IEEE Transactions on
  Signal Processing}, vol.~68, pp.~4855--4870, 2020.

\bibitem{lu2021optimal}
Y.~Lu and C.~De~Sa, ``Optimal complexity in decentralized training,'' in {\em
  International Conference on Machine Learning}, pp.~7111--7123, PMLR, 2021.

\bibitem{li2019convergence}
X.~Li, K.~Huang, W.~Yang, S.~Wang, and Z.~Zhang, ``On the convergence of fedavg
  on non-iid data,'' in {\em International Conference on Learning
  Representations}, 2019.

\bibitem{stich2019local}
S.~U. Stich, ``Local sgd converges fast and communicates little,'' in {\em
  International Conference on Learning Representations (ICLR)}, 2019.

\bibitem{liu2011accelerated}
J.~Liu and A.~S. Morse, ``Accelerated linear iterations for distributed
  averaging,'' {\em Annual Reviews in Control}, vol.~35, no.~2, pp.~160--165,
  2011.

\bibitem{deng2012mnist}
L.~Deng, ``The mnist database of handwritten digit images for machine learning
  research [best of the web],'' {\em IEEE Signal Processing Magazine}, vol.~29,
  no.~6, pp.~141--142, 2012.

\bibitem{rossi2015network}
R.~Rossi and N.~Ahmed, ``The network data repository with interactive graph
  analytics and visualization,'' in {\em Twenty-Ninth AAAI Conference on
  Artificial Intelligence}, 2015.

\bibitem{stich2019unified}
S.~U. Stich, ``Unified optimal analysis of the (stochastic) gradient method,''
  {\em arXiv preprint arXiv:1907.04232}, 2019.

\end{thebibliography}
	
	
	
	
	\appendix
	\section{Notations and Preliminaries}
	\label{app:notation}
	
	
	
	We first review some notations and facts. 
	\begin{itemize}
		\item $W=[w_{ij}]\in \mathbb{R}^{n\times n}$ is a symmetric and doubly stochastic combination matrix
		\item $\bar{W} = (I+W)/2 \in \mathbb{R}^{n\times n}$
		\item $V = (I-\bar{W})^{1/2} = (\frac{I-W}{2})^{1/2}$ and hence $I-\bar{W} = V^2$
		\item $\lambda_i(W)$ is the $i$th largest eigenvalue of matrix $W$, and $\bar{\lambda}_i(W) = (1+\lambda_i(W))/2$ is the $i$th largest eigenvalue of matrix $\bar{W}$. Note that $\lambda_i(W)\in(-1,1)$ and $\bar{\lambda}_i(W) \in (0, 1)$ for $i=2,\dots, n$.
		\item Let $\Lambda = \diag\{\lambda_1(W),\ldots, \lambda_n(W)\} \in \mathbb{R}^{n\times n}$. It holds that $W = Q\Lambda Q^T$ where $Q = [q_1, q_2,\ldots, q_n]\in \mathbb{R}^{n\times n}$ is an orthogonal matrix and $q_1 = \frac{1}{\sqrt{n}}\mathds{1}_n$. 
		\item $\bar{W} = Q\bar{\Lambda}Q^T$ where $\bar{\Lambda} = (I + \Lambda)/2$.
		\item $V = Q(I-\bar{\Lambda})^{1/2}Q^T$
		\item If a matrix $A\in \mathbb{R}^{n\times n}$ is normal, \textit{i.e.}, $A A^T = A^T A$, it holds that $A = U D U^*$ where $D$ is a diagonal matrix and $U$ is a unitary matrix.
		\item If a matrix $\Pi\in \mathbb{R}^{n\times n}$ is a permutation matrix, it holds that $\Pi^{-1} = \Pi^T$.
	\end{itemize}
	
	\noindent \textbf{Smoothness.} Since each $f_i(x)$ is assumed to be $L$-smooth in Assumption \ref{ass:smoothness}, it holds that $f(x) = \frac{1}{n}\sum_{i=1}^n f_i(x)$ is also $L$-smooth. As a result, the following inequality holds for any $x, y \in \mathbb{R}^d$:
	\begin{align}
	f_i(x) - f_i(y) - \frac{L}{2}\|x - y\|^2 &\le  \langle \nabla f_i(y), x- y \rangle \label{sdu-2}
	\end{align}
	
	\noindent \textbf{Smoothness and convexity.} If each $f_i(x)$ is further assumed to be convex (see Assumption \ref{ass:convex}), it holds that $f(x) = \frac{1}{n}\sum_{i=1}^n f_i(x)$ is also  convex. For this scenario, it holds for any $x, y \in \mathbb{R}^d$ that:
	\begin{align}
	\|\nabla f(x) - \nabla f(x^\star)\|^2 &\le 2L\big(f(x) - f(x^\star)\big) \label{sdu-1} \\
	f_i(x) -	f_i(y) &\le  \langle \nabla f_i(x), x - y \rangle  \label{sdu-3}
	\end{align}
	
	\noindent \textbf{Submultiplicativity of the Frobenius norm.} Given matrices $W\in \RR^{n\times n}$ and $\vy\in \RR^{n\times d}$, it holds that 
	\begin{align}\label{submulti}
	\|W\vy\|_F \le \|W\|_2 \|\vy\|_F.
	\end{align}
	To verify it, by letting $y_j$ be the $j$th column of $\vy$, we have $\|W\vy\|_F^2 = \sum_{j=1}^d \|Wy_j\|_2^2 \le \sum_{j=1}^d \|W\|_2^2 \|y_j\|_2^2=\|W\|_2^2\|\vy\|_F^2$. 
	
	\section{The Fundamental Decomposition}
	\subsection{Proof of Lemma \ref{lm-decom}}
	\label{app:decom}
	We now analyze the eigen-decomposition of matrix $B$:
	\begin{align}\label{xnbxcbsdbds7}
	B = 
	\ba{cc}
	\bar{W} & -V \\
	V\bar{W} & \bar{W}
	\ea.
	\end{align}

	\noindent \textbf{Proof.}
	Using $\bar{W} = Q\bar{\Lambda}Q^T$, it holds that
	\begin{align}\label{B}
	B = 
	\ba{cc}
	Q & 0 \\
	0 & Q
	\ea
	\ba{cc}
	\bar{\Lambda} & -(I - \bar{\Lambda})^{1/2} \\
	\bar{\Lambda} (I - \bar{\Lambda})^{1/2} & \bar{\Lambda}
	\ea
	\ba{cc}
	Q^T & 0 \\
	0 & Q^T
	\ea.
	\end{align}
	Note that $\bar{\Lambda} (I - \bar{\Lambda})^{1/2} = (I - \bar{\Lambda})^{1/2} \bar{\Lambda}$ because both $\bar{\Lambda}$ and $I - \bar{\Lambda}$ are diagonal matrices. We next introduce 
	\begin{align}
	E_{(i)} &= 
	\ba{cc}
	\bar{\lambda}_i & -(1-\bar{\lambda}_i)^{1/2} \\
	\bar{\lambda}_i (1-\bar{\lambda}_i)^{1/2} & \bar{\lambda}_i
	\ea \in \mathbb{R}^{2\times 2} \label{Ei} \vspace{1mm} \\
	E &= \mathrm{BlockDiag}\{E_{(1)}, \cdots, E_{(n)} \} \in \RR^{2n\times 2n}
	\end{align}
	where $\bar{\lambda}_i = \lambda_i(\bar{W})$, and $E$ is a block diagonal matrix with each $i$th bloack diagonal matrix as $E_{(i)}$. It is easy to verify that there exists some permutation matrix $\Pi$ such that 
	\begin{align}\label{237sdgsd9}
	B = 
	\ba{cc}
	Q & 0 \\
	0 & Q
	\ea
	\Pi\, E\, \Pi^T 
	\ba{cc}
	Q^T & 0 \\
	0 & Q^T
	\ea.
	\end{align}
	Next we focus on the matrix $E_{(i)}$ defined in \eqref{Ei}. Note that $E_{(1)} = I$. For $i \ge 2$, it holds that 
	\begin{align}\label{x7236}
	E_{(i)} &= 
	\underbrace{\ba{cc}
		1 & 0 \\
		0 & \bar{\lambda}_i^{1/2}
		\ea}_{C_{(i)}}
	\underbrace{\ba{cc}
		\bar{\lambda}_i & -[\bar{\lambda}_i(1-\bar{\lambda}_i)]^{1/2} \\
		\bar{[\lambda}_i (1-\bar{\lambda}_i)]^{1/2} & \bar{\lambda}_i
		\ea}_{G_{(i)}}
	\underbrace{\ba{cc}
		1 & 0 \\
		0 & \bar{\lambda}_i^{-1/2}
		\ea}_{C^{(-1)}_{(i)}}
	\end{align}
	Since $G_{(i)}$ is normal, it holds that (see Appendix \ref{app:notation})
	\begin{align}\label{xcnxcnxcn}
	G_{(i)} = U_{(i)} D_{(i)} U_{(i)}^*, \quad \mbox{where} \quad D_i = \diag\{\sigma_1(G_{(i)}), \sigma_2(G_{(i)})\},
	\end{align}
	In the above expression, $\sigma_1(G_{(i)})$ and $\sigma_2(G_{(i)})$ are complex eigenvalues of $G_{(i)}$. Moreover, it holds that $|\sigma_1(G_{(i)})| = |\sigma_2(G_{(i)})| = \bar{\lambda}_i^{1/2} < 1$. The quantity $U_{(i)} \in \mathbb{R}^{2\times 2}$ is a unitary matrix. Next we define $C = \mathrm{BlockDiag}\{C_{(1)},\ldots, C_{(n)}\}$, $U = \mathrm{BlockDiag}\{U_{(1)},\ldots, U_{(n)}\}$, and $D = \mathrm{BlockDiag}\{D_{(1)},\ldots, D_{(n)}\}$. By substituting \eqref{x7236} and \eqref{xcnxcnxcn} into \eqref{237sdgsd9}, we have
	\begin{align}\label{2ebg}
	B = 
	d \ba{cc}
	Q & 0 \\
	0 & Q
	\ea
	\Pi\, C\, U\, D\, U^*\, C^{-1}\, \Pi^T 
	\ba{cc}
	Q^T & 0 \\
	0 & Q^T
	\ea d^{-1}
	\end{align}
	where $d$ is any positive constant. Next we define 
	\begin{align}
	X = d \ba{cc}
	Q & 0 \\
	0 & Q
	\ea
	\Pi\, C\, U, \quad X^{-1} = U^*\, C^{-1}\, \Pi^T 
	\ba{cc}
	Q^T & 0 \\
	0 & Q^T
	\ea d^{-1}.
	\end{align}
	By letting $d=\sqrt{n}$ and considering the structure of $Q$, $\Pi$, $C$, and $U$, it is easy to verify that 
	\begin{align}
	X &= [r_1\ r_2\ X_R] \quad \mbox{where} \quad 
	r_1 = 
	\ba{c}
	\mathds{1}_n\\
	0
	\ea,
	\quad
	r_2 = 
	\ba{c}
	0 \\
	\mathds{1}_n
	\ea \label{xnbsdbsd7-1}\\
	X^{-1} &= 
	[\ell_1\ \ell_2\ X_L^T]^T
	\ \, 
	\mbox{where}
	\quad
	\ell_1 = 
	\ba{c}
	\frac{1}{n}\mathds{1}_n \\
	0
	\ea,\quad
	\ell_2 = 
	\ba{c}
	0 \\
	\frac{1}{n}\mathds{1}_n
	\ea \label{xnbsdbsd7-2}
	\end{align}
	With \eqref{2ebg}--\eqref{xnbsdbsd7-2}, it holds that 
	\begin{align}
	B = X D X^{-1}
	\end{align}
	where $X$ and $X^{-1}$ take the form of \eqref{xnbsdbsd7-1} and \eqref{xnbsdbsd7-2}, and 
	\begin{align}
	D = 
	\ba{ccc}
	1 & 0 & 0\\
	0 & 1 & 0\\
	0 & 0 & D_1
	\ea
	\end{align}
	and $D_1$ is a diagonal matrix with complex entries.   The magnitudes of the diagonal entries in $D_1$ are all strictly less than $1$. Next we evaluate the quantity $\|X\|\|X^{-1}\|$:
	\begin{align}\label{xnwsdn}
	\|X\|\|X^{-1}\| & \le 
	\left\| 
	\ba{cc}
	Q & 0 \\
	0 & Q
	\ea
	\right\| \|\Pi\| \|C\| \|U\| \|U^*\| \|C^{-1}\| \|\Pi^T\| \left\|
	\ba{cc}
	Q^T & 0 \\
	0 & Q^T
	\ea
	\right\| \nonumber \\
	& \overset{(a)}{=} \|C\| \|C^{-1}\| \nonumber \\
	& \le \max_i\{\bar{\lambda}^{-1/2}_i\} = \bar{\lambda}^{-1/2}_{n}
	\end{align}
	where (a) holds because $Q$ is orthogonal, $U$ is unitary, and $\Pi^T \Pi = I$. Note that 
	\begin{align}
	X_R = X S, \quad \mbox{and}\quad X_L = S^T X^{-1}
	\end{align}
	where $S = [e_3,\ldots,e_{2n}] \in \mathbb{R}^{2n\times 2(n-1)}$ and $e_j$ is the $j$th column of the identity matrix $I_{2n}$.
	It then holds that
	\begin{align}
	\|X_R\|\|X_L\| \le \|X\| \|S\| \|S^T\| \|X^{-1}\| = \|X\| \|X^{-1}\| \overset{\eqref{xnwsdn}}{\le} \bar{\lambda}^{-1/2}_{n}
	\end{align}
	\rightline{$\blacksquare$}
	
	\subsection{Proof of Lemma \ref{lm-transform-error-dyanmic}}
	\label{app-transform}
	{
		\begin{proof}
			By left-multiplying $X^{-1}$ to both sides of \eqref{xzn2websd7} and utilizing the decomposition in \eqref{eq-fundamental-decomposision}, we have 
			\begin{align}\label{xzn2websd7-2}
			\ba{c}
			\hspace{-2mm}\ell_1^T \hspace{-2mm}\\
			\hspace{-2mm}\ell_2^T \hspace{-2mm}\\
			\hspace{-2mm}X_L/c \hspace{-2mm}
			\ea \hspace{-2mm}
			\left[
			\begin{array}{c}
			\hspace{-2mm}\vx^{(k+1)} - \vx^\star  \hspace{-1.5mm}\\
			\hspace{-2mm}\vy^{(k+1)} - \vy^\star \hspace{-1.5mm}
			\end{array}
			\right] \hspace{-0.6mm} = \hspace{-0.6mm}
			\ba{ccc}
			1 & 0 & 0\\
			0 & 1 & 0\\
			0 & 0 & D_1
			\ea \hspace{-2mm}
			\ba{c} 
			\hspace{-2mm}\ell_1^T\hspace{-2mm} \\
			\hspace{-2mm}\ell_2^T\hspace{-2mm} \\
			\hspace{-2mm}X_L/c\hspace{-2mm}
			\ea
			\hspace{-2mm}
			\left[
			\begin{array}{c}
			\hspace{-2mm}\vx^{(k)} - \vx^\star \hspace{-1.5mm}\\
			\hspace{-2mm}\vy^{(k)} - \vy^\star \hspace{-1.5mm}
			\end{array}
			\right]
			- \gamma 
			\ba{c}
			\hspace{-2mm}\ell_1^T \hspace{-2mm}\\
			\hspace{-2mm}\ell_2^T \hspace{-2mm}\\
			\hspace{-2mm}X_L/c\hspace{-2mm}
			\ea
			\hspace{-2mm}
			\left[
			\begin{array}{c}
			\hspace{-2mm}\bar{W}(\nabla f(\vx^{(k)}) - \nabla f(\vx^\star) + \vs^{(k)}) \hspace{-1.5mm}\\
			\hspace{-2mm} V \bar{W}(\nabla f(\vx^{(k)}) - \nabla f(\vx^\star) + \vs^{(k)}) \hspace{-1.5mm}
			\end{array}
			\right].
			\end{align}
			With the definition of $\bar{\z}^{(k)}$  in \eqref{w-y-vs-bz-cz}, the structure of $\ell_1$ in \eqref{zb236asd00991}, and $\bar{W}\mathds{1} = \mathds{1}$, the first line in \eqref{xzn2websd7-2} becomes 
			\begin{align}\label{2bnabzbz4}
			\bar{\z}^{(k+1)} = \bar{\z}^{(k)} - \frac{\gamma}{n}\mathds{1}^T (\nabla f(\vx^{(k)}) - \nabla f(\vx^\star)) - \gamma \bar{\vs}^{(k)}.
			\end{align}
			where $\bar{\vs}^{(k)}$ is defined in \eqref{s-bar}. With the structure of $\ell_2$ in \eqref{zb236asd00991} and $V \mathds{1} = 0$, the second line in \eqref{xzn2websd7-2} becomes 
			\begin{align}\label{9zbav5a}
			\ell_2^T 
			\left[
			\begin{array}{c}
			\hspace{-2mm}\vx^{(k+1)} - \vx^\star  \hspace{-1.5mm}\\
			\hspace{-2mm}\vy^{(k+1)} - \vy^\star \hspace{-1.5mm}
			\end{array}
			\right] = \ell_2^T 
			\left[
			\begin{array}{c}
			\hspace{-2mm}\vx^{(k)} - \vx^\star  \hspace{-1.5mm}\\
			\hspace{-2mm}\vy^{(k)} - \vy^\star \hspace{-1.5mm}
			\end{array}
			\right]
			\quad \Longleftrightarrow \quad \frac{1}{n}\mathds{1}^T(\vy^{(k+1)} - \vy^\star) = \frac{1}{n}\mathds{1}^T(\vy^{(k)} - \vy^\star)
			\end{align}
			Since $\vy^\star$ lies in the range space of $V$ (see Lemma \ref{lm-opt-cond}) and $\vy^{(k)}$ also lies in the range space of $V$ when $\vy^{(0)} = 0$ (see the update of $\vy$ in \eqref{eq:d2-pd}), it holds that $\frac{1}{n}\mathds{1}^T(\vy^{(k)} - \vy^\star) = 0$. As a result, the recursion \eqref{9zbav5a} can be ignored since $\frac{1}{n}\mathds{1}^T(\vy^{(k)} - \vy^\star) = 0$ holds for all iterations. 
			
			Finally we examine the third line in \eqref{xzn2websd7-2}. To this end, we eigen-decompose $\bar{W}$ as 
			\begin{align} \label{zbnzb0zbzbz09}
			\bar{W} = 
			\underbrace{\ba{cc}
				\frac{1}{\sqrt{n}}\mathds{1} & Q_R
				\ea}_{:= Q}
			\ba{cc}
			1 & 0  \\
			0 & \bar{\Lambda}_R
			\ea
			\ba{cc}
			\frac{1}{\sqrt{n}}\mathds{1}^T \\
			Q_R^T
			\ea = \frac{1}{n}\mathds{1}\mathds{1}^T + Q_R \bar{\Lambda}_R Q_R^T
			\end{align}
			where $Q_R\in \RR^{n\times (n-1)}$ satisfies $Q_R^T Q_R=I$ and $\bar{\Lambda}_R = \mathrm{diag}\{{\lambda}_2(\bar{W}),\dots, {\lambda}_n(\bar{W})\}$. Since $V$ shares the same eigen-space  as $\bar{W}$, it holds that 
			\begin{align}\label{z236zgzba09}
			V \bar{W} = 
			\ba{cc}
			\frac{1}{\sqrt{n}}\mathds{1} & Q_R
			\ea
			\ba{cc}
			0& 0  \\
			0 & (I - \bar{\Lambda}_R)^{\frac{1}{2}} \bar{\Lambda}_R
			\ea
			\ba{cc}
			\frac{1}{\sqrt{n}}\mathds{1}^T \\
			Q_R^T
			\ea = Q_R (I - \bar{\Lambda}_R)^{\frac{1}{2}} \bar{\Lambda}_R Q_R^T
			\end{align}
			Next we rewrite $X_L = [X_{L,\ell}\ X_{L,r}]$ with $X_{L,\ell} \in \RR^{2(n-1) \times n}$ and $X_{L,r} \in \RR^{2(n-1) \times n}$. With the structure of $X$ and $X^{-1}$ in \eqref{eq-fundamental-decomposision} and the equality $X^{-1} X = I$, it holds that 
			\begin{align}\label{08z5}
			X_{L,\ell} \mathds{1} = 0 \quad \quad X_{L,r} \mathds{1} = 0.
			\end{align}
			With \eqref{zbnzb0zbzbz09}, \eqref{z236zgzba09}, and \eqref{08z5}, we have
			\begin{align}
			&\ X_L \left[
			\begin{array}{c}
			\hspace{-2mm}\bar{W}(\nabla f(\vx^{(k)}) - \nabla f(\vx^\star) + \vs^{(k)}) \hspace{-1.5mm}\\
			\hspace{-2mm} V \bar{W}(\nabla f(\vx^{(k)}) - \nabla f(\vx^\star) + \vs^{(k)}) \hspace{-1.5mm}
			\end{array}
			\right] \nonumber \\
			=&\ X_{L, \ell} \bar{W}(\nabla f(\vx^{(k)}) - \nabla f(\vx^\star) + \vs^{(k)})  + X_{L, r} V \bar{W}(\nabla f(\vx^{(k)}) - \nabla f(\vx^\star) + \vs^{(k)}) \nonumber \\
			\overset{(a)}{=}&\ X_{L, \ell} Q_R \bar{\Lambda}_R Q_R^T(\nabla f(\vx^{(k)}) - \nabla f(\vx^\star) + \vs^{(k)}) + X_{L, r} Q_R (I - \bar{\Lambda}_R)^{\frac{1}{2}} \bar{\Lambda}_R Q_R^T(\nabla f(\vx^{(k)}) - \nabla f(\vx^\star) + \vs^{(k)}) \nonumber \\
			=&\ \check{\g}^{(k)} + \check{\s}^{(k)} 
			\end{align}
			where (a) holds because of \eqref{zbnzb0zbzbz09} -- \eqref{08z5}, and $\check{\g}^{(k)}$ and $\check{\s}^{(k)}$ in the last equality are defined in \eqref{g-check} and \eqref{s-check}, respectively. With the above equality and the definition of $\check{\vz}^{(k)}$ in \eqref{w-y-vs-bz-cz}, the third line in \eqref{xzn2websd7-2} becomes 
			\begin{align}\label{ba897}
				\check{\vz}^{(k+1)} = D_1 \check{\vz}^{(k)} - \frac{\gamma}{c}\check{\vg}^{(k)} -  \frac{\gamma}{c}\check{\vs}^{(k)}.
 			\end{align}
 			Combining \eqref{2bnabzbz4} and \eqref{ba897}, we achieve the result in \eqref{transformed-error-dynamics}. 
	\end{proof}}

	\subsection{Proof of Proposition \ref{remark-z0}}
	\label{app-proof-remark}
		\begin{proof}
			We first evaluate the magnitude of $\|\vy^\star\|^2_F$. Recall that $V = (I - \bar{W})^{\frac{1}{2}}$ and $\bar{W} = (I+W)/2$ is a symmetric and doubly-stochastic matrix. If we let $\bar{\lambda}_k = \lambda_k(\bar{W})$, it holds that $1 = \bar{\lambda}_1 > \bar{\lambda}_2 \ge \cdots \ge \bar{\lambda}_n > 0$. We eigen-decompose $V = U \Lambda U^T$ with $\Lambda = \mathrm{diag}\{\lambda_k(V)\}$ and $\lambda_k(V) = (1 - \bar{\lambda}_k)^{\frac{1}{2}}$. We next introduce $V^\dagger = U \Lambda^\dagger U^T$ with $\Lambda^\dagger = \mathrm{diag}\{\lambda_k(V^\dagger)\}$ in which $\lambda_1(V^\dagger) = 0$ and $\lambda_k(V^\dagger) = \lambda_k^{-1}(V) = (1-\bar{\lambda}_k)^{-\frac{1}{2}}$ for $2\le k \le n$.  Recall optimality condition \eqref{eq-opt-cond-1} that 
			\begin{align}
			\gamma \bar{W} \nabla f(\vx^\star) + V \vy^\star = 0 \quad \Longleftrightarrow \quad V \vy^\star = - \gamma \bar{W} \nabla f(\vx^\star). \label{znbzbzbzb09123}
			\end{align}
			Since $\vy^\star$ lies in the range space of $V$, it holds that $V^\dagger V \vy^\star = \vy^\star$. This fact together with \eqref{znbzbzbzb09123} leads to 
			\begin{align}\label{znn098711}
			\vy^\star = -\gamma V^\dagger \bar{W} \nabla f(\vx^\star) \quad \Longrightarrow \quad \|\vy^\star\|_F^2 \le \gamma^2 \|V^\dagger \bar{W}\|_2^2 \|\nabla f(\vx^\star)\|_F^2 \le  \frac{\gamma^2\bar{\lambda}_2^2}{1-\bar{\lambda}_2} \|\nabla f(\vx^\star)\|_F^2 = O\left(\frac{n \gamma^2 \bar{\lambda}^2_2}{1-\bar{\lambda}_2}\right)
			\end{align}
			where we regard $\|\nabla f(\vx^\star)\|_F^2 = \sum_{i=1}^n \|\nabla f_i(x^\star)\|^2 = O(n)$ in the last inequality. 
			
			Next we evaluate the magnitude of $\mathbb{E}\|\check{\vz}^{(0)}\|^2_F$. Recall from \eqref{w-y-vs-bz-cz} that 
			\begin{align}
			\check{\vz}^{(0)} = \frac{X_{L,\ell}}{c} (\vx^{(0)} - \vx^\star) + \frac{X_{L,r}}{c} (\vy^{(0)} - \vy^\star) = -  \frac{X_{L,\ell}}{c} \vx^\star - \frac{X_{L,r}}{c} \vy^\star
			\end{align}
			where we utilized $\vx^{(0)} = 0$ and $\vy^{(0)} = 0$ in the last equality, and $X_L = [X_{L,\ell}, X_{L,r}]$. With \eqref{08z5} and the fact that $\vx^\star = x^\star \mathds{1}$, it holds that $X_{L,\ell} \vx^\star = 0$ and 
			\begin{align}
			\|\check{\vz}^{(0)} \|^2_F \le \frac{1}{c^2}\|X_{L,r}\|_2^2 \|\vy^\star\|_F^2 \overset{(a)}{\le} \frac{1}{c^2} \|X_{L}\|_2^2 \|\vy^\star\|_F^2 \overset{(b)}{=} \|\vy^\star\|_F^2 \overset{\eqref{znn098711}}{=} O\left(\frac{n \gamma^2 \bar{\lambda}^2_2}{1-\bar{\lambda}_2}\right)
			\end{align}
			where (a) holds because $\|X_{L,r}\| \le \|X_{L}\|$ (see the detail derivation in \eqref{znb213987}) and (b) holds by setting $c=\|X_L\|$). 	
			
		\end{proof}
	
	\section{Convergence Analysis for Generally-Convex Scenario}
	\subsection{Proof of Lemma \ref{lm-descent-gc}}
	\label{app-lm-descent-gc}
	\begin{proof}
		From \eqref{w-y-vs-bz-cz} we have $[\bar{\z}^{(k)}]^T = [\frac{1}{n}\mathds{1}^T(\vx^{(k)} - \vx^\star)]^T = \bar{x}^{(k)} - x^\star \in \RR^{d}$, where $\bar{x}^{(k)} = \frac{1}{n}\mathds{1}^T\vx^{(k)}$ and $x^\star$ is the global solution to problem \eqref{eq:general-prob}. With this relation, the first line of \eqref{transformed-error-dynamics} becomes 
		\begin{align}\label{first-line}
		\bar{x}^{(k+1)} - x^\star = \bar{x}^{(k)} - x^\star - \frac{\gamma}{n}\mathds{1}^T (\nabla f(\vx^{(k)}) - \nabla f(\vx^\star)) - \gamma \bar{\vs}^{(k)}.
		\end{align}
		The above equality implies that
		\begin{align}\label{zxc7234g}
		\mathbb{E}[\|\bar{x}^{(k+1)} - x^\star\|^2|\cF^{(k)}] \overset{\eqref{avg-noise}}{\le} \|\bar{x}^{(k)} - x^\star - \frac{\gamma}{n}\mathds{1}^T (\nabla f(\vx^{(k)}) - \nabla f(\vx^\star))\|^2 + \frac{\gamma^2 \sigma^2}{n}
		\end{align}
		Note that the first term can be expanded as follows. 
		\begin{align}\label{shdshsd}
		&\ \|\bar{x}^{{(k)}} - x^\star -  \frac{\gamma}{n}\sum_{i=1}^n [\nabla f_i(x_i^{(k)}) - \nabla f_i(x^\star)]\|^2 \nonumber \\
		=&\ \|\bar{x}^{{(k)}} - x^\star\|^2 - \underbrace{\frac{2\gamma}{n}\sum_{i=1}^n \langle \bar{x}^{{(k)}} - x^\star, \nabla f_i(x_i^{(k)}) - \nabla f_i(x^\star) \rangle}_{(A)} + \underbrace{\gamma^2\| \frac{1}{n}\sum_{i=1}^n [\nabla f_i(x_i^{(k)}) - \nabla f_i(x^\star)]\ \|^2}_{(B)}
		\end{align}
		We now bound the term (A):
		\begin{align}\label{23nbsdn-2}
		& \frac{2\gamma}{n}\sum_{i=1}^n \langle \bar{x}^{{(k)}} - x^\star, \nabla f_i(x_i^{(k)}) - \nabla f_i(x^\star) \rangle \nonumber \\
		=&\  \frac{2\gamma}{n}\sum_{i=1}^n \langle \bar{x}^{{(k)}} - x^\star, \nabla f_i(x_i^{(k)})  \rangle \nonumber \\
		=&\ \frac{2\gamma}{n}\sum_{i=1}^n \langle \bar{x}^{{(k)}} - x_i^{(k)}, \nabla f_i(x_i^{(k)}) \rangle  + \frac{2\gamma}{n}\sum_{i=1}^n \langle x_i^{(k)} - x^\star, \nabla f_i(x_i^{(k)})  \rangle \nonumber \\
		\overset{(a)}{\ge}&\ \frac{2\gamma}{n}\sum_{i=1}^n \Big( f_i(\bar{x}^{{(k)}}) - f_i(x_i^{(k)}) - \frac{L}{2}\|\bar{x}^{{(k)}} - x_i^{(k)}\|^2 \Big) + \frac{2\gamma}{n}\sum_{i=1}^n \Big(f_i(x_i^{(k)}) - f_i(x^\star) \Big)  \nonumber \\
		=&\ \frac{2\gamma}{n}\sum_{i=1}^n\Big( f_i(\bar{x}^{{(k)}}) - f_i(x^\star) \Big) - \frac{\gamma L}{n}\|\bar{\vx}^{{(k)}} - \vx^{(k)}\|_F^2 \nonumber \\
		=&\ 2\gamma \big( f(\bar{x}^{{(k)}}) - f(x^\star)\big) - \frac{\gamma L}{n}\|\bar{\vx}^{{(k)}} - \vx^{(k)}\|_F^2
		\end{align}
		where (a) holds because of each $f_i(x)$ is convex and $L$-smooth. We next bound term (B):
		\begin{align}\label{23nbsdn-1}
		& \gamma^2\| \frac{1}{n}\sum_{i=1}^n [\nabla f_i(x_i^{(k)}) - \nabla f_i(x^\star)]\ \|^2 \nonumber \\
		=&\ \gamma^2\| \frac{1}{n}\sum_{i=1}^n [\nabla f_i(x_i^{(k)}) - \nabla f_i(\bar{x}^{(k)}) + \nabla f_i(\bar{x}^{(k)})  -  \nabla f_i(x^\star)]\ \|^2 \nonumber \\
		\overset{\eqref{smooth-1}}{\le}&\ \frac{2\gamma^2 L^2}{n}\|\vx^{(k)} - \bar{\vx}^{(k)}\|_F^2 + 2\gamma^2 \|\nabla f(\bar{x}^{(k)}) - \nabla f(x^\star)\|^2 \nonumber \\
		\le&\ \frac{2\gamma^2 L^2}{n}\|\vx^{(k)} - \bar{\vx}^{(k)}\|_F^2 + 4L \gamma^2\big( f(\bar{x}^{(k)}) - f({x}^{\star}) \big).
		\end{align}
		where the last inequality holds because $f(x)$ is $L$-smooth.
		Substituting \eqref{23nbsdn-1} and \eqref{23nbsdn-2} into \eqref{shdshsd}, we have
		\begin{align}\label{shdshsd-2}
		&\ \|\bar{x}^{{(k)}} - x^\star -  \frac{\gamma}{n}\sum_{i=1}^n \nabla f_i(x_i^{(k)})\|^2 \nonumber \\
		\le&\ \|\bar{x}^{(k)} - x^\star\|^2 - 2\gamma(1 - 2L\gamma) \big( f(\bar{x}^{{(k)}}) - f(x^\star)\big) + \Big( \frac{\gamma L}{n} + \frac{2\gamma^2 L^2}{n}\Big)\|\bar{\vx}^{{(k)}} - \vx^{(k)}\|_F^2  \nonumber \\
		\le&\ \|\bar{x}^{(k)} - x^\star\|^2 - \gamma \big( f(\bar{x}^{{(k)}}) - f(x^\star)\big) + \frac{3\gamma L}{2n}\|\bar{\vx}^{{(k)}} - \vx^{(k)}\|_F^2 
		\end{align}
		where the last inequality holds when $\gamma \le \frac{1}{4L}$. Substituting \eqref{shdshsd-2} into \eqref{zxc7234g} and taking expectation on the filtration $\cF^{(k)}$, we achieve
		\begin{align}\label{23bsd999}
		& \mathbb{E}\|\bar{x}^{{(k+1)}} - x^\star\|^2 \le \mathbb{E}\|\bar{x}^{(k)} - x^\star\|^2 \hspace{-0.5mm}-\hspace{-0.5mm} \gamma \big(\mathbb{E} f(\bar{x}^{(k)}) \hspace{-0.5mm}-\hspace{-0.5mm} f(x^\star)\big) \hspace{-0.5mm}+\hspace{-0.5mm} \frac{3L\gamma}{2n}\mathbb{E}\|\vx^{(k)} \hspace{-0.5mm}-\hspace{-0.5mm} \bar{\vx}^{(k)}\|_F^2 \hspace{-0.5mm}+\hspace{-0.5mm} \frac{\gamma^2 \sigma^2}{n}
		\end{align}
		Substituting $\vx^{(k)} - \bar{\vx}^{(k)} = c X_{R,u} \check{\vz}^{(k)}$ (see \eqref{zzznnn}) and $[\bar{\z}^{(k)}]^T  = \bar{x}^{(k)} - x^\star$ into \eqref{23bsd999}, we achieve 
		\begin{align}\label{23bsd999-2}
		\mathbb{E}\|\bar{\z}^{(k+1)}\|^2  \le &  \mathbb{E}\|\bar{\z}^{(k)}\|^2 \hspace{-0.5mm}-\hspace{-0.5mm} \gamma \big(\mathbb{E} f(\bar{x}^{(k)}) \hspace{-0.5mm}-\hspace{-0.5mm} f(x^\star)\big) \hspace{-0.5mm}+\hspace{-0.5mm} \frac{3L\gamma c^2}{2n} \|X_{R,u}\|^2 \mathbb{E}\|\check{\vz}^{(k)}\|_F^2 \hspace{-0.5mm}+\hspace{-0.5mm} \frac{\gamma^2 \sigma^2}{n} \nonumber \\
		\le & \mathbb{E}\|\bar{\z}^{(k)}\|^2 \hspace{-0.5mm}-\hspace{-0.5mm} \gamma \big(\mathbb{E} f(\bar{x}^{(k)}) \hspace{-0.5mm}-\hspace{-0.5mm} f(x^\star)\big) \hspace{-0.5mm}+\hspace{-0.5mm} \frac{3L\gamma c^2}{2n} \|X_{R}\|^2 \mathbb{E}\|\check{\vz}^{(k)}\|_F^2 \hspace{-0.5mm}+\hspace{-0.5mm} \frac{\gamma^2 \sigma^2}{n}
		\end{align}
		where the last inequality holds because 
		\begin{align}\label{87z}
		\|X_{R,u}\| = \| 
		\ba{cc}
		I_n & 0
		\ea
		X_R \| \le \|\ba{cc}
		I_n & 0
		\ea\| \cdot  \|X_R\| = \| X_R \|.
		\end{align}
		By setting $c^2 = \|X_L\|^2$ and recalling that $\|X_L\|\|X_R\| \le \bar{\lambda}_n^{-1/2}$ from Lemma \ref{lm-decom}, we achieve \eqref{23bsd999-3}.
		
	\end{proof}
	
	\subsection{Proof of Lemma \ref{lm-consensus}}
	\label{app-gc-consensus}
	{
		\begin{proof}
			From the second line in \eqref{transformed-error-dynamics}, it holds that 
			\begin{align}
			\check{\vz}^{(k+1)} = D_1 \check{\vz}^{(k)} - \frac{\gamma}{c} \check{\vg}^{(k)} - \frac{\gamma}{c}\check{\vs}^{(k)}
			\end{align}
			We next introduce $\beta_1 = \|D_1\|$. With \eqref{D1-norm2}, we know that $\beta_1 = \bar{\lambda}^{1/2}_2$. By taking mean-square for both sides of the above recursion, we achieve
			\begin{align}\label{znx6515}
			&\ \mathbb{E}[\|\check{\vz}^{(k+1)}\|_F^2|\cF^{(k)}] \nonumber \\
			=&\ \|D_1 \check{\vz}^{(k)} - \frac{\gamma}{c}\check{\vg}^{(k)}\|_F^2 + \frac{\gamma^2}{c^2} \mathbb{E}\|\check{\vs}^{(k)}\|^2_F \nonumber \\
			\overset{(a)}{\le}&\ \frac{1}{t}\|D_1\|^2 \|\check{\vz}^{(k)}\|_F^2 + \frac{\gamma^2}{(1-t)c^2}\|\check{\vg}^{(k)}\|_F^2 + \frac{\gamma^2}{c^2}\|\check{\vs}^{(k)}\|^2_F \nonumber \\
			\overset{(b)}{\le}&\ \beta_1 \|\check{\vz}^{(k)}\|_F^2 + \frac{\gamma^2 }{(1-\beta_1)c^2}\|\check{\vg}^{(k)}\|_F^2 + \frac{\gamma^2}{c^2}\|\check{\vs}^{(k)}\|^2_F
			\end{align}
			where inequality (a) holds because of the Jensen's inequality for any $t\in (0,1)$, and inequality (b) holds by letting $t = \beta_1 = \|D_1\|$. Next we bound $\|\check{\vg}^{(k)}\|_F^2$ and $\|\check{\vs}^{(k)}\|^2_F$. Recall the definition of $\check{\vg}^{(k)}$ in \eqref{g-check}, we have
			\begin{align}\label{6askd}
			&\ \|\check{\vg}^{(k)}\|_F^2 \nonumber \\
			=&\ \|\big(X_{L, \ell} Q_R  + X_{L,r} Q_R (I - \bar{\Lambda}_R)^{\frac{1}{2}} \big) \bar{\Lambda}_R Q_R^T (\nabla f(\vx^{(k)}) - \nabla f(\vx^\star))\|_F^2 \nonumber \\
			\le&\ 2 \|X_{L, \ell} Q_R \bar{\Lambda}_R Q_R^T (\nabla f(\vx^{(k)}) - \nabla f(\vx^\star))\|^2 + 2\|X_{L,r} Q_R (I - \bar{\Lambda}_R)^{\frac{1}{2}}  \bar{\Lambda}_R Q_R^T (\nabla f(\vx^{(k)}) - \nabla f(\vx^\star))\|^2_F \nonumber \\
			\le&\ 2 \|X_{L, \ell}\|^2 \|Q_R \bar{\Lambda}_R Q_R^T\|^2 \|\nabla f(\vx^{(k)}) - \nabla f(\vx^\star)\|_F^2 \nonumber \\
			& \quad + 2 \|X_{L,r}\|^2  \|Q_R (I - \bar{\Lambda}_R)^{\frac{1}{2}}  \bar{\Lambda}_R Q_R^T\|^2 \|\nabla f(\vx^{(k)}) - \nabla f(\vx^\star)\|^2_F
			\end{align}
			First, it is easy to verify that $\|X_{L, \ell}\| \le \|X_L\|$ and $\|X_{L, r}\|\le \|X_L\|$. For example, it holds that 
			\begin{align}\label{znb213987}
			\|X_{L,\ell}\| = \| X_L 
			\ba{c}
			I_n \\
			0
			\ea\| \le \|X_L\| \left\|\ba{c}
			I_n \\
			0
			\ea \right\| = \| X_L \|.
			\end{align}
			Second, quantity $\|Q_R \bar{\Lambda}_R Q_R^T\|^2$ can be bounded as 
			\begin{align}\label{znb213987-1}
			\|Q_R \bar{\Lambda}_R Q_R^T\|^2 & = \lambda_{\max}(Q_R \bar{\Lambda}^2_R Q_R^T) \nonumber \\
			&= \lambda_{\max}(0 \cdot \frac{1}{n}\mathds{1}\mathds{1}^T + Q_R \bar{\Lambda}^2_R Q_R^T) \nonumber \\
			&= \lambda_{\max}(Q \Lambda^\prime Q^T)  = \bar{\lambda}_2^2
			\end{align}
			where $Q:= [\frac{1}{\sqrt{n}}\mathds{1} \ Q_R]\in \RR^{n\times n}$ is defined in \eqref{zbnzb0zbzbz09} and is an orthonormal matrix, and $\Lambda^\prime = \mathrm{diag}\{0, \bar{\Lambda}_R^2\}$. Apparently, the largest eigenvalue of $Q \Lambda^\prime Q^T$ is $\bar{\lambda}^2_2(W)$, which is denoted as $\bar{\lambda}_2^2$.  Similarly, we can derive 
			\begin{align}\label{znb213987-2}
			\|Q_R (I - \bar{\Lambda}_R)^{\frac{1}{2}}  \bar{\Lambda}_R Q_R^T\|^2 = \lambda_{\max}\{(1-\bar{\lambda}_i)\bar{\lambda}_i^2\} \le (1-\bar{\lambda}_n)\bar{\lambda}_2^2 \le \bar{\lambda}_2^2
			\end{align}
			because $\bar{\lambda}_i := {\lambda}_i(\bar{W})$ and $\bar{\lambda}_2 \ge \cdots \ge \bar{\lambda}_n > 0$. Finally, quantity $\|\nabla f(\vx^{(k)}) - \nabla f(\vx^\star) \|_F^2$ can be bounded as 
			\begin{align}
			\|\nabla f(\vx^{(k)}) - \nabla f(\vx^\star) \|_F^2 &= \|\nabla f(\vx^{(k)}) - \nabla f(\bar{\vx}^{(k)}) + \nabla f(\bar{\vx}^{(k)}) - \nabla f(\vx^\star) \|_F^2 \nonumber \\
			&\le 2\|\nabla f(\vx^{(k)}) - \nabla f(\bar{\vx}^{(k)})\|_F^2 + 2 \|\nabla f(\bar{\vx}^{(k)}) - \nabla f(\vx^\star)\|_F^2  \nonumber \\
			&\overset{(a)}{\le} 2L^2\|\vx^{(k)} - \bar{\vx}^{(k)}\|_F^2 + 4nL(f(\bar{x}^{(k)}) - f(x^\star)) \nonumber \\
			&\overset{\eqref{zzznnn}}{\le} 2 c^2 L^2 \|X_R\|^2 \|\check{\vz}^{(k)}\|_F^2  + 4nL(f(\bar{x}^{(k)}) - f(x^\star))  \label{237sdbns7-3}
			\end{align}
			where (a) holds because $f(x)$ is convex and $L$-smooth. Substituting \eqref{znb213987}--\eqref{237sdbns7-3} into \eqref{6askd}, we achieve
			\begin{align}\label{6askd-2}
			\|\check{\vg}^{(k)}\|_F^2 \le  8 c^2 L^2 \bar{\lambda}_2^2 \|X_{L}\|^2 \|X_R\|^2 \|\check{\vz}^{(k)}\|_F^2  + 16nL\bar{\lambda}_2^2\|X_L\|^2 (f(\bar{x}^{(k)}) - f(x^\star)).
			\end{align}
			Next we bound $\|\check{\vs}^{(k)}\|^2_F$. Recalling the definition of $\check{\vs}^{(k)}$ in \eqref{s-check}, we have
			\begin{align}\label{zbq08az}
			\|\check{\vs}^{(k)}\|_F^2 &\le 2\|X_{L, \ell} \|^2 \|Q_R \bar{\Lambda}_R Q_R^T\|^2 \| \vs^{(k)}\|_F^2  + 2 \|X_{L,r}\|^2 \|Q_R (I - \bar{\Lambda}_R)^{\frac{1}{2}} \bar{\Lambda}_R Q_R^T\|^2\|\vs^{(k)}\|^2_F \nonumber \\
			&\overset{(a)}{\le} 4 \|X_L\|^2 \bar{\lambda}_2^2 \|\vs^{(k)}\|^2_F \nonumber \\
			&\le 4n\|X_L\|^2 \bar{\lambda}_2^2\sigma^2
			\end{align}
			where inequality (a) holds because of \eqref{znb213987}--\eqref{znb213987-2}. Substituting \eqref{6askd-2} and \eqref{zbq08az} into \eqref{znx6515}, and taking expectation over the filtration $\cF^{(k)}$, we achieve 
			\begin{align}\label{znx6515-2}
			&\ \mathbb{E}\|\check{\vz}^{(k+1)}\|_F^2\nonumber \\
			\le &\ \big(\beta_1 + \frac{8  L^2 \bar{\lambda}_2^2 \gamma^2 \|X_L\|^2\|X_R\|^2}{(1-\beta_1)} \big)\mathbb{E}\|\check{\vz}^{(k)}\|_F^2 + 
			\frac{16nL \gamma^2 \bar{\lambda}_2^2\|X_L\|^2}{c^2(1-\beta_1)}(\mathbb{E}f(\bar{x}^{(k)}) - f(x^\star)) + \frac{4n\|X_L\|^2 \bar{\lambda}_2^2\gamma^2\sigma^2}{c^2} \nonumber \\
			\overset{(a)}{=} &\ \big(\beta_1 + \frac{8  L^2 \bar{\lambda}_2^2 \gamma^2 \|X_L\|^2\|X_R\|^2}{(1-\beta_1)} \big)\mathbb{E}\|\check{\vz}^{(k)}\|_F^2 + 
			\frac{16nL \gamma^2 \bar{\lambda}_2^2}{1-\beta_1}(\mathbb{E}f(\bar{x}^{(k)}) - f(x^\star)) + 4n \bar{\lambda}_2^2\gamma^2\sigma^2 \nonumber \\
			\overset{(b)}{\le} &\ \big(\frac{1+\beta_1}{2}\big)\mathbb{E}\|\check{\vz}^{(k)}\|_F^2 + \frac{16nL \gamma^2 \bar{\lambda}_2^2}{1-\beta_1}(\mathbb{E}f(\bar{x}^{(k)}) - f(x^\star)) + 4n \bar{\lambda}_2^2\gamma^2\sigma^2
			\end{align}
			where (a) holds by setting $c^2 = \|X_L\|^2$, and (b) holds by setting $\gamma$ sufficiently small such that 
			\begin{align}\label{xc762535rzz}
			\beta_1 + \frac{8\bar{\lambda}_2^2\gamma^2L^2}{1-\beta_1}\|X_L\|^2\|X_R\|^2 \le \frac{1+\beta_1}{2}.
			\end{align}
			To satisfy the above inequality, it is enough to set (recall \eqref{bound-XLXR})
			\begin{align}
			\gamma \le  \frac{(1-\beta_1)\bar{\lambda}_n^{1/2}}{4\bar{\lambda}_2 L}. 
			\end{align}

	\end{proof}}
	
	\subsection{Proof of Lemma \ref{lm-ergodic-consensus}}
	\label{app-lm-ergodic-consensus}
	
	{
		\begin{proof}
			Keep iterating \eqref{xxcznzxnxcn} we achieve for $k=1,2,\ldots$ that
			\begin{align}\label{xz675123aa}
			\mathbb{E}\|\check{\vz}^{(k)}\|_F^2 \le& \Big(\frac{1+\beta_1}{2}\Big)^k \mathbb{E}\|\check{\vz}^{(0)}\|_F^2 + \frac{8n\gamma^2\bar{\lambda}_2^2\sigma^2}{1-\beta_1}  + \frac{16 n \gamma^2 \bar{\lambda}_2^2 L}{1-\beta_1}\sum_{\ell=0}^{k-1} \Big(\frac{1+\beta_1}{2}\Big)^{k-1-\ell} \big(\mathbb{E} f(\bar{x}^{(\ell)}) - f(x^\star) \big)
			\end{align}
			We let $C = \mathbb{E}\|\check{\vz}^{(0)}\|_F^2$ be a positive constant. By taking average over $k=1,2,\ldots, T$, we achieve
			\begin{align}
			\frac{1}{T}\sum_{k=1}^T \mathbb{E}\|\check{\vz}^{(k)}\|_F^2 
			\le & \frac{C}{T}\sum_{k=1}^T \Big(\frac{1+\beta_1}{2}\Big)^k + \frac{8n\gamma^2\bar{\lambda}_2^2\sigma^2}{1-\beta_1} \nonumber \\
			& \quad + \frac{16n \gamma^2\bar{\lambda}_2^2L}{(1-\beta_1)T}\sum_{k=1}^T\sum_{\ell=0}^{k-1} \Big(\frac{1+\beta_1}{2}\Big)^{k-1-\ell} \big(\mathbb{E} f(\bar{x}^{(\ell)}) - f(x^\star) \big) \nonumber \\
			\le& \frac{C}{T}\sum_{k=1}^T \Big(\frac{1+\beta_1}{2}\Big)^k + \frac{8 n\gamma^2\bar{\lambda}_2^2\sigma^2}{1-\beta_1} \nonumber \\
			& \quad + \frac{16 n \gamma^2\bar{\lambda}_2^2 L}{(1-\beta_1)T}\sum_{\ell=0}^{T-1}\Big[\sum_{k=\ell+1}^{T} \Big(\frac{1+\beta_1}{2}\Big)^{k-1-\ell}\Big]\big(\mathbb{E} f(\bar{x}^{(\ell)}) - f(x^\star) \big) \nonumber \\
			\le& \frac{2C}{T(1-\beta_1)} + \frac{32 n \gamma^2 \bar{\lambda}_2^2 L }{(1-\beta_1)^2T}\sum_{k=0}^{T-1}\big(\mathbb{E} f(\bar{x}^{(k)}) - f(x^\star) \big) +  \frac{8 n\gamma^2\bar{\lambda}_2^2\sigma^2}{1-\beta_1} \nonumber \\
			\le& \frac{2C}{T(1-\beta_1)} + \frac{32 n \gamma^2\bar{\lambda}_2^2L}{(1-\beta_1)^2T}\sum_{k=0}^{T}\big(\mathbb{E} f(\bar{x}^{(k)}) - f(x^\star) \big) +  \frac{8 n\gamma^2\bar{\lambda}_2^2\sigma^2}{1-\beta_1}
			\end{align}
			Since $\frac{1}{T+1}\sum_{k=0}^T \mathbb{E}\|\check{\vz}^{(k)}\|_F^2 = \big(\sum_{k=1}^T \mathbb{E}\|\check{\vz}^{(k)}\|_F^2  + \mathbb{E}\|\check{\vz}^{(0)}\|_F^2)/(T+1)$ and $1 < \frac{1}{1-\beta_1}$, we achieve the result \eqref{eq:ergodic-consensus}.
			
		\end{proof}
	}
	
	\subsection{Proof of Theorem \ref{thm-generally-convex}}
	\label{app-thm-gc}
	
	
	{
	\begin{proof}
		From \eqref{23bsd999-3}, we have
		\begin{align}
		\mathbb{E} f(\bar{x}^{(k)}) \hspace{-0.5mm}-\hspace{-0.5mm} f(x^\star) \le \frac{\mathbb{E}\|\bar{\z}^{(k)}\|^2 - \mathbb{E}\|\bar{\z}^{(k+1)}\|^2}{\gamma} + \frac{3L}{2n \bar{\lambda}_n} \mathbb{E}\|\check{\vz}^{(k)}\|_F^2 + \frac{\gamma \sigma^2}{n}
		\end{align}
		By taking average over $k=0,1,\ldots,T$, we have
		\begin{align}\label{23756xcb0}
		&\ \frac{1}{T+1}\sum_{k=0}^T\big(\mathbb{E} f(\bar{x}^{(k)}) \hspace{-0.5mm}-\hspace{-0.5mm} f(x^\star)\big) \nonumber \\
		\le&\ \frac{\mathbb{E}\|\bar{\z}^{(0)}\|^2}{\gamma (T+1)} + \frac{3L}{2n \bar{\lambda}_n(T+1)}\sum_{k=0}^T \mathbb{E}\|\check{\vz}^{(k)}\|_F^2 + \frac{\gamma \sigma^2}{n} \nonumber \\
		\overset{\eqref{eq:ergodic-consensus}}{\le}&\  \frac{48 L^2 \gamma^2 \bar{\lambda}_2^2}{ (T+1)(1-\beta_1)^2 \bar{\lambda}_n}\sum_{k=0}^T\big(\mathbb{E} f(\bar{x}^{(k)}) \hspace{-0.5mm}-\hspace{-0.5mm} f(x^\star)\big) + \frac{12L\gamma^2\bar{\lambda}_2^2\sigma^2}{ \bar{\lambda}_n(1-\beta_1)}  + \frac{\gamma \sigma^2}{n} \nonumber \\
		&\ + \frac{\mathbb{E}\|\bar{\z}^{(0)}\|^2}{\gamma(T+1)} + \frac{9L\mathbb{E}\|\check{\vz}^{(0)}\|^2}{2n(T+1)(1-\beta_1)\bar{\lambda}_n}.
		\end{align}
		If $\gamma$ is sufficiently small such that 
		\begin{align}\label{23zvwd85123}
		\frac{48 L^2 \gamma^2 \bar{\lambda}_2^2}{(1-\beta_1)^2 \bar{\lambda}_n} \le \frac{1}{2}, 
		\end{align}
		inequality \eqref{23756xcb0} becomes 
		\begin{align}\label{lx7}
		\frac{1}{T+1}\sum_{k=0}^T\big(\mathbb{E} f(\bar{x}^{(k)}) \hspace{-0.5mm}-\hspace{-0.5mm} f(x^\star)\big) &\le \frac{24L\gamma^2\bar{\lambda}_2^2\sigma^2}{(1-\beta_1) \bar{\lambda}_n}  + \frac{2\gamma \sigma^2}{n} + \frac{2\mathbb{E}\|\bar{\z}^{(0)}\|^2}{\gamma(T+1)} + \frac{9L\mathbb{E}\|\check{\vz}^{(0)}\|^2}{n(T+1)(1-\beta_1)\bar{\lambda}_n} \nonumber \\
		&\overset{(a)}{\le} \frac{24L\gamma^2\bar{\lambda}_2^2\sigma^2}{(1-\beta_1) \bar{\lambda}_n}  + \frac{2\gamma \sigma^2}{n} + \frac{2\mathbb{E}\|\bar{\z}^{(0)}\|^2}{\gamma(T+1)} + \frac{18 L\gamma^2 \bar{\lambda}_2^2\|\nabla f(\vx^\star)\|_F^2}{n(T+1)(1-\beta_1) (1-\beta) \bar{\lambda}_n}
		\end{align}
		where (a) holds because $\mathbb{E}\|\check{\vz}^{(0)}\|^2 \leq \frac{\gamma^2 \bar{\lambda}_2^2\|\nabla f(\vx^\star)\|_F^2}{1-\bar{\lambda}_2}$ (see Proposition \ref{remark-z0}) and $ \bar{\lambda}_2 \le (1+\beta)/2$. 
		To satisfy \eqref{23zvwd85123}, it is enough to let 
		\begin{align}
		\gamma \le \frac{(1-\beta_1) \bar{\lambda}^{1/2}_n}{10 L \bar{\lambda}_2}.
		\end{align}
		The way to choose step-size $\gamma$ is adapted from \cite[Lemma 15]{koloskova2020unified}. For simplicity, we let
		\begin{align}\label{sndsnnds}
		B^{(k)} = \mathbb{E} f(\bar{x}^{(k)}) \hspace{-0.5mm}-\hspace{-0.5mm} f(x^\star), \ r_0 = 2 \mathbb{E}\|\bar{\z}^{(0)}\|^2, \ r_1 = \frac{2\sigma^2}{n}, \ r_2 = \frac{24L\bar{\lambda}_2^2\sigma^2}{(1-\beta_1) \bar{\lambda}_n}, \ r_3 = \frac{18 L\bar{\lambda}_2^2\|\nabla f(\vx^\star)\|_F^2}{n(1-\beta_1)(1-\beta)\bar{\lambda}_n}
		\end{align}
		and inequality \eqref{lx7} becomes
		\begin{align}\label{sndsnnds-2}
		\frac{1}{T+1}\sum_{k=0}^T B^{(k)} \le \frac{r_0}{(T+1)\gamma} + r_1 \gamma
		+ r_2 \gamma^2 + \frac{r_3 \gamma^2 }{T+1}.
		\end{align}
		Now we let 
		\begin{align}\label{zn238sd6zzz1}
		\gamma = \min\big\{\frac{1}{4L}, \frac{(1-\beta_1)\bar{\lambda}_n^{1/2}}{10 L \bar{\lambda}_2}, \left(\frac{r_0}{r_1(T+1)}\right)^{\frac{1}{2}}, \left(\frac{r_0}{r_2(T+1)}\right)^{\frac{1}{3}}, \left(\frac{r_0}{r_3}\right)^{\frac{1}{3}}\big\}.
		\end{align}
		\begin{itemize}
			
			\item If $\frac{1}{4L}$ is the smallest, we let $\gamma = \frac{1}{4L}$. With $\gamma \le \left(\frac{r_0}{r_1(T+1)}\right)^{\frac{1}{2}}$, $\gamma \le \big(\frac{r_0}{r_2(T+1)}\big)^{\frac{1}{3}}$, and $\gamma \le \big(\frac{r_0}{r_3}\big)^{\frac{1}{3}}$, \eqref{sndsnnds-2} becomes 
			\begin{align}\label{xncxn-cn-sdf-test-0}
			\frac{1}{T+1}\sum_{k=0}^T B^{(k)} \le \frac{4 L  r_0}{T+1} + \Big(\frac{r_0 r_1}{T+1}\Big)^{\frac{1}{2}} + r_2^{\frac{1}{3}}\Big(\frac{r_0}{T+1}\Big)^{\frac{2}{3}} + \frac{r_3^{\frac{1}{3}}r_0^{\frac{2}{3}}}{T+1}.
			\end{align}
			
			\item If $\frac{(1-\beta_1)\bar{\lambda}_n^{1/2}}{10 L \bar{\lambda}_2}$ is the smallest, we let $\gamma = \frac{(1-\beta_1)\bar{\lambda}_n^{1/2}}{10 L \bar{\lambda}_2}$. With $\gamma \le \left(\frac{r_0}{r_1(T+1)}\right)^{\frac{1}{2}}$, $\gamma \le \big(\frac{r_0}{r_2(T+1)}\big)^{\frac{1}{3}}$, and $\gamma \le \big(\frac{r_0}{r_3}\big)^{\frac{1}{3}}$, \eqref{sndsnnds-2} becomes 
			\begin{align}\label{xncxn-cn-sdf-test}
			\frac{1}{T+1}\sum_{k=0}^T B^{(k)} \le \frac{10 L \bar{\lambda}_2 r_0}{(1-\beta_1)(T+1)\bar{\lambda}^{1/2}_n} + \Big(\frac{r_0 r_1}{T+1}\Big)^{\frac{1}{2}} + r_2^{\frac{1}{3}}\Big(\frac{r_0}{T+1}\Big)^{\frac{2}{3}} + \frac{r_3^{\frac{1}{3}}r_0^{\frac{2}{3}}}{T+1}.
			\end{align}
			
			\item If $\left(\frac{r_0}{r_1(T+1)}\right)^{\frac{1}{2}}$ is the smallest, we let $\gamma = \big(\frac{r_0}{r_1(T+1)}\big)^{\frac{1}{2}}$.  With $\gamma \le \big(\frac{r_0}{r_2(T+1)}\big)^{\frac{1}{3}}$ and $\gamma \le \big(\frac{r_0}{r_3}\big)^{\frac{1}{3}}$, \eqref{sndsnnds-2} becomes 
			\begin{align}\label{xncxn-cn-4-test}
			\frac{1}{T+1}\sum_{k=0}^T B^{(k)} \le 2 \Big(\frac{r_0 r_1}{T+1}\Big)^{\frac{1}{2}} + r_2^{\frac{1}{3}}\Big(\frac{r_0}{T+1}\Big)^{\frac{2}{3}} + \frac{r_3^{\frac{1}{3}}r_0^{\frac{2}{3}}}{T+1}.
			\end{align}	
			
			\item If $\left(\frac{r_0}{r_2(T+1)}\right)^{\frac{1}{3}}$ is the smallest, we let $\gamma = \left(\frac{r_0}{r_2(T+1)}\right)^{\frac{1}{3}}$. With $\gamma \le \left(\frac{r_0}{r_1(T+1)}\right)^{\frac{1}{2}}$ and $\gamma \le \big(\frac{r_0}{r_3}\big)^{\frac{1}{3}}$, \eqref{sndsnnds-2} becomes 
			\begin{align}\label{xncxn-cn-3-test}
			\frac{1}{T+1}\sum_{k=0}^T B^{(k)} \le 2r_2^{\frac{1}{3}} \Big(\frac{r_0}{T+1}\Big)^{\frac{2}{3}} + \Big(\frac{r_0 r_1}{T+1}\Big)^{\frac{1}{2}}  +  \frac{r_3^{\frac{1}{3}}r_0^{\frac{2}{3}}}{T+1}.
			\end{align}

			\item If $\big(\frac{r_0}{r_3}\big)^{\frac{1}{3}}$ is the smallest, we let $\gamma = \big(\frac{r_0}{r_3}\big)^{\frac{1}{3}}$. With $\gamma \le \left(\frac{r_0}{r_1(T+1)}\right)^{\frac{1}{2}}$ and $\gamma \le \big(\frac{r_0}{r_2(T+1)}\big)^{\frac{1}{3}}$, \eqref{sndsnnds-2} becomes 
			\begin{align}\label{xncxn-cn-sdf-test-2}
			\frac{1}{T+1}\sum_{k=0}^T B^{(k)} \le \frac{2r_3^{\frac{1}{3}}r_0^{\frac{2}{3}}}{T+1} + \Big(\frac{r_0 r_1}{T+1}\Big)^{\frac{1}{2}} + r_2^{\frac{1}{3}}\Big(\frac{r_0}{T+1}\Big)^{\frac{2}{3}}.
			\end{align}
		\end{itemize}
		Combining \eqref{xncxn-cn-3-test}, \eqref{xncxn-cn-4-test} and \eqref{xncxn-cn-sdf-test}, we have 
		\begin{align}\label{eqn:gjdosfjoasd}
		\frac{1}{T+1}\sum_{k=0}^T B^{(k)} \le \frac{4 L  r_0}{T+1} + \frac{10 L \bar{\lambda}_2 r_0}{(1-\beta_1)(T+1)\bar{\lambda}^{1/2}_n} + 2\Big(\frac{r_0 r_1}{T+1}\Big)^{\frac{1}{2}} + 2 r_2^{\frac{1}{3}}\Big(\frac{r_0}{T+1}\Big)^{\frac{2}{3}} + \frac{2r_3^{\frac{1}{3}}r_0^{\frac{2}{3}}}{T+1}.
		\end{align}
		Substituting constants $r_0$, $r_1$, $r_2$ and $r_3$ into the above inequality and regarding $\mathbb{E}\|\bar{\z}^{(0)}\|^2 = O(1)$ and $\|\nabla f(\vx)\|_F^2=O(n)$, we achieve
		\begin{align}\label{xbsd87-0}
		\frac{1}{T+1}\sum_{k=0}^T B^{(k)} &= O\Big( \frac{\sigma}{\sqrt{nT}} + 
		\frac{\sigma^{\frac{2}{3}}\bar{\lambda}^{\frac{2}{3}}_2}{(1-\beta_1)^{\frac{1}{3}}T^{\frac{2}{3}} \bar{\lambda}_n^{\frac{1}{3}}}  + \frac{\bar{\lambda}_2}{(1-\beta_1)T \bar{\lambda}^{\frac{1}{2}}_n} + \frac{\bar{\lambda}^{\frac{2}{3}}_2}{(1-\beta_1)^{\frac{1}{3}} (1-\beta)^{\frac{1}{3}} T \bar{\lambda}^{\frac{1}{3}}_n} + \frac{1}{T}\Big)
		\end{align}
		With $\beta_1^2 = \bar{\lambda}_2 \le (1+\beta)/2$, we have
		\begin{align}\label{znnbbb}
		1 - \beta_1 = \frac{1 - \beta_1^2}{1 + \beta_1} \ge \frac{1-\beta}{2(1 + \beta_1)} \ge \frac{1-\beta}{4}.
		\end{align}
		Substituting \eqref{znnbbb} to \eqref{zn238sd6zzz1} and \eqref{xbsd87-0} and regarding $\bar{\lambda}_2$ and $\bar{\lambda}_n$ as constants (note that $\bar{\lambda}_n$ is bounded away from $0$. For example, if $\bar{W} = (3I + W)/4$, we have $\bar{\lambda}_2 \ge \bar{\lambda}_n \ge 1/2$), we have the result in Theorem \ref{thm-generally-convex}. 
	\end{proof}}
	
	\section{Convergence Analysis for Strongly-Convex Scenario}
	
	\subsection{Proof of Lemma \ref{lm-sc-descent}}
	\label{app-sc-descent}
	\begin{proof}
		Since each $f_i(x)$ is strongly convex, it holds that 
		\begin{align}
		f_i(x) - f_i(y) + \frac{\mu}{2}\|x - y\|^2 \le \langle \nabla f_i(x), x - y \rangle, \quad \forall~x,y\in \mathbb{R}^d
		\end{align}
		Let $x = x_i^{(k)}$ and $y = x^\star$, we have
		\begin{align}
		f_i(x_i^{(k)}) - f_i(x^\star) + \frac{\mu}{2}\|x_i^{(k)} - x^\star\|^2 \le \langle \nabla f_i(x_i^{(k)}), x_i^{(k)} - x^\star \rangle.
		\end{align}
		Following arguments from \eqref{shdshsd} to \eqref{23nbsdn-1}, and replacing the bound in \eqref{23nbsdn-2} with \begin{align}\label{23nbsdn-3}
		& \frac{2\gamma}{n}\sum_{i=1}^n \langle \bar{x}^{{(k)}} - x^\star, \nabla f_i(x_i^{(k)}) - \nabla f_i(x^\star) \rangle \nonumber \\
		=&\ \frac{2\gamma}{n}\sum_{i=1}^n \langle \bar{x}^{{(k)}} - x_i^{(k)}, \nabla f_i(x_i^{(k)}) \rangle  + \frac{2\gamma}{n}\sum_{i=1}^n \langle x_i^{(k)} - x^\star, \nabla f_i(x^{(k)})  \rangle \nonumber \\
		\overset{(a)}{\ge}&\ \frac{2\gamma}{n}\sum_{i=1}^n \Big( f_i(\bar{x}^{{(k)}}) - f_i(x_i^{(k)}) - \frac{L}{2}\|\bar{x}^{{(k)}} - x_i^{(k)}\|^2 \Big) + \frac{2\gamma}{n}\sum_{i=1}^n \Big(f_i(x_i^{(k)}) - f_i(x^\star) + \frac{\mu}{2}\|x_i^{(k)} - x^\star\|^2 \Big)   \nonumber \\
		=&\ \frac{2\gamma}{n}\sum_{i=1}^n\Big( f_i(\bar{x}^{{(k)}}) - f_i(x^\star) \Big) - \frac{\gamma L}{n}\|\bar{\vx}^{{(k)}} - \vx^{(k)}\|_F^2 + \frac{\gamma \mu}{n}\|{\vx}^{{(k)}} - \vx^{\star}\|_F^2\nonumber \\
		\ge&\ 2\gamma \big( f(\bar{x}^{{(k)}}) - f(x^\star)\big) - \frac{\gamma (L + \mu)}{n}\|\bar{\vx}^{{(k)}} - \vx^{(k)}\|_F^2 + \frac{\gamma \mu}{2 }\|\bar{x}^{{(k)}} - x^{\star}\|_F^2,
		\end{align}
		we achieve a slightly different bound from \eqref{shdshsd-2}:
		\begin{align}\label{shdshsd-sc}
		&\ \|\bar{x}^{{(k)}} - x^\star -  \frac{\gamma}{n}\sum_{i=1}^n \nabla f_i(x_i^{(k)})\|^2 \nonumber \\
		\le&\ (1-\frac{\gamma \mu}{2}) \|\bar{x}^{(k)} - x^\star\|^2 - 2\gamma(1 - 2L\gamma) \big( f(\bar{x}^{{(k)}}) - f(x^\star)\big) + \Big( \frac{\gamma (L + \mu)}{n} + \frac{2\gamma^2 L^2}{n}\Big)\|\bar{\vx}^{{(k)}} - \vx^{(k)}\|_F^2  \nonumber \\
		\le&\ (1-\frac{\gamma \mu}{2}) \|\bar{x}^{(k)} - x^\star\|^2 - \gamma \big( f(\bar{x}^{{(k)}}) - f(x^\star)\big) + \frac{5\gamma L}{2n}\|\bar{\vx}^{{(k)}} - \vx^{(k)}\|_F^2 
		\end{align}
		where the last inequality holds when $\gamma \le \frac{1}{4L}$. With \eqref{shdshsd-sc}, we can follow arguments \eqref{23bsd999}-\eqref{87z} to achieve the result in \eqref{23bsd999-sc}. 
	\end{proof}
	
	\subsection{Proof of Lemma \ref{lm-sc-ergodic-consenus}}
	\label{app-sc-ergodic-consensus-lm}
	
	\begin{proof}
		Recall from \eqref{xz675123aa} that 
		\begin{align}\label{xz675123aa-1}
		\mathbb{E}\|\check{\vz}^{(k)}\|_F^2 \le& \Big(\frac{1+\beta_1}{2}\Big)^k \mathbb{E}\|\check{\vz}^{(0)}\|_F^2 + \frac{8n\gamma^2\bar{\lambda}_2^2\sigma^2}{1-\beta_1} \nonumber \\
		& \quad + \frac{16 n \gamma^2 \bar{\lambda}_2^2 L}{1-\beta_1}\sum_{\ell=0}^{k-1} \Big(\frac{1+\beta_1}{2}\Big)^{k-1-\ell} \big(\mathbb{E} f(\bar{x}^{(\ell)}) - f(x^\star) \big)
		\end{align}
		By taking the weighted sum over $k=1,2,\ldots,T$, we achieve
		\begin{align}\label{1bsd0925}
		\sum_{k=1}^T h_k \mathbb{E}\|\check{\vz}^{(k)}\|_F^2  \le&\   \mathbb{E}\|\check{\vz}^{(0)}\|_F^2 \sum_{k=1}^T h_k \Big(\frac{1+\beta_1}{2}\Big)^k + \frac{8 n \gamma^2 \bar{\lambda}_2^2  \sigma^2}{1-\beta_1}\sum_{k=1}^T h_k \nonumber \\
		& \quad + \frac{16 n \gamma^2 \bar{\lambda}_2^2 L}{1-\beta_1} \sum_{k=1}^T h_k \sum_{\ell=0}^{k-1} \Big(\frac{1+\beta_1}{2}\Big)^{k-1-\ell} \big(\mathbb{E} f(\bar{x}^{(\ell)}) - f(x^\star) \big)
		\end{align}
		Since $h_k$ satisfies condition \eqref{h-condition}, 
		it holds (we define $B^{(\ell)} = \mathbb{E} f(\bar{x}^{(\ell)}) - f(x^\star)$) that 
		\begin{align}
		\sum_{k=1}^T h_k \sum_{\ell=0}^{k-1} \Big(\frac{1+\beta_1}{2}\Big)^{k-1-\ell} B^{(\ell)} &\le \big( 1 + \frac{1-\beta_1}{4} \big) \sum_{k=1}^T \sum_{\ell=0}^{k-1} h_\ell \Big[\big( 1 + \frac{1-\beta_1}{4} \big) \big(\frac{1+\beta_1}{2}\big)\Big]^{k-1-\ell} B^{(\ell)} \nonumber \\
		&\le 2 \sum_{k=1}^T \sum_{\ell=0}^{k-1} h_\ell \big(\frac{3+\beta_1}{4}\big)^{k-1-\ell} B^{(\ell)} \nonumber \\
		&\le 2 \sum_{\ell=0}^{T-1} h_{\ell} \sum_{k=\ell+1}^T \big(\frac{3+\beta_1}{4}\big)^{k-1-\ell} B^{(\ell)} \nonumber \\
		&\le \frac{8}{1-\beta_1} \sum_{\ell=0}^{T-1} h_{\ell} B^{(\ell)} \le \frac{8}{1-\beta_1} \sum_{\ell=0}^{T} h_{\ell} B^{(\ell)}
		\end{align}
		Substituting the above inequality into \eqref{1bsd0925}, we achieve 
		\begin{align}\label{1bsd0925-2}
		\sum_{k=1}^T h_k \mathbb{E}\|\check{\vz}^{(k)}\|_F^2  \le&\   C \sum_{k=1}^T h_k \Big(\frac{1+\beta_1}{2}\Big)^k + \frac{8 n\gamma^2 \bar{\lambda}_2^2\sigma^2}{1-\beta_1}\sum_{k=0}^T h_k  + \frac{128 n\gamma^2 \bar{\lambda}_2^2 L}{(1-\beta_1)^2} \sum_{\ell=0}^{T} h_{\ell} B^{(\ell)}
		\end{align}
		where $C = \mathbb{E}\|\check{\vz}^{(0)}\|_F^2$. Adding $h_0 C$ to both sides of \eqref{1bsd0925-2}, we achieve
		\begin{align}\label{1bsd0925-3}
		\sum_{k=0}^T h_k \mathbb{E}\|\check{\vz}^{(k)}\|_F^2  \le&\   C \sum_{k=0}^T h_k \Big(\frac{1+\beta_1}{2}\Big)^k + \frac{8 n\gamma^2 \bar{\lambda}_2^2  \sigma^2}{1-\beta_1}\sum_{k=0}^T h_k  + \frac{128 n\gamma^2 \bar{\lambda}_2^2  L}{(1-\beta_1)^2} \sum_{\ell=0}^{T} h_{\ell} B^{(\ell)}.
		\end{align}
		Furthermore, with condition \eqref{h-condition}, we have $h_k \le h_0 (1 + \frac{1-\beta_1}{4})^k$ for any $k=0,1,\ldots$. This implies
		\begin{align}\label{23bb}
		\sum_{k=0}^T h_k \Big(\frac{1+\beta_1}{2}\Big)^k \le h_0 \sum_{k=0}^T (1 + \frac{1-\beta_1}{4})^k  (\frac{1+\beta_1}{2})^k \le h_0 \sum_{k=0}^T (\frac{3+\beta_1}{4})^k \le \frac{4 h_0}{1-\beta_1}
		\end{align}
		Substituting \eqref{23bb} into \eqref{1bsd0925-3} and dividing both sides by $H_T = \sum_{k=0}^T h_k$, we achieve the final result in \eqref{1bsd0925-sc-lm}. 
	\end{proof}

	\subsection{Proof of Theorem \ref{thm2}}
	\label{app-thm-sc-convergence}
	The following proof is inspired by \cite{stich2019unified}. 
	
	\begin{proof}
		With descent inequality \eqref{23bsd999-sc}, we have 
		\begin{align}\label{23bsd999-sc-thm}
		\mathbb{E} f(\bar{x}^{(k)}) \hspace{-0.5mm}-\hspace{-0.5mm} f(x^\star) \le    (1 - \frac{\gamma \mu}{2})\frac{\mathbb{E} \|\bar{\z}^{(k)}\|^2}{\gamma} - \frac{\mathbb{E}\|\bar{\z}^{(k+1)}\|^2}{\gamma}  \hspace{-0.5mm}+\hspace{-0.5mm} \frac{5L }{2n \bar{\lambda}_n} \mathbb{E}\|\check{\vz}^{(k)}\|_F^2 \hspace{-0.5mm}+\hspace{-0.5mm} \frac{\gamma \sigma^2}{n}
		\end{align}
		Taking the weighted average over $k$, it holds that (we let $B^{(k)} = \mathbb{E} f(\bar{x}^{(k)}) - f(x^\star)$)
		\begin{align}
		\frac{1}{H_T}\sum_{k=0}^T h_k B^{(k)} \le  \frac{1}{H_T}\sum_{k=0}^T h_k \Big( \frac{(1 - \frac{\gamma \mu}{2})\mathbb{E} \|\bar{\z}^{(k)}\|^2}{\gamma} - \frac{\mathbb{E}\|\bar{\z}^{(k+1)}\|^2}{\gamma}  \Big) + \frac{5L}{2 n H_T \bar{\lambda}_n}\sum_{k=0}^T  h_k \mathbb{E}\|\check{\vz}^{(k)}\|_F^2  + \frac{\gamma \sigma^2}{n}.
		\end{align}
		If we let $h_{k} = (1 - \frac{\gamma \mu}{2}) h_{k+1}$ for $k=0,1,\ldots$, the above inequality becomes
		\begin{align}\label{znb2375zzz}
		\frac{1}{H_T}\sum_{k=0}^T h_k B^{(k)} \le \frac{h_0 \mathbb{E}\|\bar{\z}^{(0)}\|^2 }{H_T \gamma} + \frac{5L}{2 n H_T \bar{\lambda}_n}\sum_{k=0}^T  h_k \mathbb{E}\|\check{\vz}^{(k)}\|_F^2  + \frac{\gamma \sigma^2}{n}.
		\end{align}
		Since $h_{k} = (1 - \frac{\gamma \mu}{2}) h_{k+1}$, we have 
		\begin{align}
		h_k = h_{\ell} \big(\frac{1}{1 - \frac{\gamma \mu}{2}}\big)^{k - \ell}, \mbox{ for any } k\ge 0 \mbox{ and } 0\le \ell \le k.
		\end{align}
		If $\gamma$ is sufficiently small such that 
		\begin{align}\label{h-cond-2}
		\frac{1}{1 - \frac{\gamma \mu}{2}} \le 1 + \frac{1-\beta_1}{4}, \quad \mbox{(it is enough to set $\gamma \le \frac{1-\beta_1}{2\mu}$)}
		\end{align}
		then $\{h_k\}_{k=0}^\infty$ satisfy condition \eqref{h-condition}. As a result, we can substitute inequality \eqref{1bsd0925-sc-lm} into \eqref{znb2375zzz} to achieve 
		\begin{align}\label{znb2375zzz-2}
		\frac{1}{H_T}\sum_{k=0}^T h_k B^{(k)} \le \frac{h_0 \mathbb{E}\|\bar{\z}^{(0)}\|^2 }{H_T \gamma} +\frac{\gamma \sigma^2}{n} + \frac{10 L C h_0}{n H_T(1-\beta_1) \bar{\lambda}_n} + \frac{20 L \gamma^2 \bar{\lambda}_2^2 \sigma^2}{\bar{\lambda}_n(1-\beta_1)} + \frac{320 \gamma^2 \bar{\lambda}_2^2 L^2}{(1-\beta_1)^2 \bar{\lambda}_n} \frac{1}{H_T}\sum_{k=0}^T h_k B^{(k)}.
		\end{align}
		If $\gamma$ is sufficiently small such that 
		\begin{align}\label{xbnzbzbzbzz}
		\frac{320 \gamma^2 \bar{\lambda}_2^2 L^2}{(1-\beta_1)^2 \bar{\lambda}_n} \le \frac{1}{2}, \quad \mbox{(it is enough to set $\gamma \le \frac{1-\beta_1}{26L} \big(\frac{\bar{\lambda}_n^{1/2}}{\bar{\lambda}_2}\big)$)}
		\end{align}
		it holds that 
		\begin{align}\label{znb2375zzz-3}
		\frac{1}{H_T}\sum_{k=0}^T h_k B^{(k)} \le \frac{2h_0 \mathbb{E}\|\bar{\z}^{(0)}\|^2 }{H_T \gamma} + \frac{20 L C h_0}{n H_T(1-\beta_1) \bar{\lambda}_n} +\frac{2 \gamma \sigma^2}{n}  + \frac{40 L \gamma^2 \bar{\lambda}_2^2 \sigma^2}{\bar{\lambda}_n(1-\beta_1)}.
		\end{align}
		Since $H_T \ge h_T = h_0 (1 - \frac{\gamma \mu}{2})^{-T}$, we have
		\begin{align}\label{znb2375zzz-4}
		\frac{1}{H_T}\sum_{k=0}^T h_k B^{(k)} &\le \frac{2 \mathbb{E}\|\bar{\z}^{(0)}\|^2 }{\gamma}(1 - \frac{\gamma \mu}{2})^{T} + \frac{20 L C }{n (1-\beta_1) \bar{\lambda}_n} (1 - \frac{\gamma \mu}{2})^{T} +\frac{2 \gamma \sigma^2}{n}  + \frac{40 L \gamma^2 \bar{\lambda}_2^2 \sigma^2}{\bar{\lambda}_n(1-\beta_1)} \nonumber \\
		&\le \big(  \frac{2 \mathbb{E}\|\bar{\z}^{(0)}\|^2 }{\gamma} + \frac{20 L C }{n (1-\beta_1) \bar{\lambda}_n} \big) \exp(-\frac{\gamma \mu T}{2}) + \frac{2 \gamma \sigma^2}{n}  + \frac{40 L \gamma^2 \bar{\lambda}_2^2 \sigma^2}{\bar{\lambda}_n(1-\beta_1)} \nonumber \\
		&\le \big(  \frac{2 \mathbb{E}\|\bar{\z}^{(0)}\|^2 }{\gamma} + \frac{40 L  \gamma^2 \bar{\lambda}_2^2\|\nabla f(\vx^\star)\|_F^2}{ n(1-\beta_1) (1-\beta) \bar{\lambda}_n} \big) \exp(-\frac{\gamma \mu T}{2}) + \frac{2 \gamma \sigma^2}{n}  + \frac{40 L \gamma^2 \bar{\lambda}_2^2 \sigma^2}{\bar{\lambda}_n(1-\beta_1)} 
		\end{align}
		where the last inequality holds because $C = \mathbb{E}\|\check{\vz}^{(0)}\|^2_F \leq \frac{\gamma^2\bar{\lambda}_2^2\|\nabla f(\vx^\star)\|_F^2}{1-\bar{\lambda}_2}=O(\frac{n\gamma^2\bar{\lambda}_2^2}{1-\bar{\lambda}_2})$ (see Proposition \ref{remark-z0}) and $\bar{\lambda}_2\le (1+\beta)/2$, and $\gamma$ needs to satisfy condition \eqref{xbnzbzbzbzz}. Now we let 
		\begin{align}\label{zn238sd6zzz1-0}
		\gamma = \min\big\{\frac{1}{4L}, \frac{1-\beta_1}{26L} \Big(\frac{\bar{\lambda}_n^{1/2}}{\bar{\lambda}_2}\big), \frac{2\ln(2 n \mu \mathbb{E}\|\bar{\z}^{(0)}\|^2 T^2/[\sigma^2 (1-\beta)])}{\mu T}\big\}.
		\end{align}
		\begin{itemize}
		\item If $\frac{2\ln(2 n \mu \mathbb{E}\|\bar{\z}^{(0)}\|^2 T^2/[\sigma^2 (1-\beta)])}{\mu T}$ is smallest, we set 
		\begin{align}
		\gamma = \frac{2\ln(2 n \mu  \mathbb{E}\|\bar{\z}^{(0)}\|^2 T^2/[\sigma^2 (1-\beta)])}{\mu T} \quad \mbox{ so that } \quad \exp(-\frac{\gamma \mu T}{2}) = \frac{\sigma^2(1-\beta)}{2 n \mu \mathbb{E}\|\bar{\z}^{(0)}\|^2 T^2}
		\end{align}
		Substituting the above $\gamma$ into \eqref{znb2375zzz-4} and regarding $\mathbb{E}\|\bar{\z}^{(0)}\|^2 = O(1)$, we achieve  
		\begin{align}\label{znb2375zzz-5}
		\frac{1}{H_T}\sum_{k=0}^T h_k B^{(k)} = \tilde{O}\left( \frac{\sigma^2}{n T} + \frac{\sigma^2 \bar{\lambda}_2^2}{(1-\beta_1) T^2 \bar{\lambda}_n} \right).
		\end{align}
		
		\item If $\frac{1-\beta_1}{26L} \Big(\frac{\bar{\lambda}_n^{1/2}}{\bar{\lambda}_2}\big)$ is smallest, we set $\gamma = \frac{1-\beta_1}{26L} \Big(\frac{\bar{\lambda}_n^{1/2}}{\bar{\lambda}_2}\big)$. Since $\gamma \le \frac{1}{4L}$ and $\gamma \le \frac{2\ln(2 n \mu \mathbb{E}\|\bar{\z}^{(0)}\|^2 T^2/[\sigma^2 (1-\beta)])}{\mu T}$, \eqref{znb2375zzz-4} becomes 
		\begin{align}\label{znb2375zzz-6}
		\frac{1}{H_T}\sum_{k=0}^T h_k B^{(k)} = \tilde{O}\left(
		\frac{\bar{\lambda}_2}{(1-\beta_1)\bar{\lambda}^{\frac{1}{2}}_n}  \exp\{-(1-\beta_1) \big(\frac{\bar{\lambda}^{\frac{1}{2}}_n}{\bar{\lambda}_2}\big) T \}  + \frac{\sigma^2}{n T} + \frac{\sigma^2}{(1-\beta_1) T^2} \big(\frac{\bar{\lambda}^2_2}{\bar{\lambda}_n}\big) \right).
		\end{align}
		
		\item If $\frac{1}{4L}$ is smallest, we set $\gamma = \frac{1}{4L}$. Since $\gamma \le \frac{1-\beta_1}{26L} \Big(\frac{\bar{\lambda}_n^{1/2}}{\bar{\lambda}_2}\big)$ and $\gamma \le \frac{1}{4L}$ and $\gamma \le \frac{2\ln(2 n \mu \mathbb{E}\|\bar{\z}^{(0)}\|^2 T^2/[\sigma^2 (1-\beta)])}{\mu T}$, \eqref{znb2375zzz-4} becomes 
		\begin{align}\label{znb2375zzz-7}
		\frac{1}{H_T}\sum_{k=0}^T h_k B^{(k)} = \tilde{O}\left(\exp\{-T \}  + \frac{\sigma^2}{n T} + \frac{\sigma^2}{(1-\beta_1) T^2} \big(\frac{\bar{\lambda}^2_2}{\bar{\lambda}_n}\big) \right).
		\end{align}
		\end{itemize}
		Combining \eqref{znb2375zzz-5} -- \eqref{znb2375zzz-60}, substituting relation \eqref{znnbbb} to bound $1-\beta_1$, we achieve
		\begin{align}\label{znb2375zzz-60}
		\frac{1}{H_T}\sum_{k=0}^T h_k B^{(k)} = \tilde{O}\left(
		\frac{\sigma^2}{n T} + \frac{\sigma^2}{(1-\beta) T^2} \big(\frac{\bar{\lambda}^2_2}{\bar{\lambda}_n}\big)  + \frac{\bar{\lambda}_2\exp\{-(1-\beta) \big(\frac{\bar{\lambda}^{\frac{1}{2}}_n}{\bar{\lambda}_2}\big) T \}}{(1-\beta)\bar{\lambda}^{\frac{1}{2}}_n} + \exp\{-T\} \right).
		\end{align}
		Ignoring constants $\bar{\lambda}_2$ and $\bar{\lambda}_n$ (these quantities can be regarded as  constants when $\bar{\lambda}_n$ is bounded away from zero. For example, if $\bar{W} = (3I + W)/4$, we have $\bar{\lambda}_2 \ge \bar{\lambda}_n \ge 1/2$. ), and recalling the relation in \eqref{znnbbb}, we achieve the result in \eqref{thm2-results}. 
	\end{proof}
	
	\section{Proof of Theorem \ref{thm-lower-bound}}
	\label{app-lower-bound-prop}
	\begin{proof}
		We consider the minimization problem of the form \eqref{eq:general-prob} with $f_i(x) = \frac{1}{2}\|x\|^2$ where $x \in \RR$ and with $W=\beta I_n+\frac{1}{n}(1-\beta)\one_n\one_n^T$. Note that the eigenvalues of $W$ are $\lambda_1(W)=1$ and $\lambda_k(W)=\beta$, $\forall\,2\leq k\leq n$,
		Under such setting, it holds that $f_i(x) = f_j(x)$ for any $i,j \in [n]$ and there is no heterogeneity, {\em i.e.}, $b^2 = 0$ and $\rho(W-\frac{1}{n}\mathds{1}_n\mathds{1}_n^T)=\beta$. The D-SGD algorithm in this setting will iterate as follows:
		\begin{align}
		\vx^{(k+1)} &= {W}(\vx^{(k)} - \gamma \vx^{(k)} - \gamma \vs^{(k)}) = (1-\gamma){W}\vx^{(k)} - \gamma \bar{W} \vs^{(k)}  \label{27sdga109} 
		\end{align}
		where $\vx \in \RR^{n}$ is a vector, and $\vs \in \RR^{n}$ is the gradient noise.  With \eqref{27sdga109}, we have
		\begin{equation}\label{eqn:vdgvioreqw}
		    \bar{x}^{(k+1)}=(1-\gamma)\bar{x}^{(k)}-\gamma \bar{s}^{(k)}
		\end{equation}
		and 
		\begin{align}
		\bar{\vx}^{(k+1)} = (1-\gamma) \frac{1}{n}\mathds{1}\mathds{1}^T \vx^{(k)} - \gamma \frac{1}{n}\mathds{1}\mathds{1}^T \vs^{(k)}.
		\end{align}
		Moreover, we assume each element of the noise follows standard Gaussian distribution, \textit{i.e.}, $s_i^{(k)} \sim \cN(0,\sigma^2)$, and $s_i^{(k)}$ is independent of each other for any $k$ and $i$. We also assume the gradient noise $\vs^{(k)}$ is independent of $\vx^{(\ell)}$ for any $\ell\leq  k$. With these assumptions, it holds that $\mathbb{E}[\vs^{(k)} (\vs^{(k)})^T] = \sigma^2 I \in \RR^{n\times n}$. 
		Subtracting the above recursion from \eqref{27sdga109}, we have 
		\begin{align}\label{xz6213gsb}
		\vx^{(k+1)} - \bar{\vx}^{(k+1)}  = (1-\gamma) ({W} - \frac{1}{n}\mathds{1}\mathds{1}^T)(\vx^{(k)} - \bar{\vx}^{(k)})  - \gamma ({W} - \frac{1}{n}\mathds{1}\mathds{1}^T) \vs^{(k)}.
		\end{align}
		We next define matrix $P = {W} - \frac{1}{n}\mathds{1}\mathds{1}^T$. Note that $\vs^{(k)}$ is independent of $\vx^{(k)}$. By taking the  mean-square-expectation over both sides of the above equality, we have 
		\begin{align}\label{0am}
		\mathbb{E}\|\vx^{(k+1)} - \bar{\vx}^{(k+1)}\|^2 &= \mathbb{E}\|(1-\gamma) P (\vx^{(k)} - \bar{\vx}^{(k)})\|^2 + \gamma^2 \mathbb{E}\|P\vs^{(k)}\|^2 \nonumber \\
		&\overset{\eqref{xz6213gsb}}{=} \mathbb{E}\|(1-\gamma)^2 P^2 (\vx^{(k-1)} - \bar{\vx}^{(k-1)}) - \gamma (1-\gamma)P^2 \vs^{(k-1)}\|^2 + \gamma^2 \mathbb{E}\|P\vs^{(k)}\|^2 \nonumber \\
		&= \mathbb{E}\|(1-\gamma)^2 P^2 (\vx^{(k-1)} - \bar{\vx}^{(k-1)})\|^2 + \gamma^2(1-\gamma)^2\mathbb{E}\|P^2\vs^{(k-1)}\|^2 + \gamma^2 \mathbb{E}\|P\vs^{(k)}\|^2 \nonumber \\
		&= \cdots \nonumber \\
		&= \|(1-\gamma)^{k+1}P^{k+1}(\vx^{(0)} - \bar{\vx}^{(0)})\|^2 + \gamma^2 \sum_{\ell=0}^k \mathbb{E}\|(1-\gamma)^\ell P^{\ell+1} \vs^{(k-\ell)}\|^2
		\end{align} 
		In the above derivations, we used the fact that $\vs^{(k)}$ is independent of $\vx^{(k)}$ for any $k$. Without loss of generality, we can assume $\gamma\leq  \frac{2}{L} = 2$, otherwise the iteration explodes.
		Since $x^\star = 0$,  by \eqref{eqn:vdgvioreqw}, we similarly have 
		\begin{align}\label{eqn:gvbioqwn1}
		    \EE\|\bar{x}^{(k+1)}-\bar{x}^\star\|^2&= (1-\gamma)^{2(k+1)}\|\bar{x}^{(0)}-\bar{x}^\star\|^2+{\gamma^2}\sum\limits_{\ell=0}^k\mathbb{E}\|(1-\gamma)^\ell  \bar{s}^{(k-\ell)}\|^2\geq  (1-\gamma)^{2(k+1)}\|\bar{x}^{(0)}-\bar{x}^\star\|^2.
		\end{align}
		
		Next we examine $\mathbb{E}\|(1-\gamma)^\ell P^{\ell+1} \vs^{(k-\ell)}\|^2$:
\begin{align}
		&\ \mathbb{E}\|(1-\gamma)^\ell P^{\ell+1} \vs^{(k-\ell)}\|^2 \nonumber \\
		=&\ (1-\gamma)^{2\ell}\mathbb{E}\{\tr([\vs^{(k-\ell)}]^T P^{\ell+1} P^{\ell+1} \vs^{(k-\ell)})\} \nonumber \\
		=&\ (1-\gamma)^{2\ell}\mathbb{E}\{\tr(P^{2(\ell+1)}\vs^{(k-\ell)}[\vs^{(k-\ell)}]^T)\} \nonumber \\
		=&\ (1-\gamma)^{2\ell}\tr\big(P^{2(\ell+1)}\mathbb{E}\{\vs^{(k-\ell)}[\vs^{(k-\ell)}]^T\}\big) \nonumber \\
		\overset{(a)}{=}&\ \sigma^2(1-\gamma)^{2\ell} \tr(P^{2(\ell+1)}) \nonumber \\
		\overset{(b)}{=}&\ \sigma^2(1-\gamma)^{2\ell} \tr(U \Lambda_P^{2(\ell+1)} U^T) \nonumber \\
		{=}&\ \sigma^2(1-\gamma)^{2\ell} \tr(\Lambda_P^{2(\ell+1)} U^T U) \nonumber \\
		=&\ \sigma^2(1-\gamma)^{2\ell}\tr(\Lambda_P^{2(\ell+1)}) \nonumber \\
		\overset{(c)}{=}&\ (n-1)\sigma^2(1-\gamma)^{2\ell}\beta^2 
		\label{zxnzxcncxn}
		\end{align}
		where (a) holds because $\mathbb{E}[\vs^{(k)} (\vs^{(k)})^T] = \sigma^2 I \in \RR^{n\times n}$ for any $k$, and (b) holds because $\Lambda_P = \Lambda_{{W} - \frac{1}{n}\mathds{1}\mathds{1}^T}=\{0,\beta,\dots,\beta\}$. 
With \eqref{zxnzxcncxn}, we have 
\begin{align}
		\gamma^2 \sum_{\ell=0}^k \mathbb{E}\|(1-\gamma)^\ell P^{\ell+1} \vs^{(k-\ell)}\|^2 
		\ge &(n-1)\gamma^2 \sigma^2 \beta^2  \sum_{\ell=0}^k  (1-\gamma)^{2\ell}{\beta}^{2\ell}\nonumber\\
		=&(n-1) \gamma^2 \sigma^2 \beta^2 \frac{1-(1-\gamma)^{2(k+1)}{\beta}^{2(k+1)}}{1 - (1-\gamma)^2\beta^2}.\label{z2ba00999}
		\end{align}
		Substituting \eqref{z2ba00999} into \eqref{0am}, we achieve 
\begin{align}
		\mathbb{E}\|\vx^{(k)} - \bar{\vx}^{(k)}\|^2 \ge  (n-1)\gamma^2 \sigma^2 \beta^2 \frac{1-(1-\gamma)^{2k}{\beta}^{2k}}{1 - (1-\gamma)^2\beta^2}\ge  \frac{n}{2}\gamma^2 \sigma^2 \beta^2 \frac{1-(1-\gamma)^{2k}{\beta}^{2k}}{1 - (1-\gamma)^2\beta^2}.
		\end{align}
		Since $\frac{1}{n}\mathbb{E}\|\vx^{(k)} - \vx^\star\|^2 =  \mathbb{E}\|\bar{x}^{(k)} - x^\star\|^2 + \frac{1}{n}\mathbb{E}\|\vx^{(k)} - \bar{\vx}^{(k)}\|^2$, with \eqref{eqn:gvbioqwn1} and \eqref{eqn:gvbioqwn2},
		we have 
\begin{align}\label{eqn:gvbioqwn2}
		\frac{1}{n}\mathbb{E}\|\vx^{(k)} - \vx^\star\|^2 \ge (1-\gamma)^{2k}\|\bar{x}^{(0)}-x^\star\|^2+ \frac{\sigma^2\gamma^2\beta^2}{2} \frac{1-(1-\gamma)^{2k}{\beta}^{2k}}{1 - (1-\gamma)^2\beta^2}
		\end{align}
		To guarantee D-SGD to achieve the linear speedup, we require that   $\frac{1}{n}\mathbb{E}\|\vx^{(k)} - \vx^\star\|^2 \lesssim \frac{\sigma^2}{n k}$ holds for any sufficiently large $k$ (note that P-SGD will achieve the linear speedup $\frac{\sigma^2}{n k}$ for the strongly-convex scenario). Thus, it is necessary to have that 
\begin{align}\label{bzbzbz86234}
		\underbrace{(1-\gamma)^{2k}\|\bar{x}^{(0)}-x^\star\|^2}_{\text{I}}+ \underbrace{\frac{\sigma^2\gamma^2\beta^2}{2} \frac{1-(1-\gamma)^{2k}{\beta}^{2k}}{1 - (1-\gamma)^2\beta^2}}_{\text{II}} &\le\frac{\sigma^2}{nk}
		\end{align}
		up to some absolute constants.
		In other words, there exists transient time $k_{\text{trans}}$ such that  for all $k\geq k_{\text{trans}}$, the above inequality holds up to some absolute constants. We omit the potential constant factors for simplicity since our analysis can be easily adapted to the case with some absolute constants on the two sides of \eqref{bzbzbz86234} and the rate remains the same. 
		
		Next we find $\gamma:=\gamma(k)$ such that \eqref{bzbzbz86234} holds and show that such $\gamma$ exists only when $k_{\text{trans}}=\tilde{\Omega}\left(\frac{n\beta^2}{1-\beta}\right)$.
		\begin{itemize}
		\item If $\gamma\geq \frac{1}{2}$, then $\text{II}\geq \frac{\sigma^2\beta^2}{8}$ which means \eqref{bzbzbz86234} can only holds when $k=O(1)$. Therefore, to let \eqref{bzbzbz86234} holds for all sufficiently large $k$, one must consider $\gamma<\frac{1}{2}$.
		\item If $\gamma<\frac{1}{2}$, then by inequality $\exp(\frac{x}{1+x})\leq 1+x$ for $x>-1$, we have 
		\begin{align}
		    \frac{\sigma^2}{nk}\geq (1-\gamma)^{2k}\|\bar{x}^{(0)}-x^\star\|^2\geq \exp(-2k\gamma/(1-\gamma))\|\bar{x}^{(0)}-x^\star\|^2\geq e^{-4k\gamma}\|\bar{x}^{(0)}-x^\star\|^2
		\end{align}
		where the last inequality is due to the fact $\gamma<\frac{1}{2}$. Therefore, \eqref{bzbzbz86234} implies
		\begin{equation}\label{eqn:bidosm1}
		    \gamma\geq \frac{\ln(nk\|\bar{x}^{(0)}-x^\star\|^2/\sigma^2)}{4k}.
		\end{equation}
		On the other hand, since \eqref{bzbzbz86234} implies  $(1-\gamma)^{2k}\leq \frac{\sigma^2}{nk\|\bar{x}^{(0)}-{x}^\star\|^2}$, we have
		\begin{equation}
		    \text{II}\geq \frac{\sigma^2\gamma^2\beta^2}{2}\frac{1-\frac{\sigma^2}{nk\|\bar{x}^{(0)}-{x}^\star\|^2}\beta^{2k}}{1-(1-\gamma)^2\beta^2}\ge \frac{\sigma^2\gamma^2\beta^2}{2}\frac{1-\frac{\sigma^2}{nk\|\bar{x}^{(0)}-{x}^\star\|^2}}{1-(1-\gamma)^2\beta^2}
		\end{equation}
		where we assume $k$ sufficiently large such that $\frac{\sigma^2}{nk\|\bar{x}^{(0)}-{x}^\star\|^2}\leq 1$.
		Therefore, \eqref{bzbzbz86234} also implies
		\begin{align}
		    &\frac{\sigma^2}{nk}\geq \frac{\sigma^2\gamma^2\beta^2}{2}\frac{1-\frac{\sigma^2}{nk\|\bar{x}^{(0)}-{x}^\star\|^2}}{1-(1-\gamma)^2\beta^2}\nonumber\\
		    \Longleftrightarrow\;&\left(nk- \frac{\sigma^2}{\|\bar{x}^{(0)}-x^\star\|^2}+2\right)\gamma^2-4\gamma \leq \frac{2}{\beta^2}-2.\label{eqn:bidosm2}
		\end{align}
		Note that for $\gamma>0$, $f(\gamma)\triangleq (nk-\frac{\sigma^2}{\|\bar{x}^{(0)}-x^\star\|^2}+2)\gamma^2-4\gamma$ decreases first then keeps increasing with respect to $\gamma$, so \eqref{eqn:bidosm1} and \eqref{eqn:bidosm2} are compatible only when 
		\begin{align}
		   &\left(nk- \frac{\sigma^2}{\|\bar{x}^{(0)}-x^\star\|^2}+2\right) \frac{\ln(nk\|\bar{x}^{(0)}-x^\star\|^2/\sigma^2)^2}{16k^2}-4\frac{\ln(nk\|\bar{x}^{(0)}-x^\star\|^2/\sigma^2)}{4k}\leq \frac{2}{\beta^2}-2\nonumber\\
		   \Longrightarrow\;& \frac{\left(n- \frac{\sigma^2}{k\|\bar{x}^{(0)}-x^\star\|^2}+\frac{2}{k}\right)\ln(nk\|\bar{x}^{(0)}-x^\star\|^2/\sigma^2)^2-16\ln(nk\|\bar{x}^{(0)}-x^\star\|^2/\sigma^2)}{16k}\leq\frac{2}{\beta^2}-2
		\end{align}
		which leads to $k\geq \tilde{\Omega}\left(\frac{n\beta^2}{1-\beta^2}\right)=\tilde{\Omega}\left(\frac{n\beta^2}{1-\beta}\right)$. Therefore, we reach the conclusion that $k_{\text{trans}}=\tilde{\Omega}\left(\frac{n\beta^2}{1-\beta}\right)$.
		\end{itemize}

	\end{proof}
	
	{
		\section{Convergence of Algorithm \ref{Algorithm: ED_multiple_gossip}}
		\label{app-algo-2}
		In this section we will establish the convergence of D$^2$/Exact-Diffusion with multiple gossip steps. As we have discussed in Sec.~\ref{sec-reformulation-D2-mg}, there are two fundamental differences between the vanilla D$^2$/Exact-Diffusion and its variant with multiple gossip steps: 
		\begin{itemize}
			\item \textbf{Gradient accumulation.} For each outer loop $k$, each node $i$ in D$^2$/Exact-Diffusion with multiple gossip steps will compute the stochastic gradient with $g_i^{(k)} = \frac{1}{R}\sum_{r=1}^R \nabla F(x_i^{(k)};\xi_i^{(k,r)})$ with $R$ independent data samples $\{\xi_i^{(k,r)}\}_{r=1}^R$. This will result in a reduced gradient noise:
			\begin{align}
			\mathbb{E}[\|g_i^{(k)} - \nabla f_i(x_i^{(k)})\|^2 | \cF^{(k-1)}] \le \frac{\sigma^2}{R}.
			\end{align}
			
			\item \textbf{Fast gossip averaging.} With fast gossip averaging (i.e., Algorithm \ref{Algorithm: fast-gossip}), the weight matrix utilized in D$^2$/Exact-Diffusion with multiple gossip steps is $\bar{M}$ instead of $\bar{W}$ (see recursions \eqref{eq:d2-multiple-gossip} and \eqref{eq:d2-pd-mg}). In addition, it is established in Proposition \ref{prop-lambda-mbar} that $\lambda_k(\bar{M}) \in [\frac{1}{4n}, \frac{3}{4n}]$ for $2\le k \le n$ if $R = \lceil \frac{\ln(n)+4}{\sqrt{1-\beta}} \rceil$. 
		\end{itemize}
		
		\noindent We will utilize these facts to facilitate the analysis for D$^2$/Exact-Diffusion with multiple gossip steps.
		
		\subsection{Proof of Proposition \ref{prop-M-property}}\label{app:M-propert}
		This proposition directly follows the results of \cite{liu2011accelerated}. We provide the proof for completeness. Note that \eqref{eqn:M2} can be transformed into a first-order iteration as follows:
		\[
		\begin{bmatrix}
		M^{(r+1)}\\
		M^{(r)}
		\end{bmatrix}=\underbrace{\begin{bmatrix}
		(1+\eta)W & -\eta I\\
		I & 0
		\end{bmatrix}}_{W_2}\begin{bmatrix}
		M^{(r)}\\
		M^{(r-1)}
		\end{bmatrix}.
		\] 
		By \cite[Proposition 3]{liu2011accelerated}, the projection of augmented matrix $W_2$  on the subspace orthogonal to $\mathds{1}_n$ is a contraction with  spectral norm of $\frac{\beta}{1+\sqrt{1-\beta^2}}$. Since $\frac{\beta}{1+\sqrt{1-\beta^2}}\leq 1-\sqrt{1-\beta}$ for any $0\leq \beta\leq 1$, we have
		\[
		\left\|\begin{bmatrix}
		M^{(r)}\\
		M^{(r-1)}
		\end{bmatrix} \vz\right\|\leq \Big(1-\sqrt{1-\beta}\Big)^{r}\left\|\begin{bmatrix}
		M^{(0)}\\
		M^{(-1)}
		\end{bmatrix} \vz\right\|=\sqrt{2}\Big(1-\sqrt{1-\beta}\Big)^r\|\vz\|
		\]
		for any $\vz \perp \mathds{1}_n$. We thus have, for any $\vz \perp \mathds{1}_n$, that
		\[\|	M^{(r)}\vz\|\leq \sqrt{2}\Big(1-\sqrt{1-\beta}\Big)^r\|\vz\|\quad \text{i.e.,} \quad \rho(M^{(r)}-\frac{1}{n}\one_n\one_n^T)\leq \sqrt{2}\Big(1-\sqrt{1-\beta}\Big)^r.
		\]
		
		\subsection{Proof of Proposition \ref{prop-lambda-mbar}}
		\label{app-proof-mbar}
		If we choose $R=\lceil \frac{\ln(n)+4}{\sqrt{1-\beta}}\rceil $, $\tau = \frac{1}{2n}$ and denote $M \triangleq  M^{(R)}$ and $\bar{M}\triangleq (1-\tau)M+\tau I$. It follows from Proposition \ref{prop-M-property} that 
		\[
		\sqrt{2}\Big(1-\sqrt{1-\beta}\Big)^R=\sqrt{2}\exp(R\ln(1-\sqrt{1-\beta})) \overset{(a)}{\leq} \sqrt{2}\exp(-R\sqrt{1-\beta})\leq \frac{1}{4n},
		\]
		where (a) holds because of the inequality $\ln(1-x) \le -x$ for any $x\in (0,1)$. The above inequality implies \[
		\max\{|\lambda_2(M)|, |\lambda_n(M)|\}= \rho(M-\frac{1}{n}\one\one^T)\leq \frac{1}{4n}.
		\] 
		Since $\bar{M}\triangleq (1-\tau)M+\tau I$ and $\tau = \frac{1}{2n}$, we have $\lambda_k(\bar{M})=(1-\frac{1}{2n})\lambda_k(M)+\frac{1}{2n}$, the spectrum estimates of $\bar{M}$ is given by \eqref{znzbab09823}. 
		
		\subsection{Proof of Theorem \ref{thm-generally-convex-mg}}
		\label{app-proof-mg-convex}
		The gradient accumulation and the fast gossip  averaging do not affect the convergence analysis of D$^2$/Exact-Diffusion. By following the analysis of Theorem \ref{thm-generally-convex}, if the learning rate is set as 
		\begin{align}\label{znnbzzbbzz0000000}
		\gamma = \min\left\{\frac{1}{4L},\frac{(1-\tilde{\beta}) \tilde{\lambda}_n^{\frac{1}{2}}}{40 L \tilde{\lambda}_2}, \Big(\frac{\tilde{r}_0}{\tilde{r}_1(K+1)}\Big)^{\frac{1}{2}}, \Big(\frac{\tilde{r}_0}{\tilde{r}_2(K+1)}\Big)^{\frac{1}{3}},\left(\frac{\tilde{r}_0}{\tilde{r}_3}\right)^\frac{1}{3}\right\},
		\end{align}
		it holds from \eqref{xbsd87-0} that 
		\begin{align}\label{xbsd87-0-mg}
		\frac{1}{K+1}\sum_{k=0}^K\left(\EE f(\bar{x}^{(k)})-f(x^\star)\right)\le O\left(\frac{\tilde{\sigma}}{\sqrt{nK}}\hspace{-0.3mm}+\hspace{-0.3mm}\frac{\tilde{\sigma}^\frac{2}{3}\tilde{\lambda}_2^\frac{2}{3}}{(1\hspace{-0.3mm}-\hspace{-0.3mm}\tilde{\beta}_1)K^\frac{2}{3}\tilde{\lambda}_n^\frac{1}{3}}\hspace{-0.3mm}+\hspace{-0.3mm}\frac{\tilde{\lambda}_2}{(1\hspace{-0.3mm}-\hspace{-0.3mm}\tilde{\beta}_1)K\tilde{\lambda}_n^\frac{1}{2}}\hspace{-0.3mm}+\hspace{-0.3mm}\frac{\tilde{\lambda}_2^\frac{2}{3}}{(1\hspace{-0.3mm}-\hspace{-0.3mm}\tilde{\beta}_1)^\frac{1}{3}(1\hspace{-0.3mm}-\hspace{-0.3mm}\tilde{\beta})^\frac{1}{3}K\tilde{\lambda}_n^\frac{1}{3}}\hspace{-0.3mm}+\hspace{-0.3mm}\frac{1}{K}\right)
		\end{align}
		where $K$ is the number of outer loops, and by the definition of $M$, we have
		\begin{align}\label{zbn24390a65}
		&\tilde{\sigma}^2 = \sigma^2/R, \quad  \tilde{\lambda}_2 = \lambda_2(\bar{M}) \in [\frac{1}{4n}, \frac{3}{4n}], \quad \tilde{\lambda}_n = \lambda_n(\bar{M}) \in [\frac{1}{4n}, \frac{3}{4n}], \nonumber \\
		&\tilde{\beta}_1 = \sqrt{\lambda_2(\bar{M})}\in [\frac{1}{2n^\frac{1}{2}},\frac{\sqrt{3}}{2n^\frac{1}{2}}], \quad \tilde{\beta} =\max\{|\lambda_2(M)|,|\lambda_n(M)|\} \leq  \frac{3}{4n}
		\end{align}
		In addition, constants $\tilde{r}_0$, $\tilde{r}_1$, $\tilde{r}_2$, and $\tilde{r}_3$ in \eqref{znnbzzbbzz0000000} are defined as follows
		\begin{align}\label{kzunb}
		\tilde{r}_0 = 2 \mathbb{E}\|\bar{\z}^{(0)}\|^2, \ \tilde{r}_1 = \frac{2\tilde{\sigma}^2}{n}, \ \tilde{r}_2 = \frac{24L\bar{\lambda}_2^2\tilde{\sigma}^2}{(1-\tilde{\beta}_1) \tilde{\lambda}_n}, \ \tilde{r}_3 = \frac{9L\tilde{\lambda}_2^2}{(1-\tilde{\beta}_1)(1-\tilde{\beta})\tilde{\lambda}_n}. 
		\end{align}
		We let $T$ be the total number of sampled data or gossip communications, it holds that $K = T/R$. Substituting $K = T/R$ and the facts in \eqref{zbn24390a65}  into \eqref{xbsd87-0-mg}, and ignoring all constants, we achieve 
		\begin{align}
		\frac{1}{K+1}\sum_{k=0}^K \left(\mathbb{E} f(\bar{x}^{(k)}) \hspace{-0.5mm}-\hspace{-0.5mm} f(x^\star)\right) =& O\Big( \frac{{\sigma}}{\sqrt{nT}} +
		\frac{R^{\frac{1}{3}}{\sigma}^{\frac{2}{3}}}{n^{\frac{1}{3}}T^{\frac{2}{3}}}  +  \frac{R}{n^{\frac{1}{2}}T}+\frac{R}{n^{\frac{1}{3}}T}+\frac{R}{T}\Big)\nonumber \\
		=& {O}\Big( \frac{{\sigma}}{\sqrt{nT}} +
		\frac{\ln(n)^\frac{1}{3}{\sigma}^{\frac{2}{3}}}{n^{\frac{1}{3}}T^{\frac{2}{3}}(1-\beta)^{\frac{1}{6}}}  +  \frac{\ln(n)}{T(1-\beta)^{\frac{1}{2}}}\Big)\label{xbsd87-0-mg-2}
		\end{align} 
		where the last inequality holds by substituting $R = \lceil \frac{\ln(n)+4}{\sqrt{1-\beta}} \rceil={O}\left(\frac{\ln(n)}{(1-\beta)^\frac{1}{2}}\right)$. Note that the third term is less than or equal to the last term, we achieve the result in \eqref{xbsd87-mg}.

		
		\subsection{Proof of Theorem \ref{thm2-mg}}
		\label{app-proof-thm2-mg}
		By following the analysis of Theorem \ref{thm2}, if the learning rate is set as  ($T=KR$)
		\begin{align}\label{gamma-sc-mg-app}
		\gamma = \min\left\{\frac{1}{4L},   \frac{1-\tilde{\beta}_1}{26L} \Big(\frac{\tilde{\lambda}_n^{1/2}}{\tilde{\lambda}_2}\Big),\frac{2\ln(2 n \mu  \mathbb{E}\|\bar{\z}^{(0)}\|^2 K^2/[\tilde{\sigma}^2 (1-\tilde{\beta})])}{\mu K}\right\}
		\end{align}
		it holds from \eqref{zn238sd6zzz1-0} that 
		\begin{align}\label{results-mg-app}
\frac{1}{H_K}\sum_{k=0}^K h_k \big(\mathbb{E} f(\bar{x}^{(k)}) - f(x^\star) \big)=\tilde{O}\left(\frac{\tilde{\sigma}^2}{n K} +  \frac{\tilde{\sigma}^2}{(1-\tilde{\beta}) K^2} \big(\frac{\tilde{\lambda}^2_2}{\tilde{\lambda}_n}\big)  +
		\frac{\tilde{\lambda}_2}{(1-\tilde{\beta})\tilde{\lambda}^{\frac{1}{2}}_n}  \exp\{-(1-\tilde{\beta}) \big(\frac{\tilde{\lambda}^{\frac{1}{2}}_n}{\tilde{\lambda}_2}\big) K \}+\exp(-K)  \right).
		\end{align}
		where $K$ is the number of outer loop, and $h_k$ and $H_K$ are defined in Lemma \ref{lm-sc-ergodic-consenus}. Notation $\tilde{O}(\cdot)$ hides all logarithm factors.  Substituting $K = T/R$,  and the facts in \eqref{zbn24390a65} into \eqref{results-mg-app}, we have 
		\begin{align}\label{results-mg-app-2}
\frac{1}{H_K}\sum_{k=0}^K h_k\big(\mathbb{E} f(\bar{x}^{(k)}) - f(x^\star) \big)
		=& \tilde{O}\left(\frac{{\sigma}^2}{n T} + \frac{R {\sigma}^2}{nT^2}+n^{-\frac{1}{2}}
		\exp\{- n^{\frac{1}{2}}\frac{T}{R} \}  + \exp(-\frac{T}{R})\right)\nonumber \\
		=& \tilde{O}\left(\frac{{\sigma}^2}{n T} + \frac{{\sigma}^2}{n(1-\beta)^{\frac{1}{2}}T^2} + \exp\{- (1-\beta)^\frac{1}{2}T\}\right).
		\end{align}
	}
	

\end{document}